\documentclass[11pt,a4paper]{amsart}
\usepackage{amsmath,amsthm,amssymb,a4wide}
\usepackage[utf8]{inputenc}
\usepackage{graphicx}
\usepackage{hyperref}

\newcommand{\eps}{\varepsilon}
\newcommand{\Lm}{\mathcal{L}}
\newcommand{\Hm}{\mathcal{H}}
\newcommand{\X}{\mathcal{X}}
\newcommand{\RR}{\mathbb{R}}
\newcommand{\NN}{\mathbb{N}}
\def\d{\,\mathrm{d}}

\newtheorem{theorem}{Theorem}[section]
\newtheorem{lemma}[theorem]{Lemma}
\newtheorem{example}[theorem]{Example}
\newtheorem{proposition}[theorem]{Proposition}
\newtheorem{corollary}[theorem]{Corollary}
\theoremstyle{remark}
\newtheorem{remark}[theorem]{Remark}

\begin{document}

\title[Diffuse Domain Method for Elliptic Equations]{Analysis of the Diffuse Domain Method for second order elliptic boundary value problems}
\author[M. Burger]{Martin Burger$^{\dag,+}$}
\author[O. L. Elvetun]{Ole Løseth Elvetun$^*$}
\author[M. Schlottbom]{Matthias Schlottbom$^{\dag,\circ}$}
\thanks{$^\dag$ Institute for Computational and Applied Mathematics,
University of Münster, Einsteinstr. 62, 48149 M\"unster, Germany.\\
$^+$ Cells in Motion Cluster of Excellence, University of Münster. \\
$^*$ Dept. of Mathematical Sciences and Technology, Norwegian University of Life Sciences\\
$^\circ$ Corresponding author\\
Email: {\tt $\{$burger,schlottbom$\}$@uni-muenster.de, ole.elvetun@nmbu.no}
}
\date{\today}

\begin{abstract}
The diffuse domain method for partial differential equations on complicated geometries recently received strong attention in particular from practitioners, but many fundamental issues in the analysis are still widely open. In this paper we study the diffuse domain method for approximating second order elliptic boundary value problems posed on bounded domains, and show convergence and  rates of the approximations generated by the diffuse domain method to the solution of the original second order problem when complemented by Robin, Dirichlet or Neumann conditions. 

The main idea of the diffuse domain method is to relax these boundary conditions by introducing a family of phase-field functions such that the variational integrals of the original problem are replaced by a weighted average of integrals of perturbed domains. From an functional analytic point of view, the phase-field functions naturally lead to weighted Sobolev spaces for which we present trace and embedding results as well as various type of Poincar\'e inequalities with constants independent of the domain perturbations. Our convergence analysis is carried out in such spaces as well, but allows to draw conclusions also about unweighted norms applied to restrictions on the original domain. 
Our convergence results are supported by numerical examples.
\end{abstract}

\maketitle



%
%
%
%
%
%

\section{Introduction}
This paper considers the approximation properties of the diffuse domain method (also called diffuse interface method, cf. \cite{LiLowengrubRaetzVoigt2009,ratz2006pde}) when applied to linear second order elliptic equations of the form
\begin{align}\label{eq:elliptic}
 -{\rm div}(A\nabla u) + cu = f\quad\text{in } D
\end{align}
complemented by suitable boundary conditions on a sufficiently smooth domain $D\subset\RR^n$.
We focus on Neumann, Robin and Dirichlet boundary conditions, i.e. either
\begin{align}
 n\cdot A \nabla u + b u &= g \quad\text{ on } \partial D \qquad\text{or }\label{eq:bc1}\\
 u &= g \quad\text{ on } \partial D\label{eq:bc2}.
\end{align}
For solving equations of the above type we will employ variational methods.
Roughly speaking, for instance \eqref{eq:elliptic}--\eqref{eq:bc1} is reformulated as follows: Find $u$ such that for all suitable test functions $v$ 

\begin{align}\label{eq:var1}
 \int_D A\nabla u \cdot \nabla v + c u v \d x +\int_{\partial D} b uv \d\sigma
 = \int_D fv \d x + \int_{\partial D} g v \d\sigma.
\end{align}
%
In many applications the domain $D$, the boundary $\partial D$ (or equally well some interface inside the domain) is not known exactly or its geometry is complicated making a proper approximation of the integrals difficult or expensive \cite{AlandLowengrubVoigt2010,EsedogluRaetzRoeger2014,Greer2006,OttoPenzlerRaetzRumpVoigt2004,Raetz2014,TeigenSongLowengrubVoigt2011}.
In addition to the methods used in the aforementioned references, let us point to further literature dealing with methods to handle these type of difficulties;
for instance
the immersed boundary method \cite{Peskin1977}, 
the immersed interface method \cite{LeVequeLin1994}, 
the fictitious domain method \cite{GlowinskiPanPeriaux1994}, 
the unfitted finite element method \cite{BarrettElliot1987}, 
the finite cell method \cite{ParvizianDuesterRank2007},
unfitted discontinuous Galerkin methods \cite{BastianEngwer2009},
composite finite elements \cite{HackbuschSauter1997,LiehrPreusserRumpfSauterSchwen2009}; let us also refer to these papers for further links to literature and applications. 
In this work we will focus on the diffuse domain method, see for instance \cite{LiLowengrubRaetzVoigt2009}.

The diffuse domain method relies on the fact that the domain $D$ can be described by its oriented distance function $d_D(x) =  {\rm dist}(x, D) - {\rm dist}(x,\RR^n\setminus D)$, $x\in\RR^n$.
As one can easily see, we have $D=\{d_D<0\}$. In order to relax the sharp interface condition $d_D<0$, let us introduce 
	$\varphi^\eps = S(-d_D/\eps)$ 
for $\eps > 0$ small and $S$ being a sigmoidal function, i.e. non-decreasing with 
	$S(t) =  t/|t|$ for $|t|\geq 1$.
As $\eps$ tends to zero, $S(\cdot/\eps)$ converges to the sign function, and hence the phase-field function $\omega^\eps=(1+\varphi^\eps)/2$ formally converges to the indicator function $\chi_D$ of $D$.
%
The key idea to approximate the integrals in \eqref{eq:var1} is a weighted averaging of the integrals over $\{ d_D < t \}$ instead of integrating over the original domain 
$\{ d_D < 0 \}$ only (and similar for boundary integrals). Since 
$\frac{1}{2\eps} S'(\frac{\cdot}\eps)$ approximates a concentrated distribution at zero, we expect
\begin{align*}
\int_D h(x) \d x &= \int_{\{ d_D < 0 \}} h(x) \d x \\
&= \int_{-\infty}^\infty \frac{1}{2\eps} S'(-\frac{t}\eps) \int_{\{ d_D < 0 \}} h(x) \d x\d t \\
&\approx \int_{-\infty}^\infty \frac{1}{2\eps} S'(-\frac{t}\eps) \int_{\{ d_D < t \}} h(x)\d x \d t \\
&= \frac{1}2 \int_{-1}^{1} \int_{\{\varphi^\eps > s\}} h(x) \d x \d s, 
\end{align*}
where we have used the substitution $s=S(-\frac{t}\eps)$ in the last step. Now the layer cake-representation can further be used for given integrable $h$  to rewrite 
$$ \int_{-1}^{1} \int_{\{\varphi^\eps > s\}} h(x)\d x\d t= \int_\Omega \int_{-1}^{\varphi^\eps(x)}  \d s\, h(x)\d x =\int_\Omega (1+\varphi^\eps(x)) h(x) \d x. $$

By an analogous computation we obtain for the boundary integral
$$ \int_{\partial D} h(x) \d \sigma(x) \approx  \frac{1}{2} \int_{-1}^{1} \int_{\partial\{\varphi^\eps > s\}} h(x) \d\sigma(x) \d s, $$
which can be simplified via the co-area formula to 
$$  \int_{-1}^1 \int_{\partial \{\varphi^\eps > s\}} h(x) \d\sigma(x) \d t =  \int_\Omega h(x) |\nabla \varphi^\eps(x)| \d x.$$
Here, $\Omega\supset \overline{D}$ is a domain with ``simple'' geometry, i.e. a geometry which can be easily approximated.
Based on this motivation we shall define the following diffuse volume and surface integrals
%
\begin{align}\label{eq:diffuse_integrals}
 \int_D h(x) \d x \approx  \int_\Omega h(x) \omega^\eps(x) \d x\quad \text{and}\quad \int_{\partial D } h(x) \d\sigma(x)\approx \int_\Omega h(x) |\nabla \omega^\eps(x)|\d x.
\end{align}
Using this approximation in \eqref{eq:var1} leads us to the following variational problem:
 Find $u^\eps$ such that for all suitable test functions $v$ 
\begin{align}\label{eq:ddm1}
 \int_\Omega \left(A\nabla u^\eps \cdot \nabla v + c u^\eps v \right)\omega^\eps \d x +\int_{\Omega} b u^\eps v  |\nabla\omega^\eps| \d x
 = \int_\Omega fv \omega^\eps \d x + \int_{\Omega} g v|\nabla \omega^\eps| \d x.
\end{align}
Under the usual assumptions on $A$, $b$ and $c$, the bilinear form on the left-hand side of \eqref{eq:ddm1} is well-defined on the weighted Sobolev space $W^{1,2}(\Omega;\omega^\eps)$, which is the closure of smooth functions $u:\Omega\to\RR$ with finite (semi-) norm
\begin{align*}
 \|u\|_{W^{1,2}(\Omega;\omega^\eps)}^2 = \int_\Omega \left( |\nabla u|^2 + |u|^2\right) \omega^\eps \d x.
\end{align*}
The main point of the present manuscript is to estimate the error between the solution $u$ of \eqref{eq:var1} and the solution $u^\eps$ of \eqref{eq:ddm1}. Our key results are the following two theorems, the first treating the low regularity case and the second giving optimal rates under full regularity:

\begin{theorem}\label{thm:main1}
 Let $\partial D$ be of class $C^{1,1}$ and let $0\leq c\in L^\infty(\Omega)$, $0\leq b\in W^{1,\infty}(\Omega)$ and $A\in L^\infty(\Omega)$ such that  $\kappa\leq A(x) \leq \kappa^{-1}$ for all $x\in\Omega$ and some constant $\kappa>0$. Moreover, assume that $f\in L^2(\Omega;\omega^\eps)$ and $g\in W^{1,2}(\Omega;\omega^\eps)$. Furthermore, let $u\in W^{1,2}(D)$ be a solution to \eqref{eq:var1} and $u^\eps\in W^{1,2}(\Omega;\omega^\eps)$ be a solution to \eqref{eq:ddm1}. Then there exists $p>2$ and a constant $C>0$ independent of $\eps$ such that
 \begin{align*}
  \|u-u^\eps\|_{W^{1,2}(\Omega;\omega^\eps)} \leq C \eps^{\frac{1}{2}-\frac{1}{p}}.
 \end{align*}
\end{theorem}
\begin{theorem}\label{thm:main2}
 In addition to the assumptions of Theorem~\ref{thm:main1} let $\partial D$ be of class $C^{\infty}$, and let $f,g,A,b,c\in C^\infty(\overline{\Omega})$.
 Then there exists a constant $C$ independent of $\eps$ such that
 \begin{align*}
  \|u-u^\eps\|_{W^{1,2}(\Omega;\omega^\eps)} \leq C\eps^{3/2}.
 \end{align*}
\end{theorem}
Let us mention that the case $b=0$ is allowed here and corresponds to Neumann boundary conditions.
The index $p$ in Theorem~\ref{thm:main1} is related to $L^p$ regularity of $\nabla u$, see \cite{Groeger1989,Meyers1963} and Section~\ref{sec:robin} below.
In order to prove the theorems, we need a few technical ingredients.
As is obvious from the above discussion, we have to deal with a family of weighted Sobolev spaces parametrized by $\eps$.
For certain choices of $S$, 
we observe that the weight $\omega^\eps$ is proportional to a power of  a distance function near $\partial D$, see \eqref{eq:weight_near_boundary} below. 
For this type of weights and fixed $\eps$ many results have been established in literature; see for instance \cite{Horiuchi1989,Kufner1985,Necas2012,OpicKufner1990} and more generally \cite{TriebelFunctionSpaces2006}. 
Since we are dealing with a family of spaces corresponding to a family of weights $\omega^\eps$, we are particularly interested in the behavior of the weighted spaces when $\eps$ changes. 
We will present trace theorems, embedding theorems and Poincar\'e inequalities with constants independent of $\eps$, which turn out to be indispensable tools for the analysis of \eqref{eq:ddm1}, and which we think might be of interest in their own, see Section~\ref{sec:properties_weighted_sobolev}. 
In order to prove these statements, we have to revise and adapt the classical proofs of \cite{OpicKufner1990} and combine them with arguments recently used in the context of shape optimization; see \cite{Arrieta08,BoulkhemairChakib2007,DelfourZolesio2011} for such arguments applied to unweighted Sobolev spaces.

A further necessary ingredient to obtain the error estimates of Theorem~\ref{thm:main1} and Theorem~\ref{thm:main2} 
are rigorous error estimates for the approximations \eqref{eq:diffuse_integrals} in terms of $\eps$ and the regularity of the integrands.
A consequence of our results is that for $h \in L^p(\Omega;\omega^\eps)$ 
\begin{align*}
 \int_D h \d x - \int_\Omega h \omega^\eps \d x = O(\eps^{1-1/p}),\qquad\text{as }\eps\to 0,
\end{align*}
see Theorem~\ref{thm:lp_diffuse_volume}. 
Assuming $h$ and $\nabla h$ in $L^p(\Omega;\omega^\eps)$, we can exploit the special averaging procedure in the derivation of the diffuse integrals in a crucial way to obtain the improved estimate 
\begin{align*}
 \int_D h \d x - \int_\Omega h \omega^\eps \d x = O(\eps^{2-1/p}), \qquad\text{as }\eps\to 0,
\end{align*}
cf. Theorem~\ref{thm:sobolev_diffuse_volume}.
Concerning Robin boundary values, let us mention recent formal result obtained by asymptotic analysis stating an $L^2$-convergence rate for the error $u-u^\eps$ of $O(\eps^2)$ \cite{LervagLowengrub2014}. Using a problem adapted norm we also obtain a rate $O(\eps^2)$. For two-dimensional problems and under reasonable assumptions on this problem adapted norm, which we however cannot verify for our problem, we also arrive at a $L^2$-convergence rate $O(\eps^2)$. The latter is well confirmed by numerical results.

Concerning Dirichlet boundary values, the corresponding convergence rate only yields the half order compared to the Robin boundary values.
Hence in the setting of Theorem~\ref{thm:main1}, but with \eqref{eq:bc2} instead of \eqref{eq:bc1}, we obtain 
\begin{align*}
  \|u-u^\eps\|_{W^{1,2}(D)} = O(\eps^{\frac{1}{4}-\frac{1}{2p}}), \qquad\text{as }\eps\to 0,
\end{align*}
and accordingly, in the setting of Theorem~\ref{thm:main2} with \eqref{eq:bc2} in place of \eqref{eq:bc1}, we obtain
\begin{align*}
  \|u-u^\eps\|_{W^{1,2}(D)} = O(\eps^{\frac{3}{4}}),\qquad\text{as }\eps\to 0.
\end{align*}
In the best case, using the problem adapted norm, we can show a rate $O(\eps)$. 
This complies with recent results in literature, which were obtained for one-dimensional problems or numerically \cite{FranzGaertnerRoosVoigt2012,ReuterHillHarrison2012}.

The outline of the manuscript is as follows. In Section~\ref{sec:prelim} we discuss the geometry of $D$ and certain perturbations of it.
Section~\ref{sec:weighted_sobolev} introduces weighted Lebesgue and Sobolev spaces together with some basic properties. A more detailed study of weighted Sobolev spaces is presented in Section~\ref{sec:properties_weighted_sobolev}. Here we present a trace and an embedding theorem for weighted Sobolev spaces and we show that the corresponding estimates are stable with respect to $\eps$. Moreover, we prove Poincar\'e and Poincar\'e-Friedrichs inequalities for these spaces again with constants independent of $\eps$.
Approximation properties of the diffuse integrals are presented in Section~\ref{sec:conv_diffuse}. The volume integrals are investigated in Section~\ref{sec:conv_diffuse_volume}, and corresponding results for the diffuse boundary integral are subsequently shown in Section~\ref{sec:conv_diffuse_boundary}.
Section~\ref{sec:elliptic_problems} deals with the approximation of elliptic equations by the diffuse domain method for Robin, Dirichlet and Neumann type boundary values, and proofs of Theorem~\ref{thm:main1} and Theorem~\ref{thm:main2} are given.
Our results are supported by numerical results which are presented in Section~\ref{sec:numerics}.
We conclude in Section~\ref{sec:conclusion} and discuss briefly some open questions.

%
%
%
%
%
%
%

\section{Some geometric preliminaries}\label{sec:prelim}

\subsection{Domain}
Throughout the manuscript we assume that $D\subset \RR^n$ is a domain with $C^{1,1}$ boundary.
Associated to $D$ we define its oriented distance function $d_D$ by 
$$d_D(x)={\rm dist}(x,D)-{\rm dist}(x,\RR^n\setminus D) \quad\text{ for  } x\in\RR^n.$$ 
Since $\partial D$ is of class $C^{1,1}$ we have that $d_D$ is $C^{1,1}$ in a neighborhood of $\partial D$ \cite{DelfourZolesio2011}.
For $t\in\RR$ let us define the sublevel sets of $d_D$ as follows
\begin{align*}
 D_t = \{x\in \RR^n:\ d_D(x)<t\}.
\end{align*}
We clearly have the inclusions $D_{t_1}\subset D_0=D \subset D_{t_2}$ for all $t_1<0<t_2$.
Moreover, we fix a domain $\Omega\subset\RR^n$ such that $\overline{D}\subset \Omega$. 
In applications $\Omega$ is a domain with a ``simple'' geometry, for instance a ball or a bounding box.

\subsection{The tubular neighborhood \texorpdfstring{$\Gamma_\eps$}{Ge}}\label{sec:tubular}
Let us define the $\eps$-tubular neighborhood of $\partial D$ by
\begin{align}\label{eq:def_gamma_eps}
 \Gamma_\eps = D_\eps\setminus \overline{D_{-\eps}} .
\end{align}
Due to $C^{1,1}$ regularity of $\partial D$, the projection of $z\in \Gamma_\eps$ onto $\partial D$ is unique for $\eps$ sufficiently small, i.e., for each $z\in\Gamma_\eps$ there exists a unique $x\in \partial D$ such that $z=x+d_D(z)n(x)$ \cite[Chapter~7, Theorem~3.1, Theorem~8.4]{DelfourZolesio2011}.
Here, $n(x)$, $ x\in\partial D$, denotes the outward unit normal field which is related to the oriented distance function via the formula $n(x)=\nabla d_D(x)$.
This shows 
$$\Gamma_\eps = \{z\in \Omega: \exists x\in\partial D, |t|<\eps, z= x + t n(x)\}.$$
In the whole manuscript we fix $\eps_0$ so small such that the just described projection $\Gamma_{2\eps_0} \to \partial D$ is single-valued.
Thus, for each $\eps\leq \eps_0$, $D_\eps = \Gamma_\eps\cup D$, and for every $x\in\Gamma_\eps$, there holds ${\rm dist}(x,\partial D)\leq \eps$ and for every $x\in \Omega\setminus \Gamma_\eps$, there holds ${\rm dist}(x,\partial D)\geq \eps$.
Moreover, for some constant $C>0$ independent of $\eps$
\begin{align}\label{eq:boundary_measure}
 |\Gamma_\eps| \leq C \Hm^{n-1}(\partial D) \eps.
\end{align}
Here, $|\Gamma_\eps|=\Lm^{n}(\Gamma_\eps)$ is the $n$-dimensional Lebesgue-measure of $\Gamma_\eps$ and $\Hm^{n-1}(\partial D)$ is the $n-1$-dimensional Hausdorff-measure of $\partial D$.

\subsection{Transformations of the geometry}\label{sec:boundary_transformation}

Let $t\in(-\eps_0,\eps_0)$. 
Let us first consider transformations of boundaries $\partial D \to \partial D_t$.
To do so, we introduce a family of transformations $\Phi_t:\partial D\to \partial D_t$ defined by $\Phi_t(x)=x+tn(x)$. 
The Jacobian satisfies $D\Phi_t(x) = I + t D^2 d_D(x)$, with $I$ being the identity matrix on $\RR^{n\times n}$ and $D^2d_D$ being the Hessian of $d_D$, and thus,
\begin{align}\label{eq:representation_Jacobian}
 |\det D\Phi_t(x)- (1 + t \Delta d_D(x))| \leq C \|D^2 d_D\|_{L^\infty(\partial D)} t^2 \quad\text{for } |t|\leq \eps\leq \eps_0.
\end{align}
Decreasing $\eps_0$ if necessary, there holds $\frac{1}{2}\leq \det D\Phi_t(x)\leq 2$
and
\begin{align}\label{eq:det_estimate}
 |\det D\Phi_t(x)- \det D\Phi_{-t}(x)|\leq  C |t|\|D^2 d_D\|_{L^\infty(\partial D)}.
\end{align}
Denoting by $n_t$ the unit outer normal vector field of $\partial D_t$, we see that $n_t(\Phi_t(x))=n(x)$ for all $x\in \partial D$ by the choice of the tubular neighborhood $\Gamma_{\eps_0}$. Thus, here and in the following, we will just write $n(x)$ for the unit outer normal at some $x\in\partial D_t$.
In particular, $\Phi_t$ can be extended to the whole of $\Gamma_{\eps_0}$, and for this extension we have that
\begin{align*}
 \Phi_t(\Phi_{s}(x)) = \Phi_{s}(x) +t n(\Phi_{s}(x)) = x+(s+t)n(x) = \Phi_{s+t}(x),\quad s,t\in (-\eps_0,\eps_0),
\end{align*}
particularly, $\Phi_t(\Phi_{-t}(x))=x$.
For $h\in L^1(\Gamma_\eps)$, $-\eps<a<b<\eps$ and $s\in (a,b)$ we then have
\begin{align}\label{eq:boundary_transformation_formula}
 \int_{\{a<d_D<b\}} h(x) \d x
 =\int_{a}^b \int_{\partial D_s} h(x+(t-s)n(x)) |\det D\Phi_{t-s}(x)|\d\sigma_s(x) \d t.
\end{align}
Note that $x-sn(x)\in \partial D$ for $x\in\partial D_s$, and hence, $x+(t-s)n(x)\in \partial D_t$ for $x\in \partial D_s$.
Here, $\sigma_s = \Hm^{n-1} \llcorner \partial D_s$ is the surface element of $\partial D_s$, i.e.
\begin{align*}
 \sigma_s(\widetilde \Omega) = \Hm^{n-1}(\widetilde\Omega\cap \partial D_s) \quad\text{for } \widetilde\Omega\subset\RR^n.
\end{align*}

For the volume transformation $D\to D_t$, 
we define $\psi_t:[0,\infty)\to\RR$ by $\psi_t(s)=0$ for $s\geq \eps_0$, and $\psi_t(s)=\frac{t}{\eps_0^2} s^2 - \frac{2t}{\eps_0}s+t$ for $s<\eps_0$.
Then $\psi_t\in C^1([0,\infty))$, and $\|\psi_t\|_{C^1([0,\infty))}\to 0$ as $t\to 0$.
Moreover, $\psi_t$ maps $[0,\eps_0]$ one-to-one onto $[0,t]$.
We then define the diffeomorphism
\begin{align*}
\Psi_t: D\to D_t,\quad
 \Psi_t(x) = \begin{cases} x + \psi_t(|d_D(x)|) n(x), &x \in  D\cap\Gamma_{\eps_0},\\
           x, & x\in D\setminus \Gamma_{\eps_0}.
          \end{cases}
\end{align*}
Since $\nabla d_D(x)=n(x)$, the Jacobian of $\Psi_t$ is given by
\begin{align*}
 D\Psi_t(x) = I -\psi_t'(|d_D(x)|) n(x)\otimes n(x)
	     + \psi_t(|d_D(x)|) D^2 d_D(x),
\end{align*}
and $\sup_{x\in D_{\eps_0}}\|D\Psi_t(x)-I\| \to 0$ as $t\to 0$ by construction of $\psi_t$.
We note that $\Psi_{t\mid \partial D}= \Phi_t$.
%

\section{Weighted Sobolev spaces}\label{sec:weighted_sobolev}

In order to construct the weighted spaces mentioned in the introduction let us begin with defining another level set function for the domain $D$ resembling a sign function smoothed in $\Gamma_\eps$, namely
\begin{align*}
 \varphi^\eps(x)= S\left(\frac{-d_D(x)}{\eps}\right),
\end{align*}
where the function $S$ is a regularization of the sign function. 
Hence, $\varphi^\eps(x)>0$ if and only if $x\in D$.
To be precise, we assume that $S$ verifies the following assumptions.

\begin{enumerate}
 \item[(S1)] $S:\RR\to\RR$ is Lipschitz continuous, $S(t)=t/|t|$ for $|t|\geq 1$, and $S'(t)>0$ for $|t|<1$. Moreover, $S(t)=-S(-t)$ for all $t\in\RR$.
  \item[(S2)] There exist  $\zeta_1,\zeta_2>0$ and $\alpha>0$ such that for all $t\in (0,2)$
 \begin{align*}
  \zeta_1 t^\alpha \leq (1+S(t-1))/2\leq \zeta_2 t^\alpha.
 \end{align*}
 \item[(S3)] $S'(t)\leq S'(s)$ for all $0\leq s\leq t<1$.
\end{enumerate}
(S3) asserts concavity of $S$ on $(0,1)$ and this assumption is only needed in Theorem~\ref{thm:sobolev_diffuse_boundary2} below.
Assumption (S2) ensures that that the phase-field function $\omega^\eps$ is proportional to ${\rm dist}(\cdot,\partial D_\eps)^\alpha$ on $\Gamma_\eps$,
where $\omega^\eps$ is defined as a regularization of the indicator function $\chi_D$ of $D$ as follows
\begin{align*}
 \omega^\eps(x) = \frac{1}{2}\left(1+\varphi^\eps(x)\right).
\end{align*}

Obviously, we have that $D_\eps=\{x\in\Omega:\ \omega^\eps(x)>0\}$, and $\omega^\eps(x)=1$ for $x\in D$ with ${\rm dist}(x,\partial D)>\eps$.
Let us clarify (S2) in the following. We observe that ${\rm dist}(x,\partial D_\eps)=\eps-d_D(x)$ for $x\in\Gamma_\eps$. Thus, by (S2) with $t={\rm dist}(x,\partial D_\eps)/\eps$
\begin{align}\label{eq:weight_near_boundary}
\zeta_1  \left(\frac{{\rm dist}(x,\partial D_\eps)}{\eps}\right)^\alpha\leq \omega^{\eps}(x)\leq \zeta_2 \left(\frac{{\rm dist}(x,\partial D_\eps)}{\eps}\right)^\alpha
\end{align}

%
Before proceeding, let us give a few examples which may serve as prototypes for $S$.

\begin{example}\label{ex:S}

 (i) Let $S(t)=t$ for $|t|<1$. Obviously (S1) and (S3) are satisfied. Moreover, $1+S(t-1)=t$,
  and we can choose $\alpha=1$, and $\zeta_1=\zeta_2=1/2$ in (S2).
  
  (ii) Let $S(t)= (3t-t^3)/2$. Thus, $S\in C^1(\RR)$ and $S'(t)=3(1-t^2)/2>0$ for $|t|<1$. Since $S''(t)=-3t<0$ for $t>0$, (S3) is satisfied. Moreover, $1+S(t-1)=(3t^2-t^3)/2$, and (S2) is satisfied for $\alpha=2$ and $\zeta_1=1/4$ and $\zeta_2=3/4$.
  
  (iii) Let $S(t)=15 t/8-5t^3/4+3t^5/8$. Thus, $S\in C^2(\RR)$ and $S'(t)=15/8-15t^2/4+15t^4/8>0$ for $|t|<1$. Since $S''(t)=-15t/2+15t^3/2<0$ for $t>0$, (S3) is satisfied. Moreover, $1+S(t-1)=5 t^3/2-15 t^4/8+3t^5/8$, and (S2) is satisfied for $\alpha=3$ and $\zeta_1=1/8$ and $\zeta_2=5/2$.
\end{example}

Proceeding with the construction of weighted Lebesgue spaces, let us introduce the measure
\begin{align*}
 \omega^\eps(\widetilde\Omega) = \int_{\widetilde\Omega} \omega^\eps(x) \d x, \quad \widetilde\Omega\subset\RR^n \text{ measurable},
\end{align*}
which is absolutely continuous with respect to the Lebesgue measure.
Associated to $\omega^\eps$ let us further introduce for $1\leq p <\infty$ the weighted $L^p$-spaces
\begin{align*}
 L^p(D_\eps;\omega^\eps) = \{v:D_\eps\to \RR:\ |v|^p\omega^\eps \in L^1(D_\eps)\}
 \quad\text{with norm } 
 \|v\|_{L^p(D_\eps;\omega^\eps)}^p =\int_{D_\eps} |v|^p \d\omega^\eps.
\end{align*}
For $p=\infty$, we set $L^\infty(D_\eps;\omega^\eps)=L^\infty(D_\eps)$, i.e. $L^\infty(D_\eps;\omega^\eps)$ is the class of Lebesgue-measurable functions being essentially bounded.
In the following we will use also the notation $L^p(D_\eps;\delta)$ where $\delta$ is an appropriate weight function.
The following statement provides some basic relations between the weighted $L^p$-spaces.

\begin{lemma}\label{lem:inclusion}
 Let $1\leq p\leq \infty$ and let $\eps>0$. Then $L^p(D_\eps)\subset L^p(D_\eps;\omega^\eps)$, and for every $v\in L^p(D_\eps;\omega^\eps)$ there holds $v_{\mid D}\in L^p(D)$. Moreover,
 
 (i) for $\eps>\tilde \eps\geq 0$ we have
 \begin{align*}
  \|v\|_{L^p(D_{\tilde\eps};\omega^{\tilde\eps})} \leq 2^{1/p} \|v\|_{L^p(D_\eps;\omega^\eps)}\quad\text{for all } v\in L^p(D_\eps;\omega^\eps).
 \end{align*}
 (ii) for $0<\gamma<\eps/2$ we have
 \begin{align*}
  \|v\|_{L^p(D_{\eps-2\gamma};\omega^{\eps+\gamma})} \leq \left(\frac{3^{\alpha}\zeta_2}{\zeta_1}\right)^{1/p} \|v\|_{L^p(D_{\eps-2\gamma};\omega^{\eps-\gamma})}\quad\text{for all } v\in L^p(D_\eps;\omega^\eps).
 \end{align*}
\end{lemma}
\begin{proof}
 (i) For $v\in L^p(D_\eps;\omega^\eps)$ we obtain
 \begin{align*}
  \int_{D_{\tilde\eps}} |v|^p  \d \omega^{\tilde\eps} = \int_{D_{\tilde\eps\setminus D}} |v|^p \omega^{\tilde\eps} \d x + \int_{D} |v|^p \omega^{\tilde\eps} \d x.
 \end{align*}
 For $x\in D_{\tilde\eps}\setminus D$ we have $\omega^{\tilde\eps}(x)\leq \omega^\eps(x)$, and on $D$ there holds $1/2\leq \omega^{\tilde\eps}\leq 1 \leq 2\omega^\eps$. The fact that $D_{\tilde\eps}\subset D_\eps$ yields the assertion.
 
 (ii) Similar to (i) we have $\omega^{\eps+\gamma}\leq \omega^{\eps-\gamma}$ on $D$, whence
 \begin{align*}
  \|v\|_{L^p(D_{\eps-2\gamma};\omega^{\eps+\gamma})}^p \leq \int_{D} |v|^p \d\omega^{\eps-\gamma} +  \int_{D_{\eps-2\gamma}\setminus D} |v|^p \d\omega^{\eps+\gamma}.
 \end{align*}
 For $z\in D_{\eps-2\gamma}\setminus D$, we can write $z=x-tn(x)$ with $x\in\partial D_{\eps-2\gamma}$ and $0\leq t\leq 2\eps$, see Section~\ref{sec:tubular}.
 Then, using \eqref{eq:weight_near_boundary}, we have
 \begin{align*}
    \omega^{\eps+\gamma}(z) &\leq \zeta_2 \left(\frac{{\rm dist}(z,\partial D_{\eps+\gamma})}{\eps+\gamma}\right)^\alpha = \zeta_2\left(\frac{t+3\gamma}{\eps+\gamma}\right)^\alpha,\quad\text{and}\\
    \omega^{\eps-\gamma}(z) &\geq \zeta_1 \left(\frac{{\rm dist}(z,\partial D_{\eps-\gamma})}{\eps-\gamma}\right)^\alpha = \zeta_1\left(\frac{t+\gamma}{\eps-\gamma}\right)^\alpha.
 \end{align*}
 This implies
 \begin{align*}
  \frac{\omega^{\eps+\gamma}(z)}{\omega^{\eps-\gamma}(z)} \leq \frac{\zeta_2}{\zeta_1} \left(\frac{t+3\gamma}{t+\gamma}\, \frac{\eps-\gamma}{\eps+\gamma}\right)^\alpha = \frac{\zeta_2}{\zeta_1}\left(\left(1+\frac{2\gamma}{t+\gamma}\right) \frac{\eps-\gamma}{\eps+\gamma}\right)^\alpha \leq \frac{\zeta_23^\alpha}{\zeta_1} ,
 \end{align*}
 from which we easily obtain the assertion. 
\end{proof}

Associated to the weighted spaces $L^p(D_\eps;\omega^\eps)$, let us define the weighted Sobolev spaces
\begin{align*}
 W^{1,p}(D_\eps;\omega^\eps) = \{ v\in L^p(D_\eps;\omega^\eps):\ \partial_{x_i} v \in L^p(D_\eps;\omega^\eps),\ 1\leq i\leq n\}
\end{align*}
with norm
\begin{align*}
 \|v\|_{W^{1,p}(D_\eps;\omega^\eps)}^p = \int_{D_\eps} |v|^p + |\nabla v|^p \d\omega^\eps.
\end{align*}

In view of \eqref{eq:weight_near_boundary}, several results of Kufner \cite{Kufner1985} concerning power-type weights can be applied. This is due to the fact that the proofs of \cite{Kufner1985} need the power-type property of the weight only in a vicinity of the boundary; in the remaining subset of $D_\eps$, say $D$, uniform boundedness away from zero and boundedness of the weight is used, which in turn allows to use results for unweighted spaces. We will employ similar techniques in the next section.
We have the following.
For $1\leq p <\infty$, the spaces $L^p(D_\eps;\omega^\eps)$ and $W^{1,p}(D_\eps;\omega_\eps)$ are separable Banach spaces \cite[Theorem~3.6]{Kufner1985}, and $C^\infty(\overline{D_\eps})$ is dense in $W^{1,p}(D_\eps;\omega^\eps)$ \cite[Theorem~7.2]{Kufner1985}.
In case $p=2$, the spaces $L^2(D_\eps;\omega^\eps)$ and $W^{1,2}(D_\eps;\omega^\eps)$ are Hilbert spaces with obvious definition of the inner product.


\section{Properties of Sobolev spaces for diffuse interfaces}\label{sec:properties_weighted_sobolev}

In this section we establish basic results for weighted Sobolev spaces for diffuse interfaces which are crucial for the analysis of variational boundary value problems. We particularly investigate the dependence on the parameter $\eps$. As shown below, our results yield constants independent of $\eps$, for instance the trace constant or the Poincar\'e constant in weighted Sobolev spaces, see Theorem~\ref{thm:trace} or Theorem~\ref{thm:poincare_general}. These results may be of interest in their own.

\subsection{Trace lemma}
The following auxiliary lemma is a slight adaptation of \cite[Lemma~2.1]{Arrieta08}. It states that the trace operator for unweighted Sobolev spaces is uniformly bounded for certain perturbations of the domain, and it is the key observation in proving a similar statement also for weighted Sobolev spaces. For convenience of the reader, we sketch a proof.
\begin{lemma}\label{thm:trace_unweighted}
 Let $\eps_0$ be sufficiently small. Then there is a constant $C>0$ such that for each $t\in (-\eps_0,\eps_0)$ and $v\in W^{1,p}(D_t)$, $1\leq p<\infty$,
 \begin{align}\label{eq:uniform_trace}
 \int_{\partial D_t} |v|^p \d\sigma \leq C \int_{D_t} |\nabla v|^p + |v|^p \d x.
\end{align}
\end{lemma}
\begin{proof}
Let $\eps_0>0$ be sufficiently small and let $t\in(-\eps_0,\eps_0)$. 
Let us first consider the case $p=1$. Denote by $\Psi_t:D\to D_t$ and $\Phi_t:\partial D\to \partial D_t$ the transformations defined in Section~\ref{sec:boundary_transformation}.
By a change of variables $u(x)=v(\Psi_t(x))$ for $x\in D$ and $u(x)=v(\Phi_t(x))$ for $x\in\partial D$, it follows that
\begin{align*}
 \inf_{x\in D} \det(D\Psi_t(x)) \int_D |\nabla u| + |u| \d x &\leq \int_{D_t} |\nabla v| + |v| \d x \leq \sup_{x\in D} \det(D\Psi_t(x)) \int_D |\nabla u| + |u| \d x,\\
 \inf_{x\in \partial D} \det(D\Phi_t(x))\int_{\partial D} |u| \d\sigma &\leq \int_{\partial D_t} |v| \d\sigma \leq \sup_{x\in \partial D} \det(D\Phi_t(x)) \int_{\partial D}  |u| \d\sigma.
\end{align*}
In view of Section~\ref{sec:boundary_transformation}, as $t\to 0$
\begin{align*}
 C_1(t)&=\min\{ \inf_{x\in D} \det(D\Psi_t(x)),\inf_{x\in \partial D} \det(D\Phi_t(x))\} \to 1,\quad\text{and}\\
 C_2(t)&=\max\{\sup_{x\in D} \det(D\Psi_t(x)),\sup_{x\in \partial D} \det(D\Phi_t(x))\}\to 1.
\end{align*}
Denote by $C$ a bound for the norm of the trace operator $W^{1,1}(D)\to L^1(\partial D)$. We then obtain
\begin{align*}
  \int_{\partial D_t} |v| \d\sigma \leq C \frac{C_2(t)}{C_1(t)} \int_{\partial D_t}  |\nabla v| + |v| \d x.
\end{align*}
For the general case $p>1$, we apply the latter estimate to $\tilde v = |v|^{p}$. We observe that $|\nabla \tilde v| = p |v|^{p-1} |\nabla v|$, and, using Young's inequality, $p|v|^{p-1}|\nabla v| + |v|^p \leq |\nabla v|^p + |v|^p$,
which concludes the proof.
\end{proof}
\begin{theorem}[Trace lemma]\label{thm:trace}
 Let $\eps_0>0$ be sufficiently small. Then there exists a constant $C>0$ such that for $\eps \in (0,\eps_0)$ and for $ v \in W^{1,p}(D_\eps;\omega^\eps)$, $1\leq p<\infty$,
 \begin{align*}
  \int_{D_\eps} |v|^p |\nabla \omega^\eps|\d x \leq C \|v\|^p_{W^{1,p}(D_\eps;\omega^\eps)}.
 \end{align*}
\end{theorem}
\begin{proof}
 According to the coarea-formula there holds
 \begin{align*}
  \int_{D_\eps} |v|^p |\nabla \varphi^\eps| \d x = \int_{-\infty}^\infty \int_{(\varphi^\eps)^{-1}(s)} |v|^p \d\sigma \d s =\int_{-1}^1 \int_{\{\varphi^\eps=s\}} |v|^p \d\sigma \d s.
 \end{align*}
 Since $\{\varphi^\eps=s\}=\partial \{\varphi^\eps>s\} = \partial D_t$ for $t=-\eps S^{-1}(s)$, we can use \eqref{eq:uniform_trace} to obtain
 \begin{align*}
  \int_{-1}^1 \int_{\{\varphi^\eps=s\}} |v|^p \d\sigma \d s \leq C \int_{-1}^1 \int_{\{\varphi^\eps>s\}} |\nabla v|^p + |v|^p \d x \d s
  =C \int_{D_\eps} \int_{-1}^{\varphi^\eps(x)} \d s (|\nabla v|^p + |v|^p) \d x,
 \end{align*}
 where we used Fubini's theorem in the last step. The assertion follows from $|\nabla \omega^\eps|=\frac{1}{2}|\nabla\varphi^\eps|$.
\end{proof}
 The trace theorem~\ref{thm:trace} shows that for $g\in W^{1,p}(D_\eps;\omega^\eps)$ the diffuse boundary integral introduced in \eqref{eq:diffuse_integrals} actually exists.

\subsection{Embedding theorem}
The following embedding theorem uses the representation \eqref{eq:weight_near_boundary} of the weight near the boundary as a power of the distance function $\delta(x)={\rm dist}(x,\partial D_\eps)$.
Let us define the Sobolev conjugate $p_\alpha^*$ for weighted spaces
\begin{align}\label{eq:conjugate_sobolev}
 p_\alpha^* =   \frac{p(n+\alpha)}{n+\alpha-p}\quad\text{ for } p<n+\alpha,\qquad\text{and}\qquad p_\alpha^*=\infty \quad\text{ for } p\geq n+\alpha.
\end{align}
We observe that $p^*_0$ is the ``usual'' Sobolev conjugate for unweighted Sobolev spaces, see \cite{Adams1975}, and $p^*_\alpha$ is strictly decreasing with respect to $\alpha$ on $(0,\infty)$.

\begin{theorem}[Embedding]\label{thm:embdding}
 Let $0<\eps<\eps_0$, and let $\alpha>0$ be the constant from (S2). Then the following embeddings are continuous
 \begin{align*}
  W^{1,p}(D_\eps,\omega^\eps) \hookrightarrow L^q(D_\eps,\omega^\eps),\qquad 1\leq q \leq p_\alpha^*\text{ and } q<\infty.
 \end{align*}
  Moreover, there exists a constant $C$ independent of $\eps$ such that for $u\in W^{1,p}(D_\eps;\omega^\eps)$
 \begin{align}\label{eq:weighted_sobolev}
  \| u\|_{L^q(D_\eps;\omega^\eps)}  \leq C \| u \|_{ W^{1,p}(D_\eps;\omega^\eps) }.
 \end{align}
\end{theorem}
 The first part of the theorem can be found in \cite[Theorem~3]{Horiuchi1989}, see also \cite[Theorem~19.9]{OpicKufner1990} for the case $q<p_\alpha^*$. To show that the embedding is independent of $\eps$, we will give a proof in the spirit of \cite{OpicKufner1990}. To do so, we employ the following two lemmata. 
 The first of which uses Sobolev's embedding theorem on balls and a covering argument, and is similar to the arguments of \cite{OpicKufner1990}.
 The second is a Hardy-inequality-type argument for diffuse interfaces which seems to be new.
 We let $\delta(x)={\rm dist}(x,\partial D_\eps)$ in the following.

  \begin{lemma}\label{lem:sobolev_boundary}
    Let $\eps>0$ and $\alpha>0$. Furthermore, let $1\leq q<\infty$ such that $n+\alpha\geq (n+\alpha-1)q$. Then there exists a constant $C>0$ independent of $\eps$ such that for every $u\in W^{1,1}(D_\eps;\delta^\alpha)$
  \begin{align*}
   \|u\|_{L^q(D_\eps\setminus D;\delta^{\alpha})} \leq  C\left( \|u\|_{L^1(\Gamma_\eps;\delta^{\alpha-1})}  + \|\nabla u\|_{L^1(\Gamma_\eps; \delta^\alpha)}\right).
  \end{align*}
  \end{lemma}
  \begin{proof}
  We proceed as in \cite{OpicKufner1990}. Let $r(x) = \delta(x)/3$. According to the Besicovitch covering theorem, cf.\@ \cite[Lemma~18.3]{OpicKufner1990}, there exists a sequence $\{x_k\}\subset D_\eps\setminus D$ and an integer $\theta$ depending only on $n$ such that
 \begin{align*}
  D_\eps\setminus D \subset \bigcup_{k=1}^\infty B_k \subset \Gamma_\eps,\quad B_k=B_{r(x_k)}(x_k)\text{ and }
  \sum_{k=1}^\infty \chi_{B_k}(x) \leq \theta \text{ for all } x\in \RR^n.
 \end{align*}
 Employing the Sobolev embedding theorem for balls \cite{Adams1975}, we obtain as in \cite[Theorem~18.6]{OpicKufner1990}
 \begin{align}\label{eq:sobolev}
  \|u\|_{L^q(B_k;\delta^{\alpha})} \leq C \delta(x_k)^{\frac{n+\alpha}{q}-n+1-\alpha} 
  \left(\|u\|_{L^1(B_k;\delta^{\alpha-1})} + \|\nabla u\|_{L^1(B_k;\delta^\alpha)}\right)
 \end{align}
  for all $q\leq n/(n-1)$. Note the different powers of the weight on the right-hand side of \eqref{eq:sobolev}. Assuming, without loss of generality, that $\eps_0\leq 1$ and thus $\delta(x_k)\leq 1$, we can bound the right-hand side of \eqref{eq:sobolev} as long as $\alpha$ and $q$ are such that $n+\alpha\geq (n+\alpha-1)q$. Hence, by summation over $k$, we obtain the assertion with
  a constant $C$ depending on $n$ and the Sobolev embedding constant for the unit ball, but not on $\eps$. 
 \end{proof}

 \begin{lemma}[Hardy-type inequality]\label{lem:hardy_inequality}
  Let $0<\eps<\eps_0$ and let $\alpha>0$. Then there exists a constant $C>0$ independent of $\eps$ such that for every $u\in W^{1,1}(D_\eps;\delta^\alpha)$
  \begin{align*}
   \|u\|_{L^1(\Gamma_\eps;\delta^{\alpha-1})} \leq \frac{C}{\alpha}\left(\eps^\alpha\|u\|_{W^{1,1}(D)} + \|\nabla u\|_{L^1(\Gamma_\eps;\delta^\alpha)} \right).
  \end{align*}
 \end{lemma}
  \begin{proof}
 Let $u\in C^\infty(\overline{D_\eps})\cap W^{1,1}(D_\eps;\delta^\alpha)$. We obtain by using \eqref{eq:boundary_transformation_formula}, $1/2\leq |\det D\Phi_{-t}|\leq 2$, and one-dimensional integration-by-parts
 \begin{align}
  \int_{\Gamma_\eps} |u| \delta^{\alpha-1} \d x 
  &= \int_{\partial D_\eps} \int_0^{2\eps} |u(x-tn(x))| t^{\alpha-1} |\det D\Phi_{-t}|\d t \d\sigma_\eps(x) \notag\\
  &\leq \frac{2}{\alpha}\int_{\partial D_\eps} |u(x-2\eps n(x))| (2\eps)^{\alpha} \d\sigma_\eps(x) + \frac{2}{\alpha} \int_{\partial D_\eps}\int_0^{2\eps} |\nabla u(x-tn(x))| t^{\alpha} \d\sigma_\eps(x).\label{eq:hardy}
 \end{align}
 To treat the first integral in \eqref{eq:hardy}, we employ the transformation $\Phi_{2\eps}:\partial D_{-\eps}\to \partial D_\eps$ defined in Section~\ref{sec:boundary_transformation}, i.e.
 \begin{align*}
  \int_{\partial D_\eps} |u(x-2\eps n(x))| \d\sigma_\eps(x) \leq 4 \int_{\partial D_{-\eps}} |u(x)| \d\sigma_{-\eps}(x).
 \end{align*}
 From $u_{\mid D}\in W^{1,1}(D)$, and $D_{-\eps}\subset D$, and the trace lemma~\ref{thm:trace_unweighted} we deduce that there exists a constant $C>0$ independent of $\eps$ such that
 \begin{align*}
   \int_{\partial D_{-\eps}} |u(x)| \d\sigma_{-\eps}(x) \leq C \|u\|_{W^{1,1}(D)},
 \end{align*}
 i.e. $\int_{\partial D_\eps} |u(x-2\eps n(x))| \d\sigma_\eps(x) \leq C\|u\|_{W^{1,1}(D)}$.
 The assertion follows from
  \begin{align*}
   \int_{\partial D_\eps}\int_0^{2\eps} |\nabla u(x-tn(x))| t^{\alpha} \d t \d\sigma_\eps(x)\leq 2 \int_{\Gamma_\eps} |\nabla u| \delta^\alpha \d x,
  \end{align*}
  and a density argument.
  \end{proof}
  \begin{remark}
   Let us state that the arguments of \cite{Kufner1985} are based on partition of unity $\{\varphi_i\}$ subordinate to $\{B_j\}\cup \{D\}$ where $\{B_j\}$ is a finite cover of $\Gamma_\eps$. Then the Hardy-type argument of the latter proof is applied to $v_i = u \varphi_i$ which is zero on the boundary of $B_i$. Hence, the first term in \eqref{eq:hardy} vanishes. However, $\nabla v_i = \varphi_i \nabla u + u \nabla \varphi_i$, and $|\nabla\varphi_i|\sim 1/\eps$. Thus, the techniques of \cite{Kufner1985} are not directly applicable as we strive for constants uniformly bounded in terms of $\eps$.
  \end{remark}

 \begin{proof}[Proof of Theorem~\ref{thm:embdding}]
 We split the norm into the diffuse interface part and the interior part
 \begin{align*}
  \|u\|_{L^q(D_\eps;\omega^\eps)}^q = \int_{D_\eps\setminus D} |u|^q \omega^\eps\d x + \int_D |u|^q \omega^\eps \d x.
 \end{align*}
 From the Sobolev embedding theorem \cite{Adams1975}, we have that $W^{1,p}(D;\omega^\eps)\hookrightarrow L^q(D;\omega^\eps)$ is continuous for each $q\leq p^*_0= np/(n-p)$ if $p<n$, and for $q<\infty$ if $p\geq n$. Since $p_\alpha^* \leq p_0^*$ and $D$ is bounded, we only have to estimate the $L^q$-norm of $u$ on $D_\eps\setminus D$.
  
 First consider the case $p=1$, $q\leq p_\alpha^*=(n+\alpha)/(n+\alpha-1)$. For this choice, the condition $n+\alpha\geq (n+\alpha-1)q$ is obviously satisfied.
 Combining Lemma~\ref{lem:hardy_inequality} and Lemma~\ref{lem:sobolev_boundary} yields
 \begin{align*}
  \int_{D_\eps\setminus D} |u|^q \delta^\alpha \d x
  \leq C\left( \int_{\Gamma_\eps} |u|\delta^{\alpha-1} + |\nabla u| \delta^\alpha \d x \right)^{q}
  \leq 
  C\left(\eps^\alpha \|u\|_{W^{1,1}(D)} + \|\nabla u\|_{L^1(\Gamma_\eps;\delta^\alpha)} \right)^{q}.
 \end{align*}
 Multiplication of the latter inequality with $1/\eps^\alpha$, taking the $q$th root and using \eqref{eq:weight_near_boundary}, i.e. $\delta^\alpha/\eps^\alpha\approx \omega^\eps$, further gives, 
 \begin{align*}
  \|u\|_{L^q(D_\eps\setminus D;\omega^\eps)}\leq C \eps^{\alpha(1-1/q)} \left( \|u\|_{W^{1,1}(D)} + \|\nabla u\|_{L^1(\Gamma_\eps;\omega^\eps)} \right).
 \end{align*}
 Summarizing, we have shown that for each $1\leq q \leq (n+\alpha)/(n+\alpha-1)$ there holds
 \begin{align*}
  \|u\|_{L^q(D_\eps;\omega^\eps)} \leq C \|u\|_{W^{1,1}(D_\eps;\omega^\eps)}.
 \end{align*}
 For the general case $p>1$, we apply the previous results to $v=|u|^{1+q (p-1)/p}$ and $\tilde q=(1-\frac{1}{p}+\frac{1}{q})^{-1}$. 
 One easily verifies that $q\leq p(n+\alpha)/(n+\alpha-p)$ is equivalent to $\tilde q\leq (n+\alpha)/(n+\alpha-1)$.
 Moreover,
 $|v|^{\tilde q} = |u|^q$ and
  $|\nabla v| = (1+\frac{q (p-1)}{p}) |u|^{q(p-1)/p} |\nabla u|$.
 Whence, H\"older's inequality yields
 \begin{align*}
  \|v\|_{W^{1,1}(D_\eps;\omega^\eps)} &= \int_{D_\eps} |u|^{1+\frac{q(p-1)}{p}} + \big(1+\frac{q (p-1)}{p}\big) |u|^{\frac{q(p-1)}{p}} |\nabla u| \d\omega^\eps\\
  &\leq \left(\|u\|_{L^p(D_\eps;\omega^\eps)} + \big(1+\frac{q(p-1)}{p}\big) \|\nabla u\|_{L^p(D_\eps;\omega^\eps)} \right) \|u\|_{L^q(D_\eps;\omega^\eps)}^{\frac{q(p-1)}{p}}.
 \end{align*}
 This together with the identity
 \begin{align*}
  \|v\|_{L^{\tilde q}(D_\eps;\omega^\eps)} = \|u\|_{L^q(D_\eps;\omega^\eps)}^{1+\frac{q(p-1)}{p}}
 \end{align*}
  yields the assertion.
 \end{proof}

 \begin{remark}
  As already noted, $p^*_\alpha$ is strictly decreasing with respect to $\alpha$ on $(0,\infty)$. Loosely speaking, compared to the unweighted Sobolev embedding, we loose $\alpha$ spatial dimensions. For instance, if $\alpha=1$, we have that $2_1^*=6$ for $n=2$, and $2_1^*=4$ for $n=3$.
  This fact is intimately related to the Hardy inequality and isoperimetric inequalities, cf.\@ \cite{Horiuchi1989} where also counterexamples are given showing that the restriction $q\leq p_\alpha^*$ cannot be improved in general.
  However, embedding in certain H\"older spaces is possible  \cite{Horiuchi1989,OpicKufner1990}. Adapting the above proofs it should be possible to show that even in this situation the embedding constants are independent of $\eps$.
 \end{remark}

\begin{proposition}[Compactness]\label{prop:compact}
 Let $0<\eps<\eps_0$, $\alpha$ be the constant in (S2), and let $1\leq p <\infty$. Then the following embeddings are compact
 \begin{align*}
  W^{1,p}(D_\eps,\omega^\eps) \hookrightarrow L^q(D_\eps,\omega^\eps),\qquad 1\leq q < p_\alpha^*.
 \end{align*}
\end{proposition}
\begin{proof}
 Let $q<p_\alpha^*$ and let $\{u_k\}\subset W^{1,p}(D_\eps;\omega^\eps)$ be bounded; say by a constant $C_p>0$. Furthermore denote by $C_e$ the constant of the embedding $W^{1,p}(D_\eps;\omega^\eps)\to L^{p_\alpha^*}(D_\eps;\omega^\eps)$. Since $L^q(D_\eps;\omega^\eps)$ is complete, we have to show that a subsequence of $\{u_k\}$ is Cauchy in $L^q(D_\eps;\omega^\eps)$.
 Therefore, let $\iota>0$ and choose $\gamma=\min\{\eps,(C_p C_e)^{qp/(q-p_\alpha^*)} \iota/2 \}$.
 Since the embedding $W^{1,p}(D_{\eps-\gamma};\omega^\eps) \hookrightarrow L^q(D_{\eps-\gamma};\omega^\eps)$ is compact \cite{Adams1975}, we can extract a subsequence, again denoted by $\{u_k\}$, which is Cauchy in $L^q(D_{\eps-\gamma};\omega^\eps)$.
 Hence, there exists $N=N(\iota)\in\NN$ such that 
 $\|u_k-u_l\|_{L^q(D_{\eps-\gamma};\omega^\eps)}<\iota/2$
 for all $k,l\geq N$.
 Let $k,l\geq N$ in the following.
 Thus, using the triangle inequality, we have that
 \begin{align}\label{eq:help}
  \| u_k-u_l\|_{L^q(D_{\eps};\omega^\eps)} \leq \| u_k-u_l\|_{L^q(D_\eps\setminus D_{\eps-\gamma};\omega^\eps)}  + \frac{\iota}{2}.
 \end{align}
  Since $1\leq q< p_\alpha^*$, we obtain by using H\"older's inequality and the embedding theorem
 \begin{align*}
  \| u_k-u_l\|_{L^q(D_\eps\setminus D_{\eps-\gamma};\omega^\eps)} 
  &\leq \| u_k-u_l\|_{L^{p_\alpha^*}(D_\eps\setminus D_{\eps-\gamma};\omega^\eps)} \gamma^{\frac{1}{q}-\frac{1}{p_\alpha^*}}\\
  &\leq C_e \| u_k-u_l\|_{W^{1,p}(D_\eps;\omega^\eps)} \gamma^{\frac{1}{q}-\frac{1}{p_\alpha^*}}\leq \frac{\iota}{2}
 \end{align*}
 by choice of $\gamma$.
 This in combination with \eqref{eq:help} shows that $\{u_k\}$ is Cauchy in $L^q(D_\eps;\omega^\eps)$.
\end{proof}

The idea of the proof of the previous compactness result can already be found in \cite{OpicKufner1990}.

\subsection{Diffuse Poincar\'e-Friedrichs inequalities}

The last issue concerning basic results in weighted Sobolev spaces are Poincar\'e-Friedrichs inequalities, which we again want to derive with constants independent of $\eps$. We start with a quite general result:

\begin{theorem}[Poincar\'e-type inequality]\label{thm:poincare_general}
 Fix $1\leq p<\infty$. Assume that $D_\eps$ is connected for each $\eps\in [0,\eps_0]$, and let $K_\eps\subset W^{1,p}(D_\eps;\omega^\eps)$, be a family of closed cones, i.e. for $u\in K_\eps$ there holds $\lambda u\in K_\eps$ for all $\lambda>0$, such that $K_\eps$ contains only the zero function as a constant function. Then there exists a constant $C>0$ independent of $\eps$ such that
 \begin{align}\label{eq:poincare}
  \|u\|_{L^p(D_\eps;\omega^\eps)} \leq C \|\nabla u\|_{L^p(D_\eps;\omega^\eps)}\quad\text{for all } u\in K_\eps.
 \end{align}
\end{theorem}
\begin{proof}
 Assume \eqref{eq:poincare} is not true. Then there exist sequences $\{u_k\}\subset W^{1,p}(D_{\eps_k};\omega^{\eps_k})\cap K_{\eps_k}$ with $\|u_k\|_{L^p(D_{\eps_k};\omega^{\eps_k})}=1$ and $\{\eps_k\}\subset [0,\eps_0]$ such that
 \begin{align}\label{eq:help0}
  \| \nabla u_{k}\|_{L^p(D_{\eps_k};\omega^{\eps_k})}\leq \frac{1}{k}.
 \end{align}
 Since $\eps_k\in [0,\eps_0]$ the Bolzano-Weierstra\ss{} theorem implies the existence of a $\tilde \eps\in[0,\eps_0]$ such that for a subsequence, relabeled if necessary, $\eps_k\to \tilde\eps$ as $k\to \infty$.
 Hence, for all $\gamma>0$ there exists $N(\gamma)\in\NN$ such that $\eps_k\in (\tilde\eps-\gamma,\tilde\eps+\gamma)$ for all $k\geq N(\gamma)$. In the following let $0<\gamma<\tilde\eps/2$ and $k\geq N(\gamma)$.
 By H\"older's inequality and the embedding theorem for $q=p_\alpha^*$, we have
 \begin{align}\label{eq:help1}
  \|u_k\|_{L^p(D_{\eps_k}\setminus D_{\tilde\eps-2\gamma};\omega^{\eps_k})}
  \leq \|u_k\|_{L^q(D_{\eps_k}\setminus D_{\tilde\eps-2\gamma}D;\omega^{\eps_k})}  (2\gamma)^{\frac{1}{n+\alpha}}
  \leq C \|u_k\|_{W^{1,p}(D_{\eps_k};\omega^{\eps_k})} \gamma^{\frac{1}{n+\alpha}}.
 \end{align}
Furthermore, since $\tilde\eps-\gamma \leq \eps_k$, by Lemma~\ref{lem:inclusion} (i)
 \begin{align*}
  \|u_k\|_{W^{1,p}(D_{\tilde\eps-\gamma};\omega^{\tilde\eps-\gamma})} \leq \|u_k\|_{W^{1,p}(D_{\eps_k};\omega^{\eps_k})} \leq C.
 \end{align*}
 In view of Proposition~\ref{prop:compact}, we can therefore extract a subsequence, relabeled if necessary, such that $u_k\to u$ in $L^p(D_{\tilde\eps-\gamma};\omega^{\tilde\eps-\gamma})$. Moreover, we deduce from \eqref{eq:help0} and Lemma~\ref{lem:inclusion} (i) that
 \begin{align*}
  \|\nabla u_k\|_{L^{p}(D_{\tilde\eps-\gamma};\omega^{\tilde\eps-\gamma})} \leq \|\nabla u_k\|_{L^{p}(D_{\eps_k};\omega^{\eps_k})} \leq \frac{1}{k},
 \end{align*}
 and hence $u_k\to u$ in $W^{1,p}(D_{\tilde\eps-\gamma};\omega^{\tilde\eps-\gamma})$ and $\nabla u=0$. Since $K_{\tilde\eps-\gamma}$ is closed and $D_{\tilde\eps-\gamma}$ is connected, $u\in K_{\tilde\eps-\gamma}$ and $u=0$ on $D_{\tilde\eps-\gamma}$.
 In view of Lemma~\ref{lem:inclusion} (ii), we further obtain
 \begin{align*}
  \|u_k\|_{L^p(D_{\tilde\eps-2\gamma};\omega^{\eps_k})} \leq 2 \|u_k\|_{L^p(D_{\tilde\eps-2\gamma};\omega^{\tilde\eps+\gamma})}
  \leq C\|u_k\|_{L^p(D_{\tilde\eps-2\gamma};\omega^{\tilde\eps-\gamma})}
  \leq C \|u_k\|_{L^p(D_{\tilde\eps-\gamma};\omega^{\tilde\eps-\gamma})}.
 \end{align*}
 This in combination with \eqref{eq:help1} implies
 \begin{align*}
  1=\lim_{k\to \infty} \|u_k\|_{L^p(D_{\eps_k};\omega^{\eps_k})} \leq C\gamma^{\frac{1}{n+\alpha}} + \lim_{k\to \infty}\|u_k\|_{L^p(D_{\tilde\eps-2\gamma};\omega^{\eps_k})}= C\gamma^{\frac{1}{n+\alpha}}.
 \end{align*}
 Since $\gamma>0$ was arbitrary, this is the desired contradiction.
 %
\end{proof}

Let us remark that in \cite{BoulkhemairChakib2007} a similar result has been obtained for unweighted spaces, i.e. a Poincar\'e inequality with a constant which is independent of certain perturbations of $\partial D$.
For illustration of the previous result let us state the ``usual'' Poincar\'e and Friedrichs inequality in their weighted form.
\begin{corollary}\label{cor:poincare}
 Let $\eps\in [0,\eps_0]$, $1 \leq p<\infty$, and let $D_\eps$ be connected. Then there exists a constant $C$ independent of $\eps$ such that 
 \begin{align*}
  \|u-\bar u_{D_\eps}\|_{L^p(D_\eps;\omega^\eps)} \leq C \|\nabla u\|_{L^p(D_\eps;\omega^\eps)} \quad\text{for all } u\in W^{1,p}(D_\eps;\omega^\eps).
 \end{align*}
 Here $\bar u_{D_\eps}= \int_{D_\eps} u \d\omega^\eps/\|1\|_{L^1(D_\eps;\omega^\eps)}$ is the weighted mean value.
\end{corollary}
\begin{proof}
 Define $K_\eps=\{u\in W^{1,p}(D_\eps;\omega^\eps):\ \bar u_{D_\eps}=0\}$ and use Theorem~\ref{thm:poincare_general}.
\end{proof}

\begin{corollary}[Poincar\'e-Friedrichs-type inequality]\label{cor:friedrichs}
 Let $\eps\in [0,\eps_0]$, $1 \leq p<\infty$, and let $D_\eps$ be connected. Then there exists a constant $C$ independent of $\eps$ such that 
 for every $\eps\in (0,\eps_0)$ and $v\in W^{1,p}(D_\eps;\omega^\eps)$ there holds
 \begin{align*}
  \|v\|^p_{L^p(D_\eps;\omega^\eps)}\leq C_P \left( \|\nabla v\|^p_{L^p(D_\eps;\omega^\eps)} + \int_{D_\eps} |v|^p |\nabla \omega^\eps| \d x \right).
 \end{align*}
\end{corollary}
\begin{proof}
Define $K_\eps = \{ v \in W^{1,p}(D_\eps;\omega^\eps):\ \int_{D_\eps} |v|^p
|\nabla \omega^\eps| \d x=0\}$ and use Theorem~\ref{thm:poincare_general}.
\end{proof}
\begin{remark}
 In Corollary~\ref{cor:friedrichs} one can make the constant explicit if one applies the ``classical'' Poincar\'e-Friedrichs inequality \cite[6.26]{Adams1975} to $|v|^p \omega^\eps\in W^{1,1}_0(\Omega)$. 
\end{remark}
\begin{remark}
 Theorem~\ref{thm:poincare_general} also holds for the case $p=\infty$: By Rellich's theorem \cite{Adams1975} the embedding $W^{1,\infty}(D_{\tilde\eps-\gamma})\hookrightarrow C^{0,1}(\overline{D_{\tilde \eps-\gamma}})  \hookrightarrow L^\infty(D_{\tilde\eps-\gamma})$ is compact. Then, with similar arguments as above, the assumption \eqref{eq:help0} with $p=\infty$ leads to $\|u_k\|_{L^\infty(D_{\tilde\eps-\gamma})} \to 0$. Then, for $\tilde x= x+tn(x)\in D_{\eps_k}\setminus D_{\tilde\eps-\gamma}$ with $x\in \partial D_{\tilde\eps-\gamma}$ and $t\leq \gamma$, we obtain as $k\to \infty$
 \begin{align*}
  |u_k(\tilde x)| 
 \leq |u_k(x)| + \gamma \|\nabla u_k\|_{L^p(D_{\eps_k})} \leq |u_k(x)| + \gamma/k \to 0,
 \end{align*}
 where we have chosen a Lipschitz continuous representative of $u_k$.
 Hence, $\|u_k\|_{L^\infty(D_{\eps_k})} \to 0$ which contradicts $\|u_k\|_{L^\infty(D_{\eps_k})}=1$.
\end{remark}

\section{Convergence of diffuse integrals}\label{sec:conv_diffuse}

In the following two subsections we investigate the approximation properties of the diffuse integrals introduced in \eqref{eq:diffuse_integrals}.

\subsection{Convergence of diffuse volume integrals}\label{sec:conv_diffuse_volume}
We start with the case of volume integrals, for which we want to estimate the error
\begin{align*}
 E_V=\int_\Omega h(x) \d\omega^\eps(x) - \int_D h(x) \d x
\end{align*}
between the volume integral and the diffuse volume integral
in terms of $\eps$ and $h$. 
We will provide estimates for the cases $h\in L^p(D_\eps;\omega^\eps)$ and $h\in W^{1,p}(D_\eps;\omega^\eps)$ which gives stronger results improved by one order of $\eps$.

\begin{theorem}\label{thm:lp_diffuse_volume}
 Let $1< p\leq \infty$ and $h\in L^p(D_\eps;\omega^\eps)$. Then there exists a constant $C>0$ independent of $\eps$ such that
 \begin{align*}
  |E_V| \leq C \eps^{1-\frac{1}{p}} \|h\|_{L^p(\Gamma_\eps;\omega^\eps)}.
 \end{align*}
 Moreover, if $p=1$ and $h\in L^1(\Omega)$, then $E_V\to 0$ as $\eps\to 0^+$.
\end{theorem} 
\begin{proof}
\textit{$L^1$-regularity:}
Let $h\in L^1(\Omega)$. Using dominated convergence, we infer from $\omega^\eps(x)  \to \chi_D(x)$ as $\eps\to 0^+$ a.e. $x\in\Omega$ and $h \omega^\eps\leq  h$ that
\begin{align*}
 \lim_{\eps\to 0^+} E_V = 0.
\end{align*}
%
Let $h\in L^p(D_\eps;\omega^\eps)$ for fixed but arbitrary $1<p\leq \infty$. 
Using \eqref{eq:def_gamma_eps} and $\omega^\eps=1$ on $D\setminus\Gamma_\eps$, we obtain the representation
\begin{align*}
 E_V =\int_{\Gamma_\eps} h(x) \d\omega^\eps(x) - \int_{D\cap \Gamma_\eps} h(x) \d x.
\end{align*}
Using H\"olders inequality and $1\leq 2\omega^\eps$ on $D\cap\Gamma_\eps$ we can estimate the two terms as follows
\begin{align*}
 \int_{\Gamma_\eps} h(x) \d\omega^\eps(x) &\leq \|h\|_{L^p(\Gamma_\eps;\omega^\eps)} \|\omega^\eps\|_{L^1(\Gamma_\eps)}^{1-\frac{1}{p}},\\
  \int_{D\cap \Gamma_\eps} |h(x)| \d x &\leq 2\int_{D\cap \Gamma_\eps} |h(x)| \omega^\eps(x)\d x\leq  2 \|h\|_{L^p(\Gamma_\eps;\omega^\eps)} \|\omega^\eps\|_{L^1(D\cap \Gamma_\eps)}^{1-\frac{1}{p}}.
\end{align*}
Since $|\omega^\eps|\leq 1$, we deduce from \eqref{eq:boundary_measure} that
\begin{align*}
 \|\omega^\eps\|_{L^1(\Gamma_\eps)}^{1-\frac{1}{p}} \leq C\eps^{1-\frac{1}{p}},
\end{align*}
which concludes the proof.
\end{proof}

Theorem~\ref{thm:lp_diffuse_volume} for $L^p$-functions relies basically on the fact that $|\Gamma_\eps|\leq C\eps$. 
This is due to the fact that $L^p(D_\eps;\omega^\eps)$-functions can have singularities in $\Gamma_\eps$.
Note that in the case $p=1$ we expect no rate of convergence in terms of $\eps$, and the assumption $h\in L^1(\Omega)$ is stronger than those for $p>1$.
Resorting to $W^{1,p}$-functions we can exploit extensively symmetry of the phase-field function $\omega^\eps$ leading to a much stronger result.

\begin{theorem}\label{thm:sobolev_diffuse_volume}
Let $0<\eps\leq \eps_0$, and 
let $h\in W^{1,p}(D_\eps;\omega^\eps)$ for some $1\leq p\leq \infty$.
Then there exists $C>0$ independent of $\eps$ such that
\begin{align*}
 |E_V|\leq C \eps^{2-\frac{1}{p}} \| h\|_{W^{1,p}(\Gamma_\eps;\omega^\eps)}.
\end{align*}
\end{theorem}
\begin{proof}
Using a change of variables $s=S(-t/\eps)$ and Fubini's theorem, we observe that
\begin{align*}
  \int_{-\eps}^\eps \frac{1}{2\eps} S'(-\frac{t}{\eps}) \int_{\{d_D(x)<t\}} h(x)\d x \d t = \int_\Omega h(x) \frac{1+\varphi^\eps}{2} \d x.
\end{align*}
Since $\int_{-\eps}^\eps \frac{1}{2\eps} S'(-\frac{t}{\eps})\d t =1$, we further obtain
\begin{align*}
 E_V=\int_{-\eps}^\eps \frac{1}{2\eps} S'(-\frac{t}{\eps})\left( \int_{\{d_D(x)<t\}} h(x)\d x-\int_{\{d_D(x)<0\}} h(x)\d x\right) \d t.
\end{align*}
Observing that
\begin{align*}
  \int_{\{d_D(x)<t\}} h(x)\d x-\int_{\{d_D(x)<0\}} h(x)\d x &=  -\int_{\{t<d_D(x)<0\}} h(x)\d x\quad\text{for } t<0 \text{ and}\\
  \int_{\{d_D(x)<t\}} h(x)\d x-\int_{\{d_D(x)<0\}} h(x)\d x &= \phantom{-} \int_{\{0<d_D(x)<t\}} h(x)\d x\quad\text{for } t>0,
\end{align*}
and splitting the integration over $(-\eps,\eps)$ to $(-\eps,0)$ and $(0,\eps)$ and employing a change of variables $t\mapsto -t$ for the integral over $(-\eps,0)$, we further obtain
\begin{align}\label{eq:representation_EV}
 E_V=\int_{0}^\eps \frac{1}{2\eps} S'(-\frac{t}{\eps}) \left(\int_{\{0<d_D(x)<t\}} h(x)\d x-\int_{\{-t<d_D(x)<0\}} h(x)\d x\right) \d t.
\end{align}
For the last computation, we used $S(-t)=-S(t)$, i.e. $S'(-t)=S'(t)$.
To compare the difference on the right-hand side of the latter equation we use
the transformations $\Phi_s$ introduced in Section~\ref{sec:boundary_transformation}, and the transformation formula, namely
\begin{align*}
 &\phantom{=}\int_{\{0<d_D(x)<t\}} h(x)\d x-\int_{\{-t<d_D(x)<0\}} h(x)\d x\\
 &= \int_0^t \int_{\partial D} \big(h(x+s n(x))-h(x-sn(x)) \big) |\det D\Phi_s(x)|  \d\sigma(x)\d s\\
 &\phantom{=}+\int_0^t \int_{\partial D} h(x-sn(x)) (|\det D\Phi_s(x)|-|\det D\Phi_{-s}(x)|)  \d\sigma(x)\d s.
\end{align*}
The two integrals can be treated separately.
%
%
%
Using $h(x+s n(x))-h(x-sn(x))=  \int_{-s}^s \nabla h(x+\tau n(x))\cdot n(x)d\tau$, and $\frac{1}{2}\leq  |\det D\Phi_s(x)| \leq 2$ we obtain using Fubini's theorem and the transformation formula
\begin{align*}
 &\phantom{=}\int_0^t \int_{\partial D} \big(h(x+s n(x))-h(x-sn(x)) \big) |\det D\Phi_s(x)|  \d\sigma(x)\d s\\
 &\leq 2 \int_0^t \int_{-s}^s  \int_{\partial D} |\nabla h(x+\tau n(x))\cdot n(x)|\d\sigma(x) d\tau  \d s\\
 &\leq 4 t  \int_{-t}^t \int_{\partial D}  |\nabla h(x+\tau n(x))|  |\det D\Phi_\tau(x)|\d\sigma(x) d\tau \\
 &= 4 t  \int_{\Gamma_t} |\nabla h(x)|\d x \d t.
\end{align*}
For the second integral we obtain
\begin{align}\label{eq:EV_claim2}
  \int_0^t \int_{\partial D} h(x-sn(x)) (|\det D\Phi_s(x)|-|\det D\Phi_{-s}(x)|)  \d\sigma(x)\d s \leq C t \int_{\Gamma_t}|h(x)|\d x
\end{align}
which can be seeen with 
\eqref{eq:det_estimate} as
\begin{align*}
 &\phantom{\leq}\int_0^t \int_{\partial D} h(x-sn(x)) (|\det D\Phi_s(x)|-|\det D\Phi_{-s}(x)|)  \d\sigma(x)\d s\\
 &\leq C \|D^2 d_D\|_{L^\infty(\partial D)} \int_0^t  s\int_{\partial D} |h(x-sn(x))|  \d\sigma(x)\d s\\
 &\leq 2 C \|D^2 d_D\|_{L^\infty(\partial D)} t \int_0^t  \int_{\partial D} |h(x-sn(x))|  |\det D\Phi_{-s}(x)|\d\sigma(x)\d s\\
 &\leq 2 C \|D^2 d_D\|_{L^\infty(\partial D)} t \int_{\Gamma_t}|h(x)|\d x.
\end{align*}

Using these estimates we obtain from \eqref{eq:representation_EV}
\begin{align*}
 |E_V| \leq C \int_{0}^\eps \frac{1}{2\eps} S'(-\frac{t}{\eps}) t \int_{\Gamma_t} |h(x)| + |\nabla h(x)|\d x \d t.
\end{align*}
Setting $p'=p/(p-1)$, an application of H\"older's inequality thus yields
\begin{align*}
|E_V| \leq \frac{C}{2\eps} \left(\int_{0}^\eps  S'(-\frac{t}{\eps}) t^{p'} \d t\right)^\frac{1}{p'} \left(\int_0^\eps S'(-\frac{t}{\eps}) \left(\int_{\Gamma_t} |h(x)|+ |\nabla h(x)|\d x\right)^p \d t\right)^{\frac{1}{p}}.
\end{align*}
Using boundedness of $S'$ the first integral can be computed explicitly.
To treat the second integral, we use H\"older's inequality for the inner integral which gives
\begin{align*}
 |E_V| \leq \frac{C}{\eps}\eps^{2-\frac{1}{p}}\left(\int_0^\eps S'(-\frac{t}{\eps}) |\Gamma_t|^{\frac{p}{p'}}\int_{\Gamma_t} |h(x)|^p+ |\nabla h(x)|^p \d x \d t\right)^{\frac{1}{p}}.
\end{align*}
Note, that $C$ is a universal constant depending only on $S$, $D$ and $p$ but not on $\eps$ or $h$ which may change from to line.
Using $|\Gamma_t|\leq C t |\partial D|$, $t\leq \eps$ and $\frac{1}{p}+\frac{1}{p'}=1$, we therefore have
\begin{align*}
 |E_V| \leq C \eps^{2-\frac{1}{p}}\left(\frac{1}{\eps}\int_0^\eps S'(-\frac{t}{\eps}) \int_{\Gamma_t} |h(x)|^p+ |\nabla h(x)|^p \d x \d t\right)^{\frac{1}{p}}.
\end{align*}
Since $\Gamma_t=\{x\in D_\eps:\ -t<d_D(x)<t\} = \{x\in D_\eps: -s<\varphi^\eps(x)<s\}$ for $s=-S(-\frac{t}{\eps})$, a corresponding transformation yields
\begin{align*}
 &\frac{1}{\eps}\int_0^\eps S'(-\frac{t}{\eps}) \int_{\Gamma_t} |h(x)|^p+|\nabla h(x)|^p \d x \d t = \int_0^1 \int_{\{-s<\varphi^\eps<s\}} |h(x)|^p+|\nabla h(x)|^p \d x \d s\\
 &=\int_{\Gamma_\eps} \int_{|\varphi^\eps(x)|}^1 \d s \big( |h(x)|^p+ |\nabla h(x)|^p \big)\d x \leq 2\| h\|_{W^{1,p}(\Gamma_\eps;\omega^\eps)}^p
\end{align*}
where we used that $1-|\varphi^\eps|\leq 2 \omega^\eps$ on $\Gamma_\eps$, and
\begin{align*}
 \{(x,s)\in \RR^{n+1}: 0<s<1,\ -s<\varphi^\eps(x)<s\} = \{(x,s)\in\RR^{n+1}: |\varphi^\eps(x)|<s<1\}. 
\end{align*}
This yields the assertion.
\end{proof}

For sake of completeness, let us state a corresponding approximation result for H\"older continuous function, i.e. we say that $h\in C^{0,\nu}(\overline{\Omega})$, if $h$ is continuous on $\overline{\Omega}$ and if
\begin{align*}
 |h|_\nu = \sup_{x\neq y} \frac{h(x)-h(y)}{|x-y|^\nu} <\infty.
\end{align*}
We write $\| h\|_{C^{0,\nu}(\overline{\Omega})} = \sup_{x\in\Omega} |h(x)| + |h|_{\nu}$.

\begin{lemma}\label{lem:hoelder_diffuse_volume}
Let $0<\eps\leq \eps_0$, and 
let $h\in C^{0,\nu}(\overline{\Gamma_\eps})$ for some $0<\nu \leq 1$.
Then there exists $C>0$ independent of $\eps$ such that
\begin{align*}
 |E_V| \leq C \|h\|_{C^{0,\nu}(\overline{\Gamma_\eps})}\eps^{\nu+1}.
\end{align*}
\end{lemma}
\begin{proof}
It is easy to show the estimates
\begin{align*}
  \int_0^t \int_{\partial D} \big(h(x+s n(x))-h(x-sn(x)) \big) |\det D\Phi_s(x)|  \d\sigma(x)\d s \leq  C t^{\nu} \|h\|_{C^{0,\nu}(\overline{\Gamma_\eps})}
\end{align*}
and
\begin{align*}
  \int_0^t \int_{\partial D} h(x-sn(x)) (|\det D\Phi_s(x)|-|\det D\Phi_{-s}(x)|)  \d\sigma(x)\d s \leq C t \|h\|_{C^{0}(\overline{\Gamma_\eps})}.
\end{align*}
The proof is completed by integration over $t$ with similar arguments as in the proof of Theorem~\ref{thm:sobolev_diffuse_volume}.
\end{proof}

\subsection{Convergence of diffuse boundary integrals}\label{sec:conv_diffuse_boundary}
In this section we investigate the accuracy of the diffuse boundary integral approximation. For this sake consider
\begin{align*}
 E_B =\int_\Omega g(x) |\nabla \omega^\eps(x)| \d x - \int_{\partial D} g(x) \d\sigma(x).
\end{align*}
In the following we reduce the treatment of $E_B$ to that of $E_V$ from the previous section.
Using $\nabla d_D(x)=n(x)$ and the divergence theorem, we see that
\begin{align*}
 \int_{\partial D} g(x) \d\sigma(x) &= \int_{\partial D} g(x) \nabla d_D(x)\cdot n(x) \d\sigma(x)
 =\int_D {\rm div}(g(x) \nabla d_D(x)) \d x.
\end{align*}
Note that, $g_{\mid D} \in W^{1,p}(D)$ for any $g\in W^{1,p}(D_\eps;\omega^\eps)$, and thus $g$ has a trace on $\partial D$.
To treat the diffuse boundary integral, we first observe that $|\nabla \omega^\eps|=-\nabla d_D\cdot \nabla \omega^\eps$ on $\Gamma_\eps$. For the definition of $\Gamma_\eps$ see \eqref{eq:def_gamma_eps}.
Therefore, integration-by-parts shows that
\begin{align*}
 \int_\Omega g(x) |\nabla \omega^\eps(x)| \d x= -\int_\Omega g(x)\nabla d_D(x) \nabla \omega^\eps(x) \d x
 &=\int_\Omega {\rm div}(g(x)\nabla d_D(x)) \omega^\eps(x) \d x.
\end{align*}
Notice, that due to ${\rm supp}(\omega^\eps) \subset \Omega$ there are no boundary integrals.
Thus, we have that
\begin{align*}
 E_B = \int_\Omega {\rm div}(g(x)\nabla d_D(x)) \d\omega^\eps(x) - \int_D {\rm div}(g(x) \nabla d_D(x)) \d x.
\end{align*}
Setting $h={\rm div}(g\nabla d_D)$, we can use Theorem~\ref{thm:lp_diffuse_volume}, Theorem~\ref{thm:sobolev_diffuse_volume} and Lemma~\ref{lem:hoelder_diffuse_volume} of the previous section. 

\begin{lemma}\label{lem:lp_diffuse_boundary}
 Let $\partial D$ be of class $C^{1,1}$ and let $1\leq p\leq \infty$. Moreover, let $g\in W^{1,p}(D_\eps;\omega^\eps)$ for some $0<\eps<\eps_0$. Then there exists a constant $C>0$ independent of $\eps$ such that
 \begin{align*}
  |E_B| 
  \leq C\|g\|_{W^{1,p}(D_\eps;\omega^\eps)} \eps^{1-\frac{1}{p}}.
 \end{align*}
\end{lemma}
\begin{proof}
If $\partial D\in C^{1,1}$, then $d_D\in C^{1,1}$ \cite{DelfourZolesio2011} and, in this case, $g\in W^{1,p}(D_\eps;\omega^\eps)$ implies $h\in L^p(D_\eps;\omega^\eps)$ for $1\leq p\leq \infty$, which in turn implies $E_B=O(\eps^{1-1/p})$ by Theorem~\ref{thm:lp_diffuse_volume}.
\end{proof}

\begin{lemma}\label{lem:sobolev_diffuse_boundary}
 Let $\partial D$ be of class $C^{2,1}$. Moreover, let $g\in W^{2,p}(D_\eps;\omega^\eps)$ for some $0<\eps<\eps_0$ and $1\leq p\leq \infty$. Then there exists a constant $C>0$ independent of $\eps$ such that
 \begin{align*}
  |E_B|
\leq C\|g\|_{W^{2,p}(D_\eps;\omega^\eps)} \eps^{2-\frac{1}{p}}.
 \end{align*}
\end{lemma}
\begin{proof}
If $\partial D\in C^{2,1}$, then $d_D\in C^{2,1}$ \cite{DelfourZolesio2011} and, in this case, $g\in W^{2,p}(D_\eps;\omega^\eps)$ implies $h\in W^{1,p}(D_\eps;\omega^\eps)$, which in turn implies $E_B=O(\eps^{2-1/p})$ by Theorem~\ref{thm:sobolev_diffuse_volume}.
\end{proof}

The estimate of Lemma~\ref{lem:sobolev_diffuse_boundary} assumes $W^{2,p}$-regularity of the whole integrand. For our analysis we will also need a slightly different statement:

\begin{theorem}\label{thm:sobolev_diffuse_boundary2}
 Assume $\partial D$ is of class $C^{1,1}$, let (S3) hold and let $1\leq p\leq q \leq\infty$. Furthermore, let $u\in W^{2,q}(D_\eps;\omega^\eps)$ satisfy $u=0$ on $\partial D$ and let $v\in W^{1,p'}(D_\eps;\omega^\eps)$ with $p'=p/(p-1)$.
 Then there exists a constant $C$ independent of $\eps$, $u$ and $v$ such that for $q'=q/(q-1)$
 \begin{align*}
  \int_{\Gamma_\eps} uv |\nabla \omega^\eps| \d x\leq  C (\eps^{1+\frac{1}{q'}} \|u\|_{W^{2,q}(D_\eps;\omega^\eps)} +\eps^{1+\frac{1}{p}} \|u\|_{W^{2,p}(D_\eps;\omega^\eps)})  \|v\|_{W^{1,p'}(D_\eps;\omega^\eps)}.
 \end{align*}
\end{theorem}
The higher integrability of $u$ improves the first part of the estimate whereas the higher integrability of $v$ improves the second part. For $q=p$ the rate is $O(\eps^{1+\frac{1}{p'}}+\eps^{1+\frac{1}{p}})$ which is optimal for $p=2$. For $q=p'$ we obtain the best possible rate $O(\eps^{1+\frac{1}{p}})$.
\begin{proof}
We start with an inequality for $w\in W^{1,1}(D_\eps;\omega^\eps)$.
An application of \eqref{eq:boundary_transformation_formula}, \eqref{eq:representation_Jacobian} and Theorem~\ref{thm:trace} yields
\begin{align}
 \big| \int_{\Gamma_\eps} w |\nabla \omega^\eps| \d x &- \int_{-\eps}^\eps\frac{1}{2\eps}S'(-\frac{t}{\eps})\int_{\partial D} w(x+tn(x)) (1+t\Delta d_D(x))\d\sigma(x) \d t\big| \label{newclaim1} \\
 &\leq  C \int_{-\eps}^\eps\frac{1}{2\eps}S'(-\frac{t}{\eps})\int_{\partial D} |w(x+tn(x))| \eps^2 \d\sigma(x)\d t \nonumber \\
 &\leq  C\eps^2 \int_{\Gamma_\eps} |w| |\nabla\omega^\eps| \d x\nonumber \\
 &\leq   C\eps^2 \|w\|_{W^{1,1}(D_\eps;\omega^\eps)}, \nonumber 
\end{align}
with $C$ independent of $\eps$.
 
 
 Now let $w\in W^{2,q}((-\eps,\eps))$, then there exists a constant C such that
\begin{align}
 |\int_{-\eps}^\eps S'(-\frac{t}{\eps})\int_0^t w'(s) \d s \d t | \leq \eps^{3-\frac{3}{q}} C  \left(\int_0^\eps S'(-\frac{t}{\eps})\int_0^t \int_{-s}^s |w''(r)|^q \d r\d s \d t\right)^{\frac{1}{q}}.  \label{newclaim2}
\end{align}
This can be seen as follows: 
By change of variables and application of the fundamental theorem of calculus, we obtain
\begin{align*}
 \int_{-\eps}^\eps S'(-\frac{t}{\eps})\int_0^t w'(s) \d s \d t &= \int_{0}^\eps S'(-\frac{t}{\eps}) \int_0^t w'(s)-w'(-s) \d s \d t\\
 &=  \int_{0}^\eps S'(-\frac{t}{\eps})\int_0^t \int_{-s}^s w''(r) \d r \d s \d t.
\end{align*}
Repeated application of H\"older's inequality gives
\begin{align*}
 &|  \int_{0}^\eps S'(-\frac{t}{\eps})\int_0^t \int_{-s}^s w''(r) \d r \d s \d t| \\
 &\leq 2^{\frac{1}{q'}}  \left(\int_{0}^\eps S'(-\frac{t}{\eps})\int_0^t s \d s\d t\right)^{\frac{1}{q'}} \left(\int_0^\eps S'(-\frac{t}{\eps})\int_0^t\int_{-s}^s |w''(r)|^q \d r \d s\d t\right)^{\frac{1}{q}}.
\end{align*}
Using boundedness of $S'$ and calculating $\int_{0}^\eps \int_0^t s \d s\d t=\eps^3/6$ yields the assertion.

  
We are now in the position to give a proof of the theorem.
By setting $w=uv$ in \eqref{newclaim1}, we have that
 \begin{align*}
  \int_{\Gamma_\eps} uv |\nabla \omega^\eps| \d x- \int_{-\eps}^\eps\frac{1}{2\eps}S'(-\frac{t}{\eps}) \int_{\partial D}  u(x+tn(x))v(x+tn(x)) (1+t\Delta d_D(x)) \d\sigma(x)\d t\\
  \leq   C\eps^2 \|u\|_{W^{1,p}(D_\eps;\omega^\eps)}\|v\|_{W^{1,p'}(D_\eps;\omega^\eps)}.
 \end{align*}

Thus, to prove the theorem, it is sufficient to estimate the second integral on the left-hand side of the latter inequality.
Using $v(x+tn(x))=v(x) + \int_0^t \nabla v(x+sn(x))\cdot n(x) \d s$ and $u(x+tn(x))= \int_0^t \nabla u(x+sn(x))\cdot n(x) \d s$, we see that
\begin{align*}
  u(x+tn(x))v(x+tn(x))&= v(x)\int_0^t \nabla u(x+sn(x))\cdot n(x) \d s \\
                      &+ \int_0^t \nabla v(x+sn(x))\cdot n(x) \d s\int_0^t \nabla u(x+sn(x))\cdot n(x) \d s.
\end{align*}
We treat the two terms on the right-hand side separately.
For the first one, we will use \eqref{newclaim2} with $w(s)= u(x+sn(x))$, H\"older's inequality and $q\geq p$ which yields
\begin{align*}
 &\left| \frac{1}{2\eps} \int_{\partial D}  v(x)\int_{-\eps}^\eps S'(-\frac{t}{\eps}) \int_0^t \nabla u(x+sn(x))\cdot n(x) \d s \d t \d\sigma(x)\right| \\
 &\leq C\frac{ \eps^{3-\frac{3}{q}} }{2\eps} \int_{\partial D}  v(x) \left( \int_0^\eps S'(-\frac{t}{\eps})\int_0^t \int_{-s}^s |n(x)\cdot D^2 u(x+r n(x))\cdot n(x)|^q dr\d s\d t\right)^{\frac{1}{q}} \d\sigma(x) \\
 &\leq  C \eps^{2-\frac{3}{q}} \left(\int_{\partial D}  |v|^{p'}\d\sigma)\right)^{\frac{1}{p'}} \left(  \int_0^\eps S'(-\frac{t}{\eps})\int_0^t \int_{\Gamma_s} |D^2 u|^q \d x\d s\d t \right)^{\frac{1}{q}}\\
 &\leq  C \eps^{2-\frac{1}{q}} \left(\int_{\partial D}  |v|^{p'}\d\sigma)\right)^{\frac{1}{p'}} \left(  \int_0^\eps \frac{1}{2\eps} S'(-\frac{t}{\eps}) \int_{\Gamma_t} |D^2 u|^q \d x\d t \right)^{\frac{1}{q}}\\
 &\leq  C\eps^{2-\frac{1}{q} }  \|v\|_{W^{1,p'}(D_\eps;\omega^\eps)} \|u\|_{W^{2,q}(D_\eps;\omega^\eps)},
\end{align*}
where we have used Lemma~\ref{thm:trace_unweighted} to treat the term involving $v$ and the transformation formula to treat the term involving $u$, see the last lines of the proof of Theorem~\ref{thm:sobolev_diffuse_volume}.
For the second term we
first use H\"older's inequality twice
\begin{align*}
 &\big| \frac{1}{2\eps} \int_{\partial D} \int_{0}^\eps S'(-\frac{t}{\eps}) \left(\int_0^t \nabla v(x+sn(x))\cdot n(x)\d s\right) \left(\int_0^t \nabla u(x+sn(x))\cdot n(x) \d s\right) \d t\d\sigma(x)\big| \\
 &\leq \frac{1}{2\eps} \int_{\partial D} \int_{0}^\eps t S'(-\frac{t}{\eps}) \left(\int_0^t |\nabla v(x+sn(x))|^{p'}\d s\right)^{\frac{1}{p'}} \left(\int_0^t |\nabla u(x+sn(x))|^p \d s\right)^{\frac{1}{p}} \d t\d\sigma(x) \\
 &\leq \frac{\eps}{2} \left( \int_{\partial D}  \frac{1}{\eps}\int_{0}^\eps S'(-\frac{t}{\eps})\int_0^t |\nabla v(x+sn(x))|^{p'} \d s\d t\d\sigma(x)\right)^{\frac{1}{p'}}\\
 & \quad\qquad
 \left( \int_{\partial D}  \frac{1}{\eps} \int_{0}^\eps S'(-\frac{t}{\eps})\int_0^t |\nabla u(x+sn(x))|^{p} \d s  \d t\d\sigma(x)\right)^{\frac{1}{p}}.
\end{align*}
Then, using \eqref{eq:boundary_transformation_formula}, we obtain similarly as in the proof of Theorem~\ref{thm:sobolev_diffuse_volume}
\begin{align*}
  \int_{\partial D}  \frac{1}{\eps} \int_{0}^\eps S'(-\frac{t}{\eps})\int_0^t |\nabla v(x+sn(x))|^{p'} \d s\d t\d\sigma(x) 
\leq C \|v\|_{W^{1,p'}(D_\eps;\omega^\eps)}^{p'},
\end{align*}
and by \eqref{eq:boundary_transformation_formula}, (S3), and by Theorem~\ref{thm:trace}
\begin{align*}
 &       \int_{\partial D} \frac{1}{\eps} \int_{0}^\eps S'(-\frac{t}{\eps}) \int_0^t |\nabla u(x+sn(x))|^{p} \d s  \d t\d\sigma(x)\\
 &\leq   \eps \int_{\partial D} \int_0^\eps \frac{1}{\eps} S'(-\frac{s}{\eps}) |\nabla u(x+sn(x))|^{p} \d s  \d\sigma(x)\\
 &\leq C \eps \int_{\Gamma_{\eps}}|\nabla u|^{p} |\nabla \omega^\eps| \d x\\
 &\leq C \eps \|u\|_{W^{2,p}(D_\eps;\omega^\eps)}^p.
\end{align*}
The integrals over $(-\eps,0)$ as well as the ones involving $t\Delta d_D$ can be treated similarly.
Collecting all terms yields the assertion.
\end{proof}

Up to now, we have always assumed the boundary data $g$ to be regular. For completeness, let us also consider the case $g\in L^p(\partial D)$ only. Then $g$ is defined a.e. on $\partial D$, and we can define an extension a.e. on $\Gamma_\eps$ by
\begin{align}\label{eq:extension}
 \tilde g(x+tn(x)) = g(x),\qquad -\eps\leq t \leq \eps,\ x\in\partial D.
\end{align}
\begin{lemma}\label{lem:diffuse_boundary_lp}
 Let $g\in L^p(\partial \Omega)$ and let $v\in W^{1,p'}(D_\eps;\omega^\eps)$ with $1\leq p\leq\infty$, $p'=p/(p-1)$, and $0<\eps\leq\eps_0$. Then there exists a constant $C$ independent of $\eps$ such that
 \begin{align*}
  |\int_{\Gamma_\eps} \tilde g v |\nabla\omega^\eps|\d x - \int_{\partial D} gv \d\sigma|\leq C\eps^{1/p}\|g\|_{L^p(\partial D)} \|v\|_{W^{1,p'}(D_\eps;\omega^\eps)}
 \end{align*}
 with $\tilde g$ being the extension defined in \eqref{eq:extension}.
\end{lemma}
\begin{proof}
Using $\int_{-\eps}^\eps S'(-t/\eps)\d t=2\eps$ and the transformation formula, we obtain 
\begin{align*}
  &\int_{\Gamma_\eps} \tilde g v |\nabla\omega^\eps|\d x - \int_{\partial D} gv \d\sigma \\
  &= \int_{-\eps}^\eps \frac{1}{2\eps}S'(-\frac{t}{\eps})\int_{\partial D} \big(v(x+tn(x)) \det D\Phi_t(x) - v(x)\big) g(x) \d\sigma(x)\d t.
\end{align*}
Thus, using \eqref{eq:representation_Jacobian}, there exists $C>0$ independent of $\eps$ such that
\begin{align*}
  \left|\int_{\Gamma_\eps} \tilde g v |\nabla\omega^\eps|\d x - \int_{\partial D} gv \d\sigma \right| \leq \left| \int_{\partial D} g(x) \int_{-\eps}^\eps \frac{1}{2\eps}S'(-\frac{t}{\eps}) \big(v(x+tn(x)) - v(x)\big)\d t \d\sigma(x) \right|  \\
  + C \left|\int_{\partial D} g(x) \int_{-\eps}^\eps \frac{1}{2\eps}S'(-\frac{t}{\eps}) tv(x+tn(x)) \d t \d\sigma(x) \right|.
\end{align*}
We treat the two integrals on the right-hand side separately.
Repeated use of H\"older's inequality and $\nabla v(x+tn(x))-v(x) =\int_0^t \nabla v(x+sn(x))\cdot n(x)\d s$ yields similarly as in the proof of Theorem~\ref{thm:sobolev_diffuse_volume}
\begin{align*}
 &|\int_{\partial D} g(x) \int_0^\eps \frac{1}{2\eps}S'(-\frac{t}{\eps}) \big(v(x+tn(x)) - v(x)\big) \d t \d\sigma(x)|\\
 &\leq \|g\|_{L^p(\partial D)} \left(\int_{\partial D} \left(\int_0^\eps \frac{1}{2\eps}S'(-\frac{t}{\eps})  \int_0^t |\nabla v(x+sn(x))|d s \d t\right)^{p'} \d\sigma(x)\right)^{\frac{1}{p'}}\\
  &\leq \|g\|_{L^p(\partial D)} |\Gamma_\eps|^{\frac{1}{p}}\left(\int_0^\eps \frac{1}{2\eps}S'(-\frac{t}{\eps}) \int_{\Gamma_t} |\nabla v(x)|^{p'} \d x \d t\right)^{p'}. 
\end{align*}
An analogue estimate hold for the integral over $(-\eps,0)$.
Since, $|\Gamma_\eps|^{\frac{1}{p}}\leq C \eps^{1/p}$ this is the desired estimate for the first integral. The second can be estimated similarly, i.e.
\begin{align*}
 \left|\int_{\partial D} g(x) \int_{-\eps}^\eps \frac{1}{2\eps}S'(-\frac{t}{\eps}) tv(x+tn(x)) \d t \d\sigma(x) \right|
 \leq \eps \|g\|_{L^p(\partial D)} \left( \int_{\Gamma_{\eps}} |v|^{p'} |\nabla \omega^\eps| \d x \right)^{\frac{1}{p'}}
\end{align*}
The last term can be estimated using the trace theorem~\ref{thm:trace}.
\end{proof}

For the sake of completeness, we also state an analog to Lemma~\ref{lem:hoelder_diffuse_volume}.
\begin{lemma}\label{lem:hoelder_diffuse_boundary}
 Let $\partial D$ be of class $C^{2,\nu}$ for some $0<\nu \leq 1$. Moreover, let $g\in C^{1,\nu}(\overline{\Gamma_\eps})$ for some $0<\eps<\eps_0$. Then there exists a constant $C>0$ independent of $\eps$ such that
 \begin{align*}
  |E_B| 
  \leq C\|g\|_{C^{1,\nu}(\overline{\Gamma_\eps})} \eps^{1+\nu}.
 \end{align*}
\end{lemma}
\begin{proof}
 Since $d_D\in C^{2,\nu}(\overline{\Gamma_\eps})$, we have $ g\nabla d_D\in C^{1,\nu}(\overline{\Gamma_\eps})$. The assertion follows from Lemma~\ref{lem:hoelder_diffuse_volume}.
\end{proof}

\section{Diffuse elliptic problems}\label{sec:elliptic_problems}

In this section we investigate three typical second order elliptic boundary value problems. We start with Robin-type problems, which build the basis for further investigations.
For rather irregular data, we obtain a weak sublinear convergence result in terms of $\eps$. Superlinear convergence is achieved by requiring smooth data.
In Section~\ref{sec:dirichlet} we treat Dirichlet boundary conditions which can be reduced to the analysis of a Robin problem by means of the well-known penalty method.
In Section~\ref{sec:neumann} we consider Neumann boundary conditions and establish well-posedness of the diffuse domain method. Apart from the well-posed the convergence results can be derived as in the Robin case.

\subsection{Robin boundary conditions}\label{sec:robin}

Consider the following second order elliptic equation with Robin-type boundary condition:
Find $u$ such that
\begin{align}
 -{\rm div}(A\nabla u) + cu = f &\quad\text{in } D,\label{eq:elliptic_pde}\\
  n\cdot A\nabla u + b u = g &\quad\text{on } \partial D.\label{eq:bc_robin}
\end{align}

In order to obtain (weak) solutions to \eqref{eq:elliptic_pde}--\eqref{eq:bc_robin}, let us consider the following weak formulation:
Find $u\in W^{1,2}(D)$ such that 
\begin{align}\label{eq:robin_var}
 a(u,v)=\ell(v)\quad\text{for all } v\in W^{1,2}(D),
\end{align}
with bilinear and linear form
\begin{align*}
 a(u,v)  = \int_D A\nabla u \cdot \nabla v + c u v \d x +\int_{\partial D} b uv \d\sigma,\qquad
 \ell(v) = \int_D fv \d x + \int_{\partial D} g v \d\sigma.
\end{align*}
In order to prove well-posedness of the weak form \eqref{eq:robin_var} via the Lax-Milgram lemma we make the following assumptions:
\begin{enumerate}
 \item[(C1)] $0<b_0\leq b \in W^{1,\infty}(\Omega)$, $0\leq c\in L^\infty(\Omega)$.
 \item[(C2)] $A\in L^{\infty}(\Omega)^{n\times n}$ is a symmetric positive definite matrix, i.e. there exists $\kappa>0$ such that for a.e. $x\in\Omega$
 \begin{align*}
  \kappa^{-1} |\xi|^2 \leq \xi\cdot A(x)\xi \leq \kappa |\xi|^2\quad\text{for all } \xi \in \RR^n.
 \end{align*}
\end{enumerate}
\begin{lemma}
 Let (C1)--(C2) hold. Moreover, let $f\in L^2(D)$ and $g\in W^{1,2}(D)$. Then there exists a unique $u\in W^{1,2}(D)$ satisfying \eqref{eq:robin_var}, and there exists $C>0$ such that
 \begin{align*}
  \|u\|_{W^{1,2}(D)}\leq C (\|f\|_{L^2(D)} +\|g\|_{L^2(\partial D)} ).
 \end{align*}
\end{lemma}

The diffuse approximation of \eqref{eq:robin_var} is now: Find $u^\eps\in W^{1,2}(D_\eps;\omega^\eps)$ such that 
\begin{align}\label{eq:robin_var_diffuse}
 a^\eps(u^\eps,v)=\ell^\eps(v) \quad\text{for all } v\in W^{1,2}(D_\eps;\omega^\eps),
\end{align}
where the corresponding bilinear and linear form are given by
\begin{align*}
 a^\eps(u^\eps,v)&=\int_\Omega A\nabla u^\eps \cdot \nabla v + c u^\eps v \d\omega^\eps +\int_{\Omega} b u^\eps v  |\nabla \omega^\eps| \d x\\
 \ell^\eps(v) &= \int_\Omega f v  \d\omega^\eps + \int_{\Omega} g v |\nabla\omega^\eps|\d x.
\end{align*}
\begin{lemma}
 Let (C1)--(C2) hold. Moreover, let $f\in L^2(D_\eps;\omega^\eps)$ and $g\in W^{1,2}(D_\eps;\omega^\eps)$. Then there exists a unique $u^\eps\in W^{1,2}(D_\eps;\omega^\eps)$ satisfying \eqref{eq:robin_var_diffuse}, and there exists $C>0$ independent of $\eps$ such that
 \begin{align*}
  \|u^\eps\|_{W^{1,2}(D_\eps;\omega^\eps)}\leq C (\|f\|_{L^2(D_\eps;\omega^\eps)} +\|g\|_{W^{1,2}(D_\eps;\omega^\eps)} ).
 \end{align*}
\end{lemma}
\begin{proof}
 Continuity of $a^\eps$ and $\ell^\eps$ with respect to the $W^{1,2}(D_\eps;\omega^\eps)$-topology follows from boundedness of the coefficients and Theorem~\ref{thm:trace}.
 Coercivity of $a^\eps$ on $W^{1,2}(D_\eps;\omega^\eps)$ is a direct consequence of the positivity of $A$ and the Poincar\'e-Friedrichs inequality, see Corollary~\ref{cor:friedrichs}. An application of the Lax-Milgram lemma yields the assertion.
\end{proof}

Denoting by $u$ and $u^\eps$ the corresponding solutions to \eqref{eq:robin_var} and \eqref{eq:robin_var_diffuse}, respectively, we next want to estimate the error $u-u^\eps$ with respect to the $W^{1,2}(D_\eps;\omega^\eps)$-norm which directly implies estimates in the $W^{1,2}(D)$-norm as well.
By regularity of $\partial D$, we can assume that $u:D\to\RR$ is extended to $\Omega$ preserving $W^{1,2}(\Omega)$-regularity.
Hence, the error $u-u^\eps$ satisfies
\begin{align}\label{eq:error}
 a^\eps(u-u^\eps,v)= a^\eps(u,v)-a(u,v)+\ell(v)-\ell^\eps(v)\quad \text{for all }v\in W^{1,2}(D_\eps;\omega^\eps). 
\end{align}

\subsubsection{Sublinear convergence}
In order to obtain a first estimate for the error $u-u^\eps$, we estimate the right-hand side of \eqref{eq:error} by employing the embedding theorem~\ref{thm:embdding}. We recall the definitions $p_\alpha^*=(n+\alpha)p/(n+\alpha-p)$, see \eqref{eq:conjugate_sobolev}, and
\begin{align*}
  \|\ell\|_{W^{1,2}(D_\eps;\omega^\eps)'} = \sup_{v\in W^{1,2}(D_\eps;\omega^\eps)} \frac{\ell(v)}{\|v\|_{W^{1,2}(D_\eps;\omega^\eps)}},
\end{align*}
which is the norm of $\ell$ as an element of the dual space of $W^{1,2}(D_\eps;\omega^\eps)$.

\begin{lemma}\label{lem:rhs1}
 Let $f\in L^2(D_\eps,\omega^\eps)$ and $g\in W^{1,2}(D_\eps;\omega^\eps)$. 
 Then there exists a constant $C$ independent of $\eps$ such that 
 \begin{align*}
  \|\ell^\eps-\ell\|_{W^{1,2}(D_\eps;\omega^\eps)'} \leq C\left(\|f\|_{L^2(D_\eps;\omega^\eps)}+\|g\|_{W^{1,2}(D_\eps;\omega^\eps)} \right) \eps^{\frac{1}{n+\alpha}}. 
 \end{align*}
\end{lemma}
\begin{proof}
 Let $v\in W^{1,2}(D_\eps;\omega^\eps)$. Due to the weighted Sobolev embedding \eqref{eq:weighted_sobolev} we have $v \in L^p(D_\eps,\omega^\eps)$ for $p = 2_\alpha^*$. Hence, $fv \in L^q(D_\eps,\omega^\eps)$ with $q=2p/(2+p)$ due to H\"older's inequality.
 Similarly, since $\nabla(gv)=g\nabla v + \nabla g v \in L^q(D_\eps,\omega^\eps)$ with $q$ as before, we have that $gv \in W^{1,q}(D_\eps,\omega^\eps)$.
 Using Theorem~\ref{thm:lp_diffuse_volume} and Lemma~\ref{lem:lp_diffuse_boundary} we obtain
 \begin{align*}
  |\ell^\eps(v)-\ell(v)| &\leq C\left(\|f\|_{L^2(D_\eps;\omega^\eps)}+\|g\|_{W^{1,2}(D_\eps;\omega^\eps)} \right)\|v\|_{W^{1,2}(D_\eps;\omega^\eps)} \eps^{1-\frac{1}{q}} 
 \end{align*}
 The assertion follows from $1-\frac{1}{q}=\frac{1}{2}-\frac{1}{2_\alpha^*}=\frac{1}{n+\alpha}$.
\end{proof}
In order to obtain convergence rates, we need some regularity of $u$.
\begin{lemma}\label{lem:rhs2}
 Let $u\in W^{1,p}(D_\eps;\omega^\eps)$ for some $p>2$. Then there exists a constant $C$ independent of $\eps$ such that
 \begin{align*}
  \|a^\eps(u,\cdot)-a(u,\cdot)\|_{W^{1,2}(D_\eps;\omega^\eps)'}  \leq C \|u\|_{ W^{1,p}(D_\eps;\omega^\eps) } \eps^{\frac{1}{2}-\frac{1}{p}}.
 \end{align*}
\end{lemma}
\begin{proof}
 Let $v\in W^{1,2}(D_\eps;\omega^\eps)$ be arbitrary. 
 Since $A$ is bounded and $u\in W^{1,p}(D_\eps;\omega^\eps)$, it is $A\nabla u\cdot \nabla v \in L^q(D_\eps;\omega^\eps)$ with $q=2p/(2+p)$. 
 Similarly $cuv \in L^q(D_\eps;\omega^\eps)$. Using $\alpha \in W^{1,\infty}(\Omega)$, we see that $\alpha u v \in W^{1,q}(D_\eps;\omega^\eps)$ with $q$ as before. 
 The result now follows by applying Theorem~\ref{thm:lp_diffuse_volume} and Lemma~\ref{lem:lp_diffuse_boundary} similar as in the proof of Lemma~\ref{lem:rhs1}.
\end{proof}

Having estimated the errors in right-hand side and bilinear form we can proceed to the main approximation results in this section:
\begin{theorem}\label{thm:rate_lp_robin}
 Let (C1)--(C2) hold. Moreover, assume that $u\in W^{1,p}(D)$ with $2\leq p \leq 2_\alpha^*$ is a solution to \eqref{eq:robin_var} and $u^\eps\in W^{1,2}(D_\eps;\omega^\eps)$ is a solution to \eqref{eq:robin_var_diffuse}. Then there exists a constant $C>0$ independent of $\eps$ such that
 \begin{align*}
  \|u-u^\eps\|_{W^{1,2}(D_\eps;\omega^\eps)} \leq C\left(\|u\|_{ W^{1,p}(D_\eps;\omega^\eps)}+\|f\|_{L^2(D_\eps;\omega^\eps)}+\|g\|_{W^{1,2}(D_\eps;\omega^\eps)} \right)\eps^{\frac{1}{2}-\frac{1}{p}}.
 \end{align*}
\end{theorem}
\begin{proof}
 Coercivity of $a^\eps$ and \eqref{eq:error} imply
 \begin{align*}
   \|u-u^\eps\|_{W^{1,2}(D_\eps;\omega^\eps)} \leq C (\|a^\eps(u,\cdot)-a(u,\cdot)\|_{W^{1,2}(D_\eps;\omega^\eps)'}  + \|\ell^\eps-\ell\|_{W^{1,2}(D_\eps;\omega^\eps)'}),
 \end{align*}
 and the assertion follows from Lemma~\ref{lem:rhs1} and Lemma~\ref{lem:rhs2}.
\end{proof}

\begin{remark}
 Note that, according to \cite{Groeger1989}, see also \cite{EggerSchlottbom2010}, 
 there always exists a $p>2$ such that $u\in W^{1,p}(D)$, whence $u\in W^{1,p}(D_\eps;\omega^\eps)$ by extension. 
 If $p=2_\alpha^*$, we obtain the best possible rate $O(\eps^{1/(n+\alpha)})$, i.e. $O(\eps^{1/3})$ in two space dimensions and $S$ as in Example~\ref{ex:S} (i).
\end{remark}
\begin{remark}
 (i) Assuming $g=0$ and $f\in L^2(D)$ extended by zero to $\Omega$ an inspection of the proof of Theorem~\ref{thm:lp_diffuse_volume} shows that for each $v\in W^{1,2}(D_\eps;\omega^\eps)$
 \begin{align*}
  \int_{D_\eps} fv \d\omega^\eps - \int_{D} fv \d x \leq C\eps^{\frac{1}{n}}\|v\|_{W^{1,2}(D)}\|f\|_{L^2(D)}
 \end{align*}
 which is due to the embedding $W^{1,2}(D)\hookrightarrow L^{2^*_0}(D)$ and the fact that $v_{\mid D} \in W^{1,2}(D)$. This immediately leads to a stronger result in Lemma~\ref{lem:rhs1} independent of $\omega^\eps$.
 However, for proving Lemma~\ref{lem:rhs2}, we have to estimate the term
 \begin{align*}
  \int_{D_\eps} A\nabla u\cdot\nabla v \d\omega^\eps - \int_{D} A\nabla u\cdot\nabla v \d x.
 \end{align*}
 Here, on the one hand, to preserve regularity of $u$, setting $u=0$ on $D_\eps\setminus D$ is not possible. On the other hand setting $A=0$ on $D_\eps\setminus D$ is not allowed since then \eqref{eq:ddm1} is not well-posed anymore.

 (ii) If $f\in L^\infty(\Omega)$, $g\in W^{1,\infty}(\Omega)$ and $u\in W^{1,\infty}(\Omega)$, then using the techniques from above, we would obtain the bound
 \begin{align*}
  \|u-u^\eps\|_{W^{1,2}(D_\eps;\omega^\eps)} = O(\eps^{\frac{1}{2}})
 \end{align*}
 as $\eps \to 0$ since the test function $v$ is merely $W^{1,2}(D_\eps;\omega^\eps)$.
 For smooth test functions $v$ and $f\in L^p(D_\eps;\omega^\eps)$ and $u, g\in W^{1,p}(D_\eps;\omega^\eps)$, the right-hand side of \eqref{eq:error} is bounded by
 a constant multiple (depending on $f$, $g$ and $u$) of $\eps^{1-1/p} \|v\|_{W^{1,\infty}(D_\eps)}$.
 This estimate, however, does not lead to $W^{1,2}(D_\eps;\omega^\eps)$-estimates for the error anymore. We will return to these type of estimates in the next section. We also mention that an inspection of several proofs above shows that crucial terms drop out if the involved functions are symmetric with respect to $\partial D$ (mirrored along the normal direction). Thus, using symmetric extensions of data and solutions as well as a restriction to a Sobolev space of functions symmetric with respect to $\partial D$ could give higher order rates. However, since this does not correspond to the computational practice and would extremely complicate the numerical solution, this seems not of particular practical relevance and hence we do not pursue this direction further.
\end{remark}

\subsubsection{From linear to quadratic convergence}
In literature there exist very recent formal results for the diffuse domain method that state a rate of convergence for the $L^2$-norm of $O(\eps^2)$ for the Poisson equation with Robin boundary conditions \cite{LervagLowengrub2014}. To give a precise and rigorous statement of such a better rate we need additional regularity of the domain, the data and the solutions.
Furthermore, we resort also to other functions spaces.
For a smooth function $v$ and $1\leq p\leq \infty$ we let $p'=p/(p-1)$ and define 
\begin{align*}
 \|v\|_{\X_p^\eps} = \|a^\eps(v,\cdot)\|_{W^{1,p'}(D_\eps;\omega^\eps)'}=\sup\{a^\eps(v,\phi):\ {\phi \in C^\infty(\overline{D_\eps}), \|\phi\|_{W^{1,p'}(D_\eps;\omega^\eps)}\leq 1} \}.
\end{align*}
We let $\X_p^\eps=\{v\in C^{\infty}(\overline{D_\eps}):\ \|v\|_{\X_p^\eps}<\infty\}$ denote the completion of $C^\infty(\overline{D_\eps})$ with respect to $\|\cdot\|_{\X_p^\eps}$. Then $(\X_p^\eps,\|\cdot\|_{\X_p^\eps})$ is a Banach space.
Due to the Riesz representation theorem, Corollary~\ref{cor:friedrichs}, and assumptions (C1)-(C2) on the coefficients, we see that $\X_2^\eps=W^{1,2}(D_\eps;\omega^\eps)$ with equivalent norms.
Furthermore, due to Theorem~\ref{thm:trace} we easily see that $\|u\|_{\X_p^\eps} \leq C \|u\|_{W^{1,p}(D_\eps;\omega^\eps)}$. 
For the other direction,
we need a solvability result. If for any $\ell \in W^{1,p'}(D_\eps;\omega^\eps)'$ there exists $u\in W^{1,p}(D_\eps;\omega^\eps)$ such that $a^\eps(u,v)=\ell(v)$ for all $v\in W^{1,p'}(D_\eps;\omega^\eps)$ and $\|u\|_{W^{1,p}(D_\eps;\omega^\eps)}\leq C \|\ell\|_{W^{1,p'}(D_\eps;\omega^\eps)'}$ for some constant $C$, then $\|u\|_{W^{1,p}(D_\eps;\omega^\eps)} \leq \|u\|_{\X_p^\eps}$.
Let us emphasize that such a result is not known to us for the case $p\neq 2$; we refer to \cite{EggerSchlottbom2010,Groeger1989} for a corresponding result in the unweighted case.

Let us thus start with an error estimate in $\X_p^\eps$. For simplicity, we will assume smooth data. It should become clear from the proof how to lower these regularity assumptions.

\begin{theorem}\label{thm:rate_sobolev_robin_p}
 Let $\partial D$ be of class $C^{\infty}$, and let $f,g\in C^\infty(\overline{\Omega})$ and let (C1)--(C2) hold. 
 Moreover, let $A\in C^{\infty}(\overline{\Omega})^{3\times 3}$, $c\in C^{\infty}(\overline{\Omega})$, $b\in C^{\infty}(\overline{\Omega})$, and let $u^\eps \in W^{1,2}(D_\eps;\omega^\eps)$ denote the solution to \eqref{eq:robin_var_diffuse}, and let $u\in W^{1,2}(D)$ denote the solution to \eqref{eq:robin_var}.
 Then, for $1\leq p\leq \infty$ there exists a constant $C$ independent of $\eps$ such that
 \begin{align*}
  \|u-u^\eps\|_{\X_p^\eps} \leq C\eps^{1+\frac{1}{p}}.
 \end{align*}
\end{theorem}
\begin{proof}
Due to \cite[Thm.~2.4.2.7, Rem.~2.5.1.2]{Grisvard85} and the smoothness of the data, we have that $u\in W^{k,2}(D)$ for any $k\in \NN$, i.e. $u\in C^{\infty}(\overline{D})$ by embedding, and $u$ is a classical solution to \eqref{eq:elliptic_pde}--\eqref{eq:bc_robin}.
Therefore, integrating by parts on the right hand side of \eqref{eq:error}, we deduce that for any $v\in W^{1,p'}(D_\eps;\omega^\eps)$ the error satisfies
\begin{align*}
 a^\eps(u-u^\eps,v) &= \int_{D} {\rm div}(A\nabla u) v \d x  -\int_{D_\eps}  {\rm div}(A\nabla u) v \d\omega^\eps+\int_{D_\eps} cuv \d\omega^\eps - \int_{D} cuv \d x\\
 & - \int_{D_\eps} A\nabla u\cdot\nabla \omega^\eps v \d x +\int_{D_\eps} b u v  |\nabla \omega^\eps| \d x- \int_{D_\eps} gv |\nabla\omega^\eps| \d x\\
 & +\int_D fv \d x - \int_{D_\eps} fv \d\omega^\eps .
\end{align*}
In view of Theorem~\ref{thm:sobolev_diffuse_volume} we have the estimates
\begin{align*}
 |\int_{D_\eps}  {\rm div}(A\nabla u) v \d\omega^\eps -\int_{D} {\rm div}(A\nabla u) v \d x | &\leq C \eps^{2-\frac{1}{p'}} \|{\rm div}(A\nabla u)\|_{W^{1,\infty}(D_\eps)} \|v\|_{W^{1,p'}(D_\eps;\omega^\eps)},\\
 |\int_{D_\eps} cuv \d\omega^\eps -\int_{D} cuv \d x| &\leq C  \eps^{2-\frac{1}{p'}} \|cu\|_{W^{1,\infty}(D_\eps)} \|v\|_{W^{1,p'}(D_\eps;\omega^\eps)},\\
 |\int_D fv \d x - \int_{D_\eps} fv \d\omega^\eps | &\leq C  \eps^{2-\frac{1}{p'}} \|f\|_{W^{1,\infty}(D_\eps)} \|v\|_{W^{1,p'}(D_\eps;\omega^\eps)}.
\end{align*}
Since $\nabla \omega^\eps = -n |\nabla \omega^\eps|$, the remaining terms can be estimated as follows
\begin{align*}
  \int_{D_\eps} (n\cdot A\nabla u + b u -g) v |\nabla \omega^\eps| \d x \leq C \eps^{1+\frac{1}{p}} \|n\cdot A\nabla u + b u -g\|_{W^{2,\max\{p,p'\}}(D_\eps;\omega^\eps)} \|v\|_{W^{1,p'}(D_\eps;\omega^\eps)}
\end{align*}
where we used $n\cdot A\nabla u + b u -g= 0$ on $\partial D$ and Theorem~\ref{thm:sobolev_diffuse_boundary2}.
Hence, taking the supremum over all $v\in W^{1,p'}(D_\eps;\omega^\eps)$ and observing that $2-\frac{1}{p'}= 1+\frac{1}{p}$ yields the assertion.
\end{proof}

\begin{corollary}\label{cor:rate_sobolev_robin_p2}
 Let the assumptions of Theorem~\ref{thm:rate_sobolev_robin_p} hold true.
 Then, there exists a constant $C$ independent of $\eps$ such that
 \begin{align*}
  \|u-u^\eps\|_{W^{1,2}(D_\eps;\omega^\eps)} \leq C\eps^{\frac{3}{2}}.
 \end{align*}
\end{corollary}
\begin{proof}
Set $p=2$ in Theorem~\ref{thm:rate_sobolev_robin_p}. The assertion follows from $\X_2^\eps=W^{1,2}(D_\eps;\omega^\eps)$ with equivalent norms.
\end{proof}

\begin{remark}\label{rem:rate_sobolev_robin_l2}
 Setting $p=1$ in Theorem~\ref{thm:rate_sobolev_robin_p}, we obtain $\|u-u^\eps\|_{\X_1^{\eps}} \leq C\eps^{2}$.
 Let us assume that the norms of $W^{1,1}(D_\eps;\omega^\eps)$ and $\X_1^\eps$ are equivalent (uniform with respect to $\eps$). Then continuity of the embedding $W^{1,1}(D)\hookrightarrow L^{\frac{n}{n-1}}(D)$ implies the existence of a constant $C$ independent of $\eps$ such that
 \begin{align*}
  \|u-u^\eps\|_{L^{\frac{n}{n-1}}(D)} \leq C\eps^{2}.
 \end{align*}
 In particular for $n=1$, we obtain $\|u-u^\eps\|_{L^{p}(D)}=O(\eps^{2})$ for any $1\leq p\leq \infty$, and for $n=2$ we obtain $\|u-u^\eps\|_{L^{2}(D)}=O(\eps^{2})$, thus we recover the formal results of \cite{LervagLowengrub2014}.
\end{remark}

\subsection{Dirichlet boundary conditions}\label{sec:dirichlet}
In this section we consider the diffuse domain approximation of second order elliptic equations with Dirichlet boundary conditions:
Find $u$ such that
\begin{align}
 -{\rm div}(A\nabla u) + cu = f &\quad\text{in } D,\label{eq:elliptic_pde2}\\
   u = g &\quad\text{on } \partial D.\label{eq:bc_dirichlet}
\end{align}
In order to obtain (weak) solutions to \eqref{eq:elliptic_pde2}--\eqref{eq:bc_dirichlet}, let us consider the following weak formulation:
Find $u\in W^{1,2}(D)$ such that 
\begin{align}
 a(u,v)&=\ell(v)\quad\text{for all } v\in W_0^{1,2}(D)\text { such that } u=g \ \text{ on }\partial D,\label{eq:dirichlet_var}
\end{align}
with bilinear and linear form
\begin{align*}
 a(u,v)  = \int_D A\nabla u \cdot \nabla v + c u v \d x,\qquad \ell(v) = \int_D fv \d x.
\end{align*}
Here $W^{1,2}_0(D)$ is the kernel of the trace operator on $W^{1,2}(D)$.
The weak form \eqref{eq:dirichlet_var} is well-posed under assumptions (C1)--(C2) which is shown by using the Lax-Milgram lemma.
It is well-known that the solution $u$ to \eqref{eq:dirichlet_var} is characterized as the solution of the minimization problem
\begin{align*}
 a(v,v)-\ell(v) \to \min_{v\in W^{1,2}(D)} \text{ such that } v=g \text{ in } W^{1/2,2}(\partial D).
\end{align*}
Using the Lagrange formalism this constrained optimization problem is equivalent to finding a saddle-point $(u,\lambda)\in W^{1,2}(D)\times W^{-1/2,2}(\partial D)$ of the Lagrangian
\begin{align}\label{eq:lagrangian}
 L(v,\mu) = a(v,v) - \ell(v) - \langle \mu, g-v\rangle\qquad \text{with } v\in W^{1,2}(D),\ \mu\in W^{-1/2,2}(\partial D).
\end{align}
Here, $W^{-1/2,2}(\partial D)$ is the topological dual space of the $W^{1,2}(D)$-trace space $W^{1/2,2}(\partial D)$, and $\langle \cdot,\cdot\rangle$ denotes the duality pairing between $W^{-1/2,2}(\partial D)$ and $W^{1/2,2}(\partial D)$.
The variational characterization of the saddle-point problem is the following: Find $(u,\lambda)\in W^{1,2}(D)\times W^{-1/2,2}(\partial D)$ such that
\begin{align}
 a(u,v) + \langle \lambda, v\rangle &= \ell(v)\qquad \text{for all } v\in W^{1,2}(D), \label{eq:sp_dirichlet1}\\
 \langle \mu, u\rangle &= \langle \mu,g\rangle \qquad \text{for all } \mu\in W^{-1/2,2}(\partial D).\label{eq:sp_dirichlet2}
\end{align}
We have by definition of the norm on $W^{-1/2,2}(\partial D)$ that 
\begin{align*}
  \|\mu\|_{W^{-1/2,2}(\partial D)} = \sup_{v\in W^{1/2,2}(\partial D)\setminus\{0\}}\frac{\langle \mu, v \rangle}{\|v\|_{W^{1/2,2}(\partial D)}}
\end{align*}
which asserts an inf-sup condition for the bilinear form $(\mu,v)\mapsto \langle \mu,v\rangle$.
Well-posedness of the latter saddle-point problem can then be shown by using Brezzi's splitting theorem \cite{Brezzi74}, cf.\@ \cite[Chapter III]{Braess2007}.
Next, let us introduce a penalized version of \eqref{eq:sp_dirichlet1}--\eqref{eq:sp_dirichlet2} which establishes a connection to elliptic problems with Robin boundary condition discussed in Section~\ref{sec:robin}: 
Let $\beta>0$.
Find $(u_\beta,\lambda_\beta)\in W^{1,2}(D)\times W^{-1/2,2}(\partial D)$ such that
\begin{align}
 a(u_\beta,v) + \langle \lambda_\beta, v\rangle &= \ell(v)\qquad \text{for all } v\in W^{1,2}(D), \label{eq:sp_dirichlet1_penalty}\\
 \langle \mu, u_\beta\rangle -\beta \langle \lambda_\beta ,\mu \rangle &= \langle \mu,g\rangle \qquad \text{for all } \mu\in W^{-1/2,2}(\partial D).\label{eq:sp_dirichlet2_penalty}
\end{align}
In slight abuse of notation, $\langle \lambda, \mu \rangle$ denotes the inner product on $W^{-1/2,2}(D)$, and is defined as
$\langle \lambda, \mu \rangle_{W^{-1/2,2}(\partial D)} = \langle J\lambda, J\mu \rangle_{W^{1/2,2}(\partial D)}$.
Here, $J: W^{-1/2,2}(\partial D)\to W^{1/2,2}(\partial D)$ is the Riesz isomorphism and $J\lambda$ is given as the trace of the solution to the Neumann problem
\begin{align*}
 -\Delta w + w = 0 \quad\text{in } D,\qquad \partial_n w = \lambda \quad\text{on } \partial D.
\end{align*}
We have that
\begin{align*}
 \|\mu\|_{W^{-1/2,2}(\partial D)} = \langle \mu, \mu \rangle_{W^{-1/2,2}(\partial D)}^{1/2} 
 = \sup_{v\in W^{1,2}(D)\setminus\{0\}} \frac{ \langle w, v\rangle_{W^{1,2}(D)}} {\|v\|_{W^{1,2}(D)}} = \|w\|_{W^{1,2}(D)}.
\end{align*}
Well-posedness of \eqref{eq:sp_dirichlet1_penalty}--\eqref{eq:sp_dirichlet2_penalty} can be shown with a penalty version of Brezzi's splitting theorem, cf.\@ e.g. \cite{Braess2007}. In particular $(u_\beta,\lambda_\beta)$ is bounded in $W^{1,2}(D)\times W^{-1/2,2}(\partial D)$ independent of $\beta$.

Since $u_\beta$ depends Lipschitz-continuously on $\beta$, the error between the solution to \eqref{eq:sp_dirichlet1}--\eqref{eq:sp_dirichlet2} and \eqref{eq:sp_dirichlet1_penalty}--\eqref{eq:sp_dirichlet2_penalty} is $O(\beta)$; for a proof let us refer to \cite[Ch. III, Thm~4.11, Cor.~4.15]{Braess2007}.
\begin{lemma}\label{lem:penalty}
 Let $(u,\lambda), (u_\beta,\lambda_\beta)\in W^{1,2}(D)\times W^{-1/2,2}(\partial D)$ be solutions to \eqref{eq:sp_dirichlet1}--\eqref{eq:sp_dirichlet2} and \eqref{eq:sp_dirichlet1_penalty}--\eqref{eq:sp_dirichlet2_penalty}, respectively. Then there exists a constant $C$ independent of $\beta$ such that
 \begin{align*}
  \|u-u_\beta\|_{W^{1,2}(D)} + \|\lambda-\lambda_\beta\|_{W^{-1/2,2}(D)} \leq C \beta.
 \end{align*}
\end{lemma}
%
Using $\mu=v/\beta$ with $v\in W^{1,2}(D)$ in \eqref{eq:sp_dirichlet2_penalty} and adding the resulting equation to \eqref{eq:sp_dirichlet1_penalty} yields the following reduced problem:
Find $u_\beta\in W^{1,2}(D)$ such that
\begin{align*}
 a(u_\beta,v) + \frac{1}{\beta}\int_{\partial D} u_\beta v\d\sigma &= \ell(v)+\frac{1}{\beta}\int_{\partial D} gv\d\sigma\qquad \text{for all } v\in W^{1,2}(D).
\end{align*}
This is a weak form of a Robin-type problem  with boundary condition $n\cdot A\nabla u_\beta + \frac{1}{\beta} u_\beta=\frac{1}{\beta}g$ on $\partial D$.
 This method of relaxation of the Dirichlet boundary condition is widely known as the penalty method \cite{Babuska73}.
Let $u_\beta^\eps$ denote the diffuse approximation to $u_\beta$ as defined in Section~\ref{sec:robin}, i.e. $u_\beta^\eps$ satisfies
\begin{align}\label{eq:dirichlet_var_diffuse}
 \int_\Omega A\nabla u_\beta^\eps \cdot \nabla v + c u_\beta^\eps v \d\omega^\eps + \frac{1}{\beta}\int_{\Omega} u_\beta^\eps v  |\nabla \omega^\eps| \d x
 = \int_\Omega f v  \d\omega^\eps + \frac{1}{\beta}\int_{\Omega} g v |\nabla\omega^\eps|\d x
\end{align}
for all $v\in W^{1,2}(D_\eps;\omega^\eps)$.
Combining the estimates in Lemma~\ref{lem:penalty} and Theorem~\ref{thm:rate_lp_robin}, we have
\begin{align}\label{eq:error_dirichlet}
 \|u- u_\beta^\eps\|_{W^{1,2}(D)} \leq \|u - u_\beta\|_{W^{1,2}(D)} + 2\| u_\beta - u_\beta^\eps\|_{W^{1,2}(D_\eps;\omega^\eps)}\leq C (\beta + \frac{1}{\beta}\eps^{\frac{1}{2}-\frac{1}{p}})
\end{align}
for $p\leq 2_\alpha^*$, and $u\in W^{1,p}(D)$. Choosing $\beta=\eps^\sigma$, $\sigma>0$, yields
\begin{align*}
 \|u- u_\beta^\eps\|_{W^{1,2}(D)} \leq C( \eps^\sigma + \eps^{\frac{1}{2}-\frac{1}{p}-\sigma}).
\end{align*}
Balancing the exponents on the right-hand side, we obtain the optimal choice $\sigma = \frac{1}{4}-\frac{1}{2p}$. The corresponding estimates are then given by the next theorems:

\begin{theorem}\label{thm:rate_lp_dirichlet}
 Let (C1)--(C2) hold. Moreover, assume that $u\in W^{1,p}(D)$ with $2\leq p \leq 2_\alpha^*$ is a solution to \eqref{eq:dirichlet_var} and $u_\beta^\eps\in W^{1,2}(D_\eps;\omega^\eps)$ is a solution to \eqref{eq:dirichlet_var_diffuse}. Then for $\beta=\eps^\sigma$ and $\sigma = \frac{1}{4}-\frac{1}{2p}$ there exists a constant $C>0$ independent of $\eps$ such that
 \begin{align*}
  \|u-u_\beta^\eps\|_{W^{1,2}(D)} \leq C \eps^{\frac{1}{4}-\frac{1}{2p}}.
 \end{align*}
\end{theorem}
\begin{theorem}\label{thm:rate_sobolev_dirichlet_p}
 Let $\partial D$ be of class $C^{\infty}$, and let $f,g\in C^\infty(\overline{\Omega})$ and let (C1)--(C2) hold. 
 Moreover, let $A\in C^{\infty}(\overline{\Omega})^{n\times n}$, $c\in C^{\infty}(\overline{\Omega})$, and let $u_\beta^\eps \in W^{1,2}(D_\eps;\omega^\eps)$ denote the solution to \eqref{eq:dirichlet_var_diffuse}, and let $u\in W^{1,2}(D)$ denote the solution to \eqref{eq:dirichlet_var}.
  Then, for $\beta=\eps^\sigma$ and $\sigma = \frac{3}{4}$ there exists a constant $C$ independent of $\eps$ such that
  \begin{align*}
    \|u-u_\beta^\eps\|_{W^{1,2}(D)} \leq C\eps^{\frac{3}{4}}.
  \end{align*}
\end{theorem}

\begin{remark}
 Due to the regularity of $\partial D$, we can extend $u-u_\beta$ to $\RR^n$ such that $\|u-u_\beta\|_{W^{1,2}(\RR^n)}\leq C \|u-u_\beta\|_{W^{1,2}(D)}$ \cite{Adams1975}. In view of \eqref{eq:error_dirichlet} and the following chain of inequalities
 \begin{align*}
  \|u-u_\beta\|_{W^{1,2}(D_\eps;\omega^\eps)}\leq \|u-u_\beta\|_{W^{1,2}(D_\eps)}\leq \|u-u_\beta\|_{W^{1,2}(\RR^n)}\leq C \|u-u_\beta\|_{W^{1,2}(D)}\leq C\beta
 \end{align*}
 the $W^{1,2}(D)$-norm in Theorem~\ref{thm:rate_lp_dirichlet} and Theorem~\ref{thm:rate_sobolev_dirichlet_p} can be replaced by $W^{1,2}(D_\eps;\omega^\eps)$-norm.
 Note, however, that for $v\in W^{1,2}(D_\eps;\omega^\eps)$ we have $v_{\mid D}\in W^{1,2}(D)$, but for the extension $\tilde v$ of $v_{\mid D}$ from $D$ to $\RR^n$ in general $\tilde v_{\mid D_\eps} \neq v$.
\end{remark}

\begin{remark}\label{rem:rate_sobolev_dirichlet_p2}
 In order to obtain an analogous statement of Theorem~\ref{thm:rate_sobolev_robin_p} for the Dirichlet case, we would need an analog of Lemma~\ref{lem:penalty} for the $W^{1,p}$-norm. Hence, for illustration let us assume that $\|u-u_\beta\|_{W^{1,p}(D)}\leq C\beta$ for $1\leq p\leq \infty$. Moreover, by regularity of $\partial D$, we can assume stability of the extension of $u$ and $u_\beta$ to $\Omega$, i.e. $\|u-u_\beta\|_{W^{1,p}(\Omega)}\leq C\|u-u_\beta\|_{W^{1,p}(D)}$. Then, we arrive at the estimate
 \begin{align*}
  \|u-u_\beta^\eps\|_{\X_p^\eps} &\leq \|u-u_\beta\|_{\X_p^\eps} + \|u_\beta-u_\beta^\eps\|_{\X_p^\eps}\\
  &\leq C \|u-u_\beta\|_{W^{1,p}(\Omega)} + \|u_\beta-u_\beta^\eps\|_{\X_p^\eps}\\
  &\leq C(\beta + \eps^{1+\frac{1}{p}}/\beta)
  \leq C \eps^{\frac{1}{2}+\frac{1}{2p}}
 \end{align*}
 using $\beta =  \eps^{\frac{1}{2}+\frac{1}{2p}}$ and Theorem~\ref{thm:rate_sobolev_robin_p}.
 Assuming furthermore that the norms of $W^{1,1}(D_\eps;\omega^\eps)$ and $\X_1^\eps$ are equivalent (uniform with respect to $\eps$), and using continuity of the embedding $W^{1,1}(D)\hookrightarrow L^{\frac{n}{n-1}}(D)$ we infer that
 \begin{align*}
  \|u-u^\eps_\beta\|_{L^{\frac{n}{n-1}}(D)} \leq C\eps.
 \end{align*}
 In particular for $n=1$, we obtain $\|u-u^\eps_\beta\|_{L^{p}(D)}=O(\eps)$ for any $1\leq p\leq \infty$, and for $n=2$ we obtain $\|u-u^\eps_\beta\|_{L^{2}(D)}=O(\eps)$.
 The reader should compare this to the results of \cite{FranzGaertnerRoosVoigt2012} where for $n=1$ a rate $O(\eps^{1-\delta})$ for any $\delta>0$ in the $L^\infty$-norm is shown.
 Moreover, in \cite{ReuterHillHarrison2012} an $L^2$-rate $O(\eps)$ for Poisson's equation in three dimensions has been obtained numerically. There, it is also suggested to choose $\beta=\eps$, which complies with our analysis. Let us note however that the diffuse domain method in \cite{ReuterHillHarrison2012} is somewhat different from ours.
\end{remark}

\subsection{Neumann boundary conditions}\label{sec:neumann}
Consider the following second order elliptic equation with Neumann-type boundary condition:
Find $u$ such that
\begin{align}
 -{\rm div}(A\nabla u) + cu = f &\quad\text{in } D,\label{eq:elliptic_pde3}\\
  n\cdot A\nabla u  = g &\quad\text{on } \partial D.\label{eq:bc_neumann}
\end{align}
In order to obtain (weak) solutions to \eqref{eq:elliptic_pde3}--\eqref{eq:bc_neumann}, let us consider the following weak formulation:
Find $u\in W_\diamond^{1,2}(D)=\{v\in W^{1,2}(D): \int_D v dx = 0\}$ such that 
\begin{align}\label{eq:neumann_var}
 a(u,v)=\ell(v)\quad\text{for all } v\in W_\diamond^{1,2}(D),
\end{align}
with bilinear and linear form
\begin{align*}
 a(u,v)  = \int_D A\nabla u \cdot \nabla v + c u v \d x ,\qquad
 \ell(v) = \int_D fv \d x + \int_{\partial D} g v \d\sigma.
\end{align*}
In view of the usual Poincar\'e inequality for $W^{1,2}(D)$, the weak form \eqref{eq:neumann_var} is well-posed under the assumptions (C1)-(C2).
\begin{lemma}
 Let (C1)--(C2) hold. Moreover, let $f\in L^2(D)$ and $g\in W^{1,2}(D)$. Then there exists a unique $u\in W_\diamond^{1,2}(D)$ satisfying \eqref{eq:neumann_var}, and there exists $C>0$ such that
 \begin{align*}
  \|u\|_{W^{1,2}(D)}\leq C (\|f\|_{L^2(D)} +\|g\|_{L^2(\partial D)} ).
 \end{align*}
\end{lemma}
The diffuse approximation of \eqref{eq:neumann_var} is then: Find $u^\eps\in W_\diamond^{1,2}(D_\eps;\omega^\eps)=\{v\in W^{1,2}(D_\eps;\omega^\eps): \int_{D_\eps} v d\omega^\eps = 0\}$ such that 
\begin{align}\label{eq:neumann_var_diffuse}
 a^\eps(u^\eps,v)=\ell^\eps(v) \quad\text{for all } v\in W_\diamond^{1,2}(D_\eps;\omega^\eps),
\end{align}
where the corresponding bilinear and linear form are given by
\begin{align*}
 a^\eps(u^\eps,v)=\int_\Omega A\nabla u^\eps \cdot \nabla v + c u^\eps v \d\omega^\eps ,\qquad
 \ell^\eps(v) = \int_\Omega f v  \d\omega^\eps + \int_{\Omega} g v |\nabla\omega^\eps|\d x.
\end{align*}
\begin{lemma}
 Let (C1)--(C2) hold. Moreover, let $f\in L^2(D_\eps;\omega^\eps)$ and $g\in W^{1,2}(D_\eps;\omega^\eps)$. Then there exists a unique $u^\eps\in W_\diamond^{1,2}(D_\eps;\omega^\eps)$ satisfying \eqref{eq:robin_var_diffuse}, and there exists $C>0$ independent of $\eps$ such that
 \begin{align*}
  \|u^\eps\|_{W^{1,2}(D_\eps;\omega^\eps)}\leq C (\|f\|_{L^2(D_\eps;\omega^\eps)} +\|g\|_{W^{1,2}(D_\eps;\omega^\eps)} ).
 \end{align*}
\end{lemma}
\begin{proof}
 Continuity of $a^\eps$ and $\ell^\eps$ with respect to the $W^{1,2}(D_\eps;\omega^\eps)$-topology is obvious.
 Coercivity of $a^\eps$ on $W_\diamond^{1,2}(D_\eps;\omega^\eps)$ is a direct consequence of the positivity of $A$ and the Poincar\'e inequality, see Corollary~\ref{cor:poincare}. An application of the Lax-Milgram lemma yields the assertion.
\end{proof}

Having established existence of solutions to the Neumann problems, convergence results can now be derived as in the Robin case above when setting $b=0$. We leave this to the reader.
Let us mention that the restriction from $W^{1,2}(D_\eps;\omega^\eps)$ to the space $W_\diamond^{1,2}(D_\eps;\omega^\eps)$ is only necessary if $\inf_{x\in \Omega} c(x)=0$. In this case the algebraic condition $\int_{\Omega} u \d\omega^\eps=0$ has to be treated with care in a numerical implementation. We do not want to go into details here, but let us refer the reader to \cite{BochevLehoucq05}. If otherwise $\inf_{x\in \Omega} c(x)>0$, we could equally well pose \eqref{eq:neumann_var_diffuse} in the space $W^{1,2}(D_\eps;\omega^\eps)$ and the implementational details are similar to those of the Robin case. We thus will not dwell on the Neumann case in our further discussion.

\section{Numerical Results}\label{sec:numerics}

In the following we report the results of numerical tests related to the above investigations used conformal first order finite elements. Our particular interest here is not the efficient solution of realistic problems, but rather to test the sharpness of error estimates in different situations by computational experiments. In order to have an ''exact`` solution $u$ we solve the original problem with sharp interface on a very fine mesh such that the numerical error is negligible. Moreover, in the computation of the diffuse domain solution we make sure that the largest mesh parameter, i.e. $h_{\textnormal{max}}$, is for all computations smaller than $\eps^2$ such that the numerical accuracy does not pollute the experimental order of convergence. 
The error $e^\eps = u - u^\eps$ will always be measured in the relative norms
  \begin{equation}
   \frac{\|u-u^\eps\|_{W^{k,p}(D)}}{\|u\|_{W^{k,p}(D)}}, \nonumber
  \end{equation}
  where $W^{0,p} = L^p$. We provide several log-log plots of errors vs. $\eps$, which shall be comparable to the theoretical orders represented by lines in those plots, see Figure \ref{fig:logA}, \ref{fig:logB}, \ref{fig:logplotD} and \ref{fig:logE}. Since the constants in the estimates cannot be made explicit, we have to fix one value and hence decide to plot the theoretical rates in all log-log plots such that they coincide with the experimental rates for the largest value of $\eps$, see Figure \ref{fig:logA}, \ref{fig:logB}, \ref{fig:logplotD} and \ref{fig:logE}. 

For most simulations (Case A-D below) we work with the domain $D = \{(x_1,x_2): x_1^2 + x_2^2 < 0.5\}$, which obviously satisfies all regularity requirements. The mesh representation
of this domain $D$ consists of $3,336,340$ vertices. The mesh representation of the 
domain $D_\eps$ is simply a scaling of the mesh representation of $D$ with $1 + \eps$.  Finally we present an example with the domain $D = (0,1) \times (0,1)$, i.e. the unit square (Case E), which indicates that the same rates still hold for piecewise smooth domains.  
In all test cases we use the function $S$ from Example~\ref{ex:S} (i).

\subsection{Case A: Robin BC with smooth parameters}\label{subsec:caseA}
In our first simulations, we consider the boundary value problem \eqref{eq:elliptic_pde}-\eqref{eq:bc_robin},
with the smooth parameters
$$ A(x_1,x_2) = c(x_1,x_2) = 1, \quad f(x_1,x_2) = 10\sin(\pi x_1) - 5x_2^2, \quad g(x_1,x_2) = 0, \quad \alpha(x_1,x_2) = 1. $$
From Theorem \ref{thm:rate_sobolev_robin_p}, Corollary \ref{cor:rate_sobolev_robin_p2} and Remark \ref{rem:rate_sobolev_robin_l2},
we expect the error $e^\eps$ to converge with a rate $O(\eps^2)$ in $W^{1,1}(D)$ and $L^2(D)$ and a rate of $O(\eps^{3/2})$ in $W^{1,2}(D)$. Furthermore we expect rates of order  $O(\eps)$ in  $W^{1,\infty}(D)$. We mention that except the rate in  $W^{1,2}(D)$ these expectations rely on assumptions we cannot verify rigorously. 
From Table \ref{tab:epsA} and the log-log plot of Figure \ref{fig:logA} we observe that the numerical results reproduce these rates
very accurately, indicating the sharpness of our estimates and the validity of the assumptions. 
In Figure \ref{fig:comp_caseA}, the solutions $u$ and $u^\eps|_D$ for the Robin boundary problem are presented. From a 
visual perspective, these solution are almost identical.

\begin{figure}
\centerline{%
\begin{tabular}{c@{\hspace{1pc}}c}
\includegraphics[scale=.35]{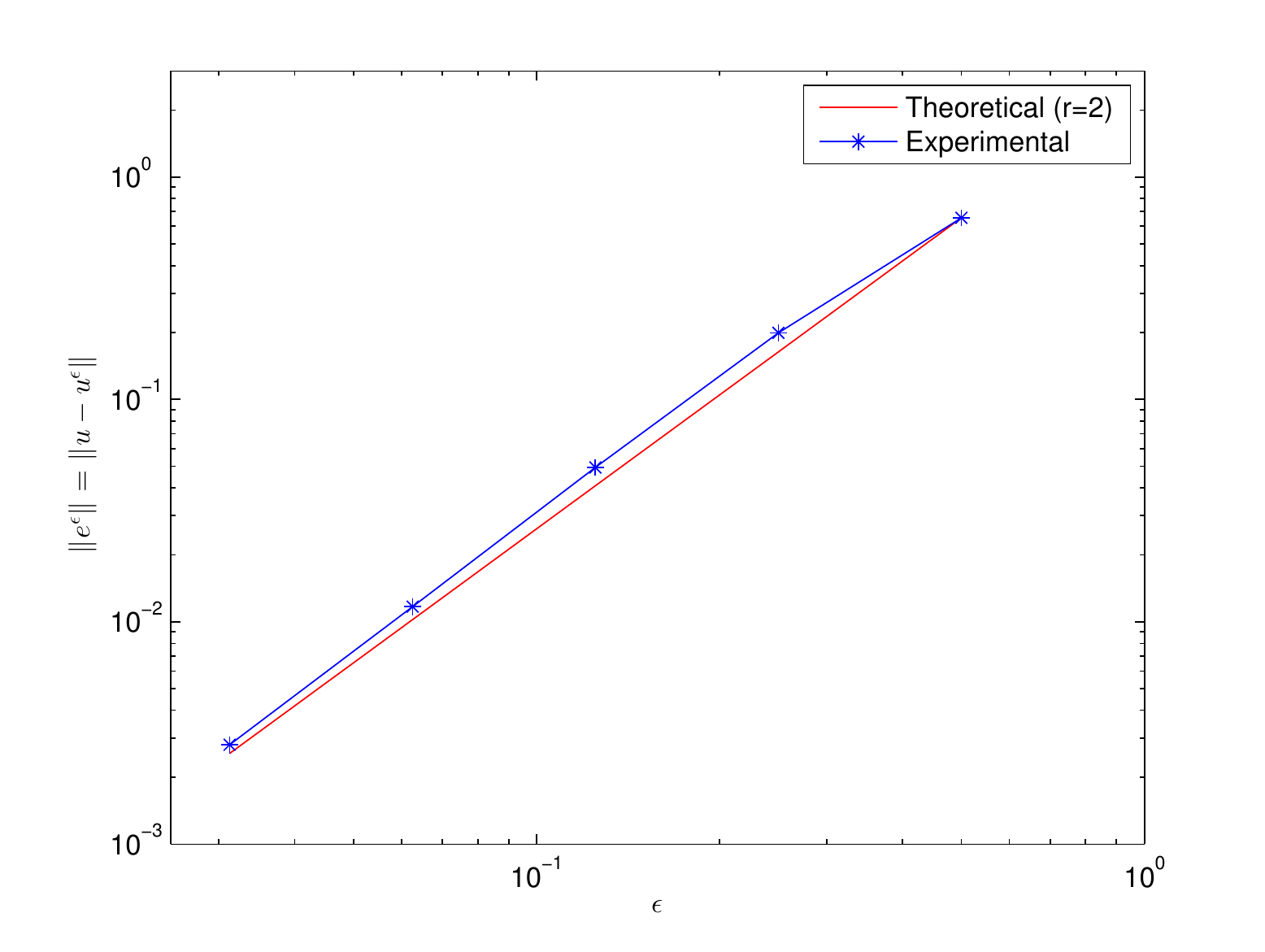} &
\includegraphics[scale=.35]{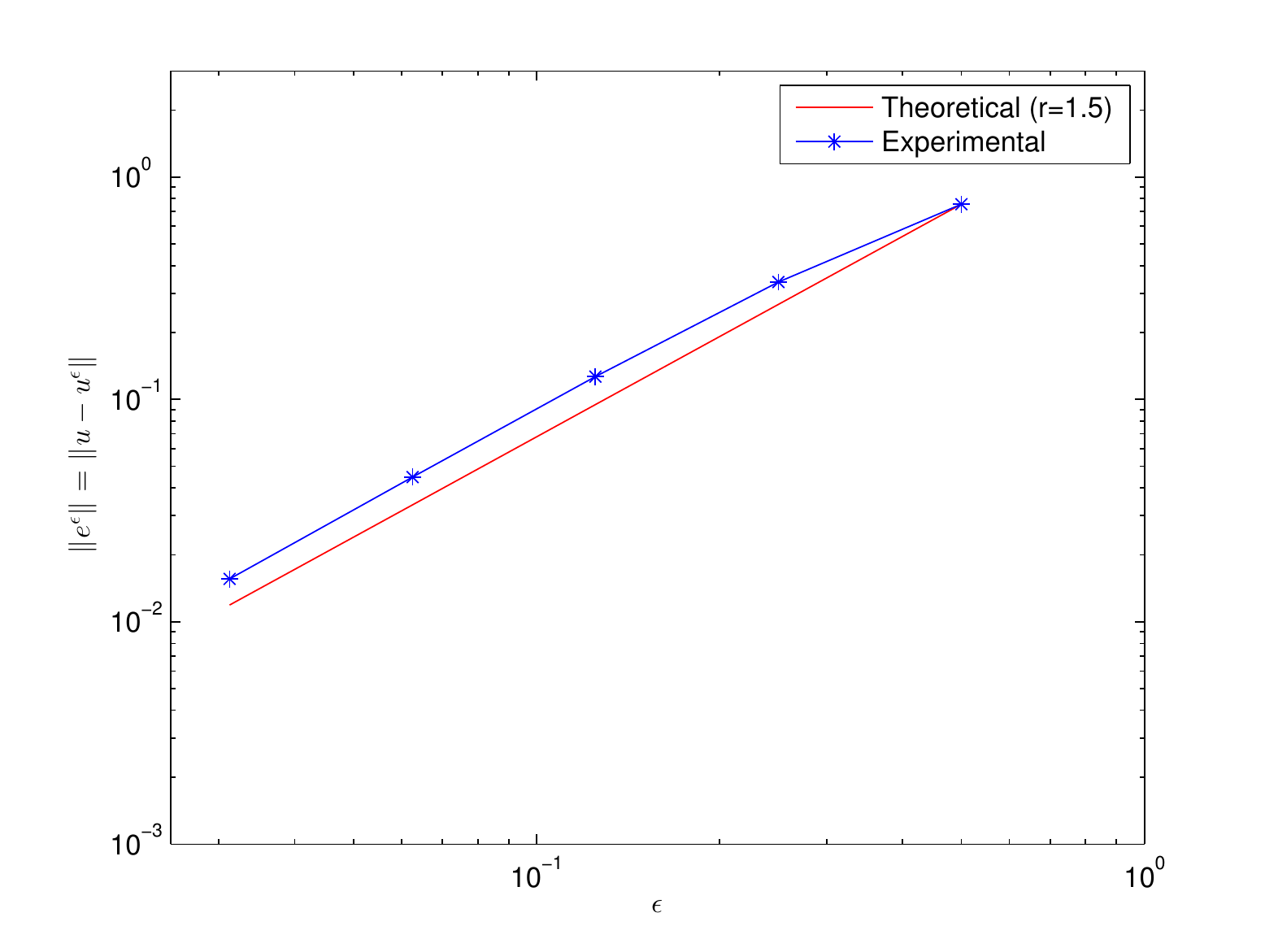} \\
(a)~~ Convergence in $L^2$-norm. & (b)~~ Convergence in $W^{1,2}$-norm. \\
\includegraphics[scale=.35]{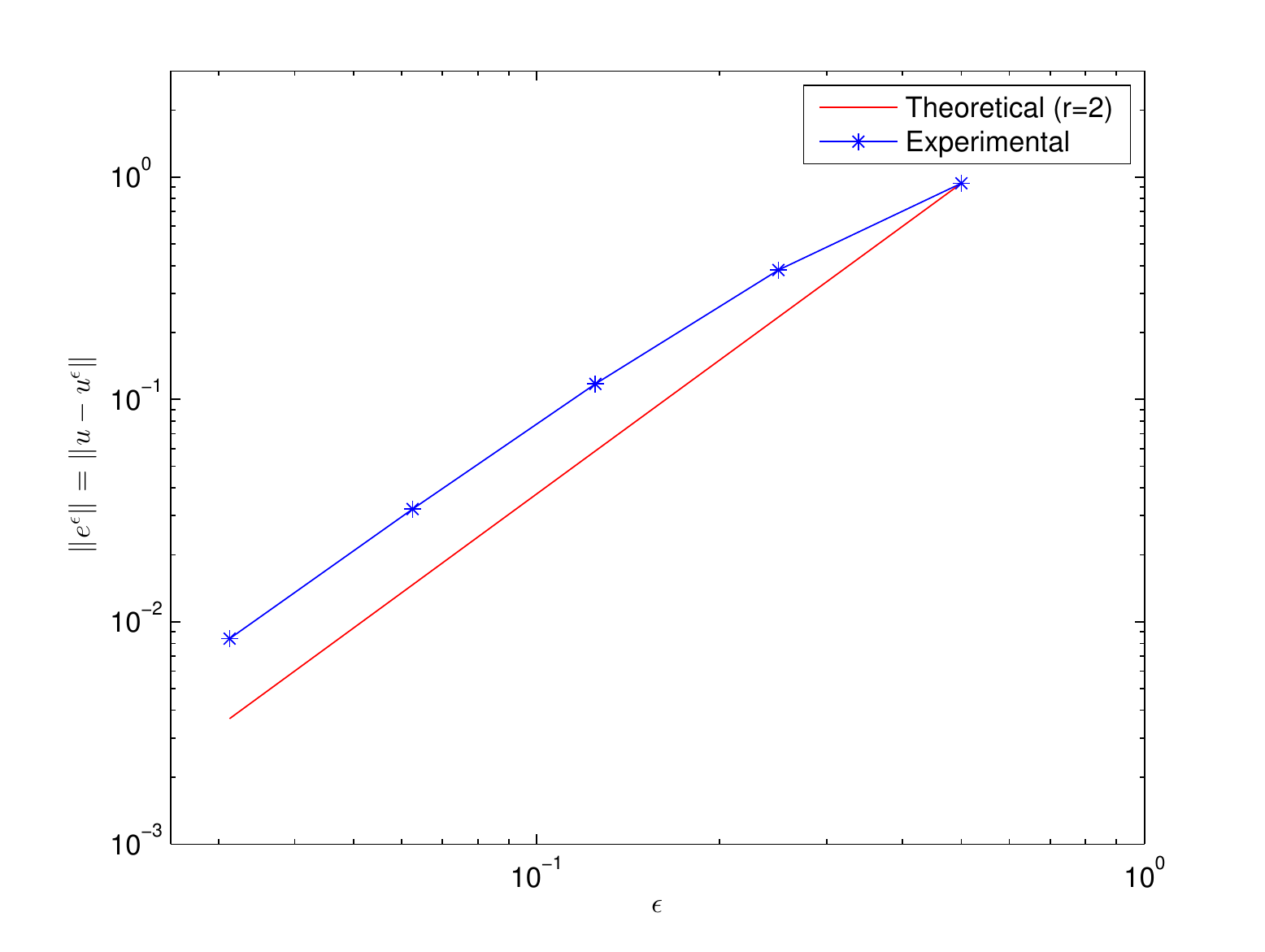} &
\includegraphics[scale=.35]{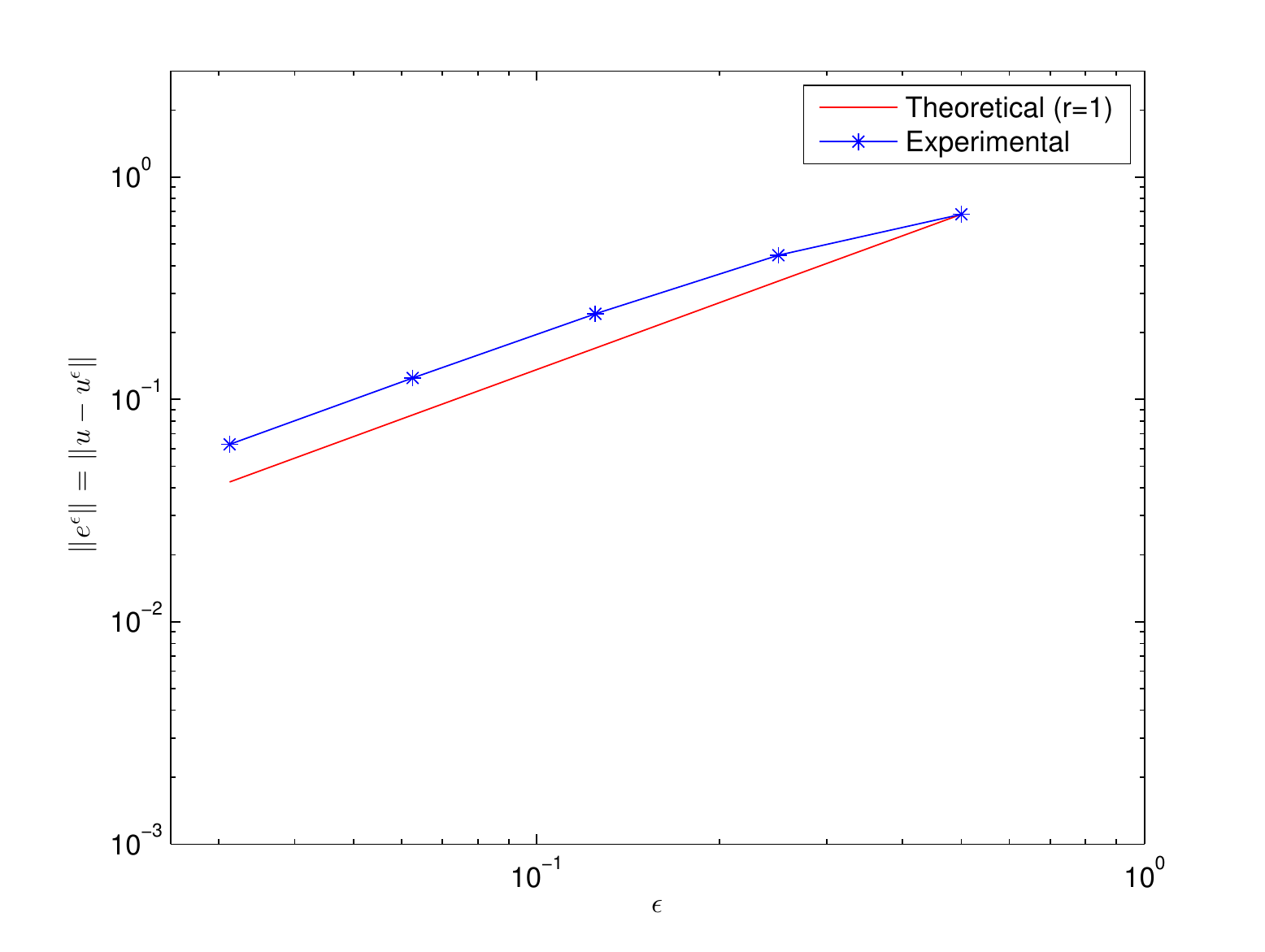} \\
(c)~~ Convergence in $W^{1,1}$-norm. & (d)~~ Convergence in $W^{1,\infty}$-norm. 
\end{tabular}}
\caption{A log-log plot of the convergence rates in Case A. In each subplot we see the actual 
convergence rate (experimental), compared to the theoretical rate of order $O(\eps^r)$. In subplots (a) and (c) $r=2$, in (b) $r=1.5$ and in (d) $r=1$.}\label{fig:logA}
\end{figure}

\begin{table}[ht] \centering
    \begin{tabular}{ | l | l | l | l | l |}
      \hline
      $\eps$  & $\|e^\eps\|_{L^2}$ $ (\log_2(\frac{e^k}{e^{k+1}}))$ & $\|e^\eps\|_{W^{1,2}}$ 
      $ (\log_2(\frac{e^k}{e^{k+1}}))$ & $\|e^\eps\|_{W^{1,1}}$ $(\log_2(\frac{e^k}{e^{k+1}}))$ 
      & $\|e^\eps\|_{W^{1,\infty}}$  \\ \hline \
      $2^{-1}$ & 0.654779        & 0.755160        &  0.938120        & 0.680350        \\ 
      $2^{-2}$ & 0.199128 (1.71) & 0.337471 (1.16) &  0.381236 (1.30) & 0.444653 (0.61) \\ 
      $2^{-3}$ & 0.049532 (2.01) & 0.126688 (1.41) &  0.117386 (1.70) & 0.242804 (0.87) \\ 
      $2^{-4}$ & 0.011767 (2.07) & 0.044757 (1.50) &  0.032170 (1.87) & 0.124745 (0.96) \\ 
      $2^{-5}$ & 0.002818 (2.06) & 0.015637 (1.52) &  0.008464 (1.93) & 0.062845 (0.99) \\ \hline
          \end{tabular}
  \caption{The error $e^\eps = u - u^\eps$ for different norms in Case A.}\label{tab:epsA}
\end{table}

\begin{figure}
\centerline{%
\begin{tabular}{c@{\hspace{1pc}}c}
\includegraphics[scale=.15]{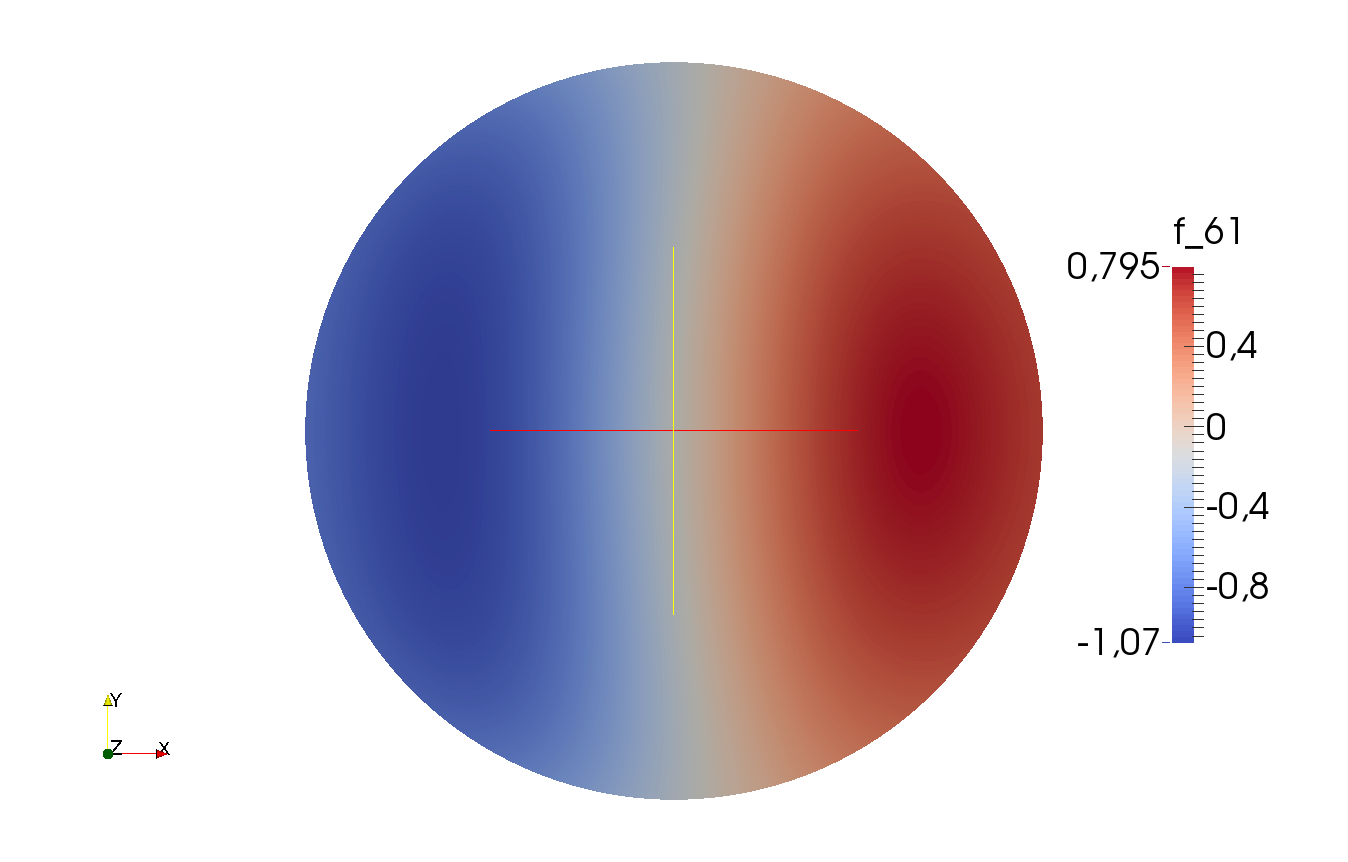} &
\includegraphics[scale=.15]{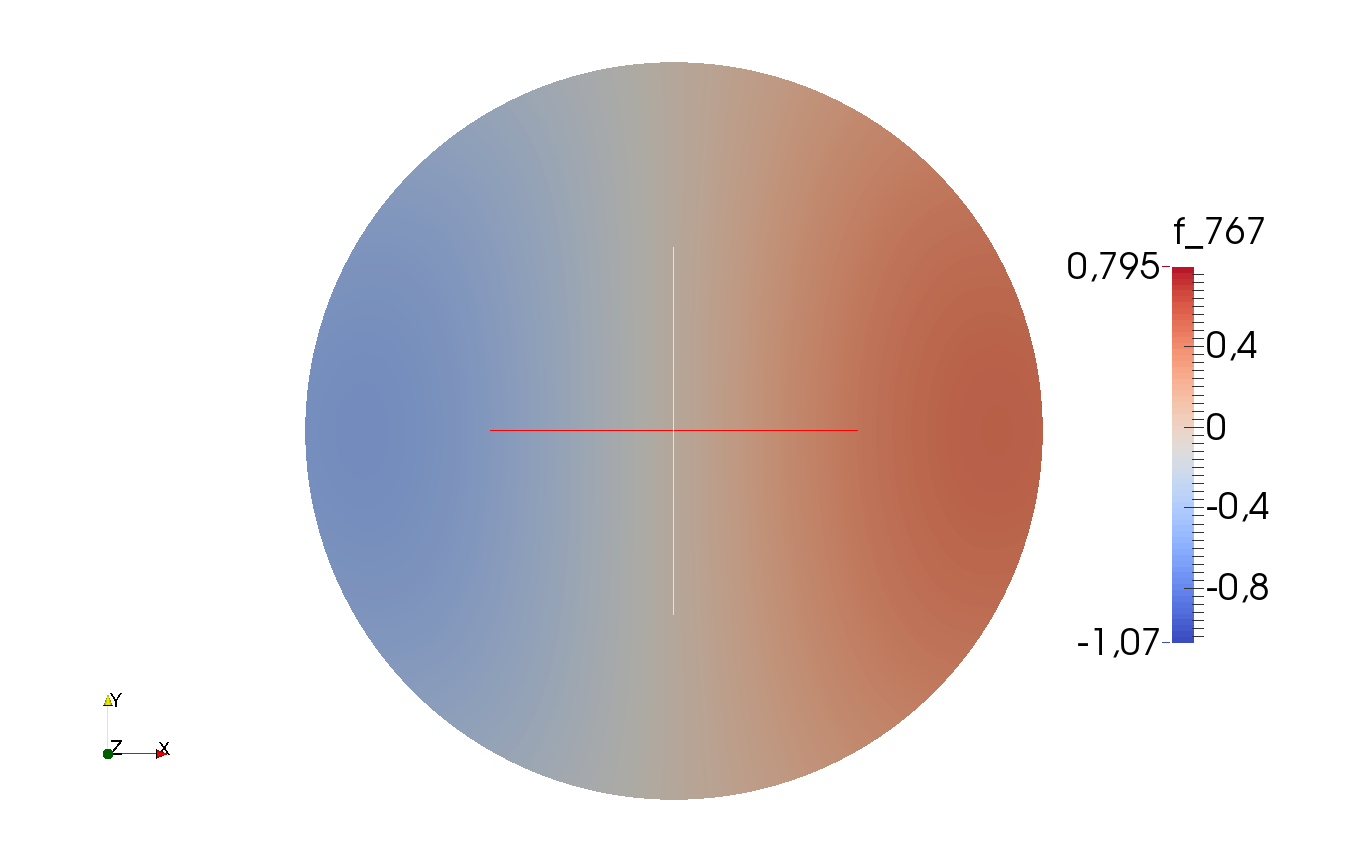} \\
(a)~~ The solution $u$ of \eqref{eq:robin_var_diffuse} for $\eps = 2^{-1}$. & (b)~~ The solution $u^\eps$ of
\eqref{eq:robin_var_diffuse} for $\eps = 2^{-5}$.  
\end{tabular}}
\caption{Comparison of two diffuse domain solutions in Case A. The solution displayed in (b) is visually
identical to the exact solution of \eqref{eq:robin_var}.}\label{fig:comp_caseA}
\end{figure}

\subsection{Case B: Robin BC with discontinuous A matrix}
If the parameter A is no longer smooth, but instead $A \in L^{\infty}(\Omega)^{2\times2}$, the assumptions
for Theorem \ref{thm:rate_sobolev_robin_p} are no longer satisfied. In the second example, we choose a discontinuous
$A\in L^{\infty}(\Omega)^{2\times 2}$ as
\begin{equation*}
 A(x_1,x_2) = \begin{bmatrix} k_1(x_1,x_2) &0 \\ 0& k_2(x_1,x_2), \end{bmatrix}
\end{equation*}
where $k_1, k_2$ are piecewise constant functions with a jump discontinuity close to $\partial D$.
All other parameters are the same as in Case A.

\begin{figure}
 \includegraphics[scale=.5]{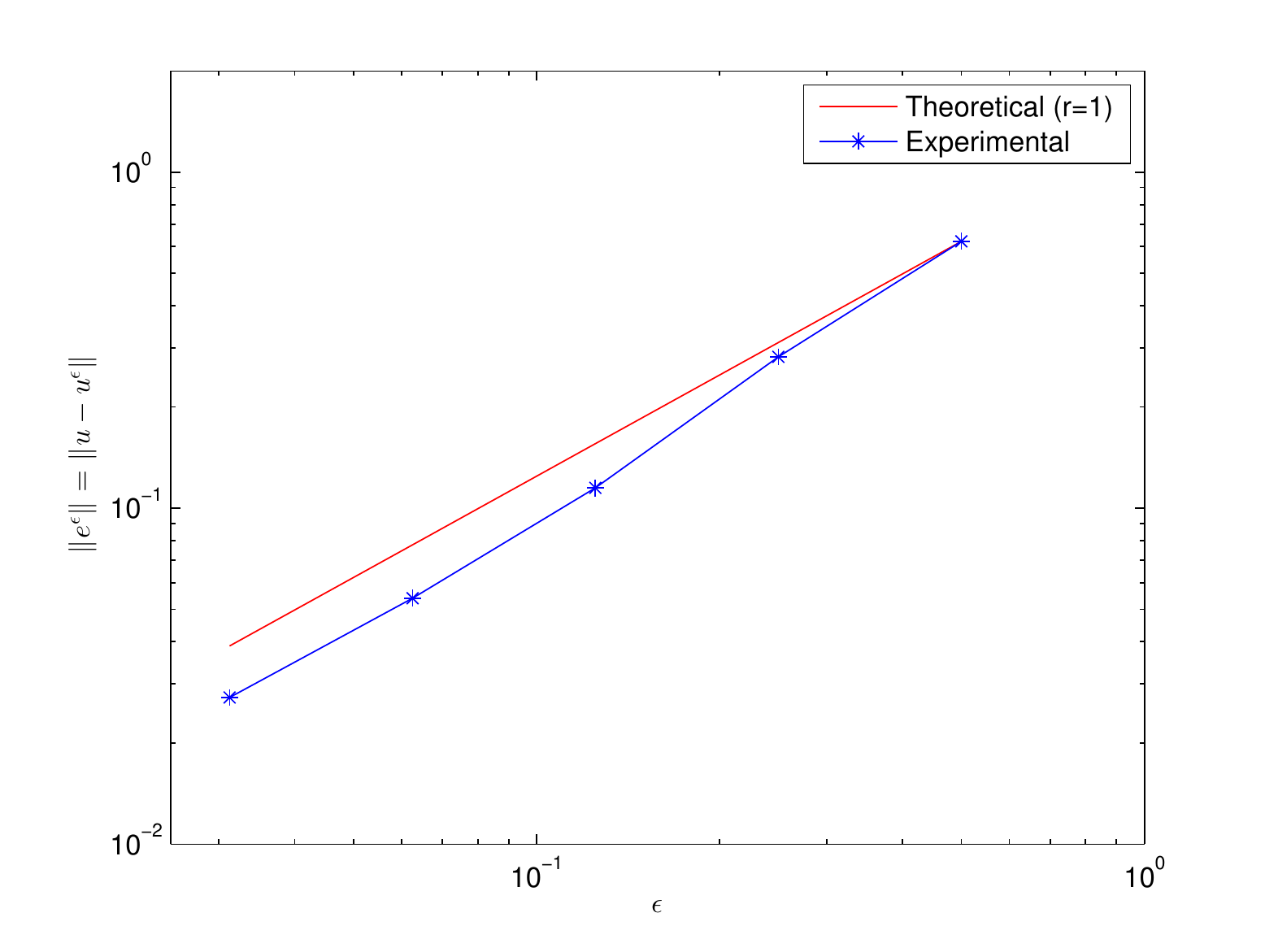}
 \caption{A log-log plot of the $W^{1,2}$-convergence in Case B. We see the actual 
convergence rate (experimental), compared to the theoretical rate of order $O(\eps)$.} \label{fig:logB}
\end{figure}

From Table \ref{tab:epsB} and the log-log plot in  Figure \ref{fig:logB}, we see that the convergence rate of the error is one order worse than in Case A.
In particular, we obtain linear convergence in $W^{1,2}$, which is still better than the theoretical result of order $\eps^{\frac{1}2}$ we obtain in the non-smooth case for $u \in W^{1,\infty}$. However we observe clearly the influence of non-smooth $A$ on the convergence rate when comparing to case A. For the $L^2$-convergence, the rate 
is more inconsistent, as it appears to jump from quadratic to linear when $\eps = 2^{-3}$.
Although the convergence rate in $W^{1,2}$ is only linear in this case, a visual inspection of the solutions shown
in Figure \ref{fig:comp_caseB} reveals that the solution by the diffuse domain method is still almost identical to the 
exact solution for $\eps = 2^{-5}$.

\begin{table}[ht] \centering
    \begin{tabular}{ | l | l | l |}
      \hline
      $\eps$  & $\|e^\eps\|_{L^2}$ $ (\log_2(\frac{e^k}{e^{k+1}}))$ & $\|e^\eps\|_{W^{1,2}}$ 
      $ (\log_2(\frac{e^k}{e^{k+1}}))$  \\ \hline \
      $2^{-1}$ & 0.465577        & 0.622542          \\ 
      $2^{-2}$ & 0.127344 (1.87) & 0.282197 (1.14)  \\ 
      $2^{-3}$ & 0.026850 (2.25) & 0.114941 (1.30)  \\ 
      $2^{-4}$ & 0.014065 (0.93) & 0.053956 (1.09)  \\ 
      $2^{-5}$ & 0.007958 (0.82) & 0.027376 (0.98)  \\ \hline 
          \end{tabular}
  \caption{The error $e^\eps = u - u^\eps$ for different norms in Case B.}\label{tab:epsB}
\end{table}

\begin{figure}
\centerline{%
\begin{tabular}{c@{\hspace{1pc}}c}
\includegraphics[scale=.13]{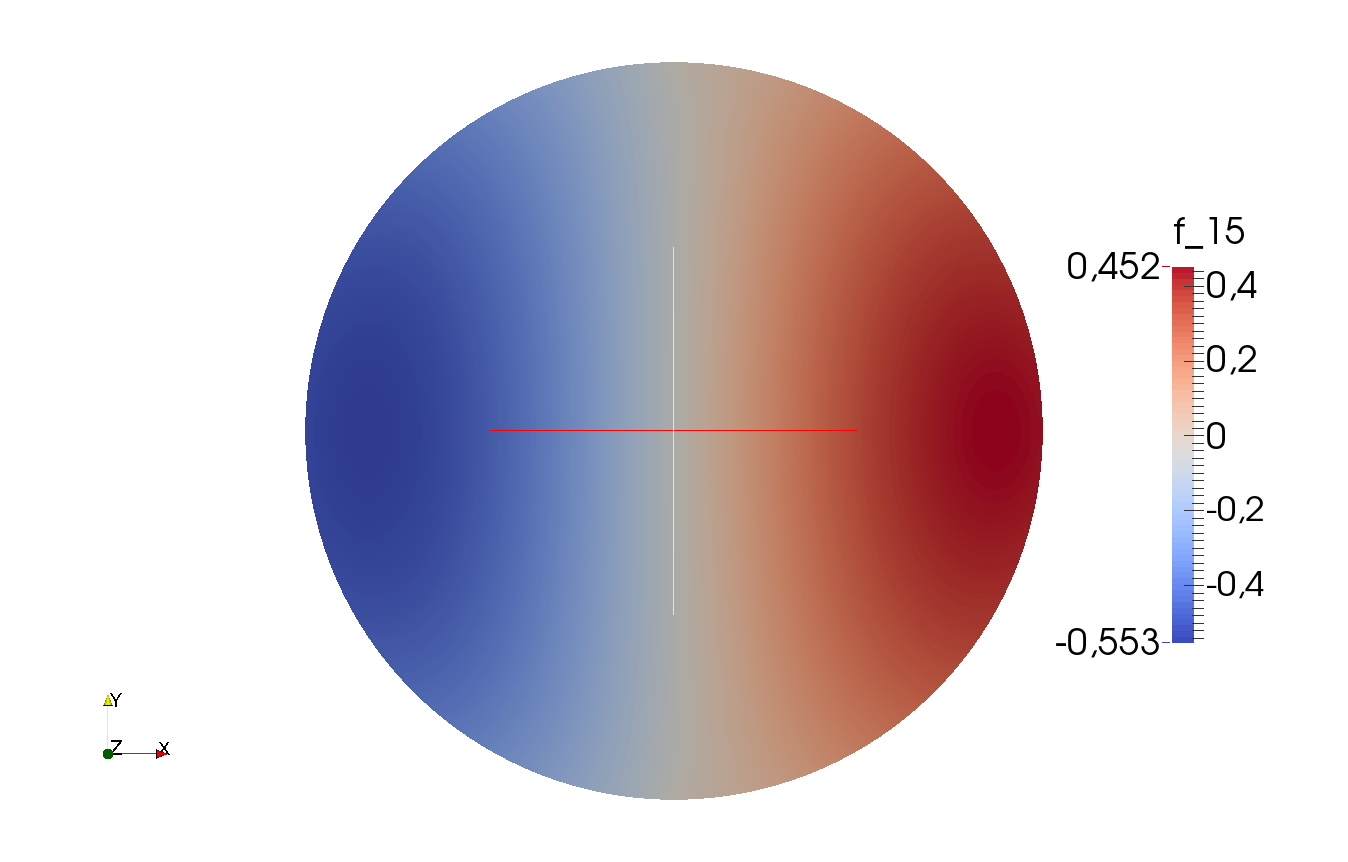} &
\includegraphics[scale=.13]{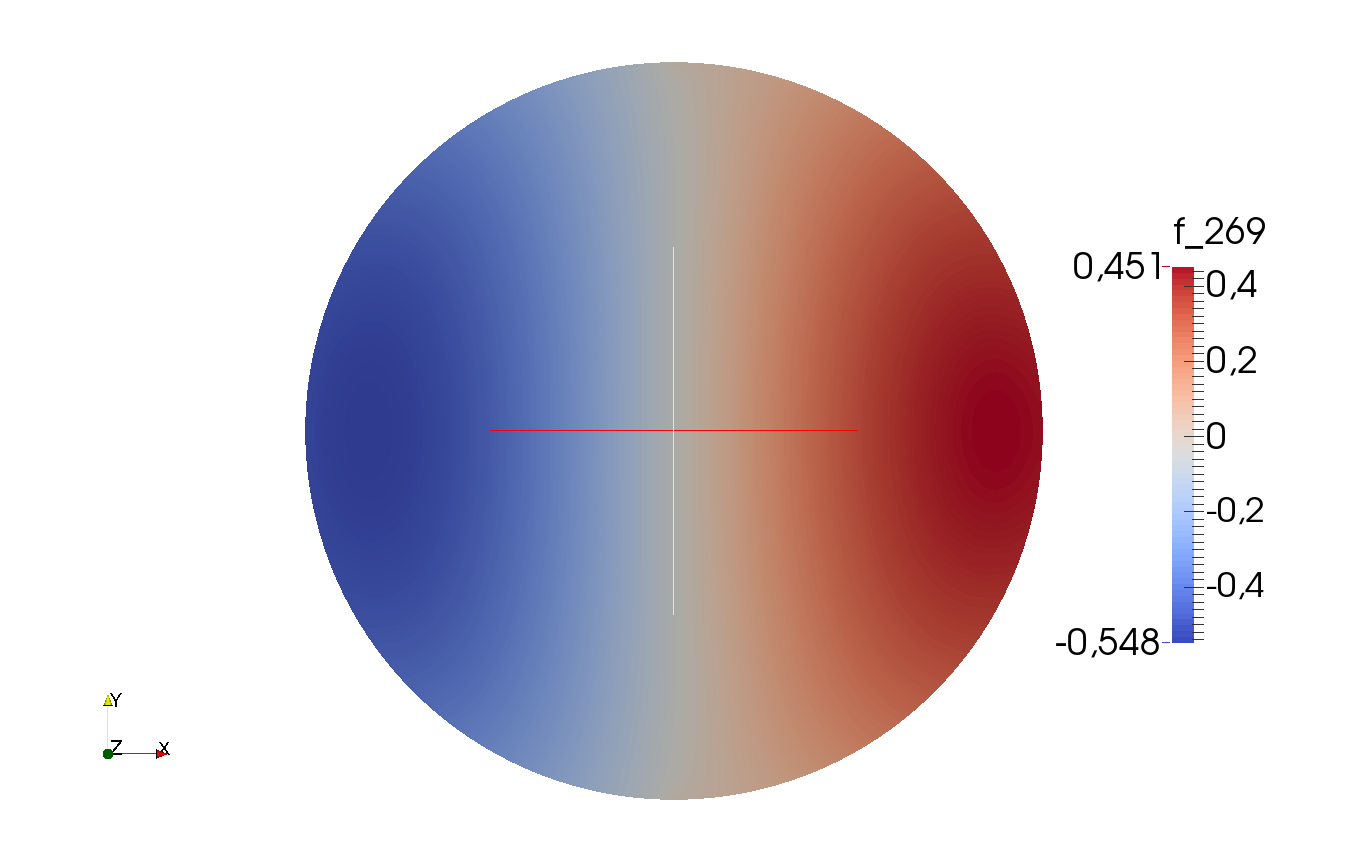} \\
(a)~~ The solution $u$ of \eqref{eq:robin_var}. & (b)~~ The solution $u^\eps$ of
\eqref{eq:robin_var_diffuse}. 
\end{tabular}}
\caption{Comparison of exact solution and diffuse domain solution in Case B. Both solution restricted to $D$.
Here $\eps = 2^{-5}$.}\label{fig:comp_caseB}
\end{figure}

\subsection{Case C: Robin BC with non-smooth parameters}\label{subsec:caseB}

Furthermore, in Case C, we also 
refine the mesh around the discontinuity to increase accuracy.

In Case A, we obtained a convergence rate in $W^{1,2}$ of order $O(\eps^{3/2})$. Here, everything were smooth.
In Case B, when working with a discontinuous matrix A, the convergence rate in $W^{1,2}$ drops to order $O(\eps)$.
If, however, the function $f$ in \eqref{eq:elliptic_pde} is in $L^2$, but not in $L^p$ for $p \gg 2$, 
Theorem \ref{thm:rate_lp_robin} yields $W^{1,2}$-convergence of, in worst case, order $O(\eps^{1/3})$.

To explore this issue numerically, we define 
\begin{align}\label{fcn:lp}
 f(x_1,x_2) = \frac{1}{|x-y|^\mu}, \text{ where } y \in \partial D \text{ is fixed.} 
\end{align}
Thus, we get that $f \in L^p(D)$ whenever $\frac{2}{\mu} > p$. All other parameters
are given as in Case A. With a similar reasoning as in Lemma~\ref{lem:rhs1}, using $n=2$ and $\alpha=1$, we expect a rate of convergence of $O(\eps^{5/6-1/p})$.

\begin{figure}
 \includegraphics[scale=.5]{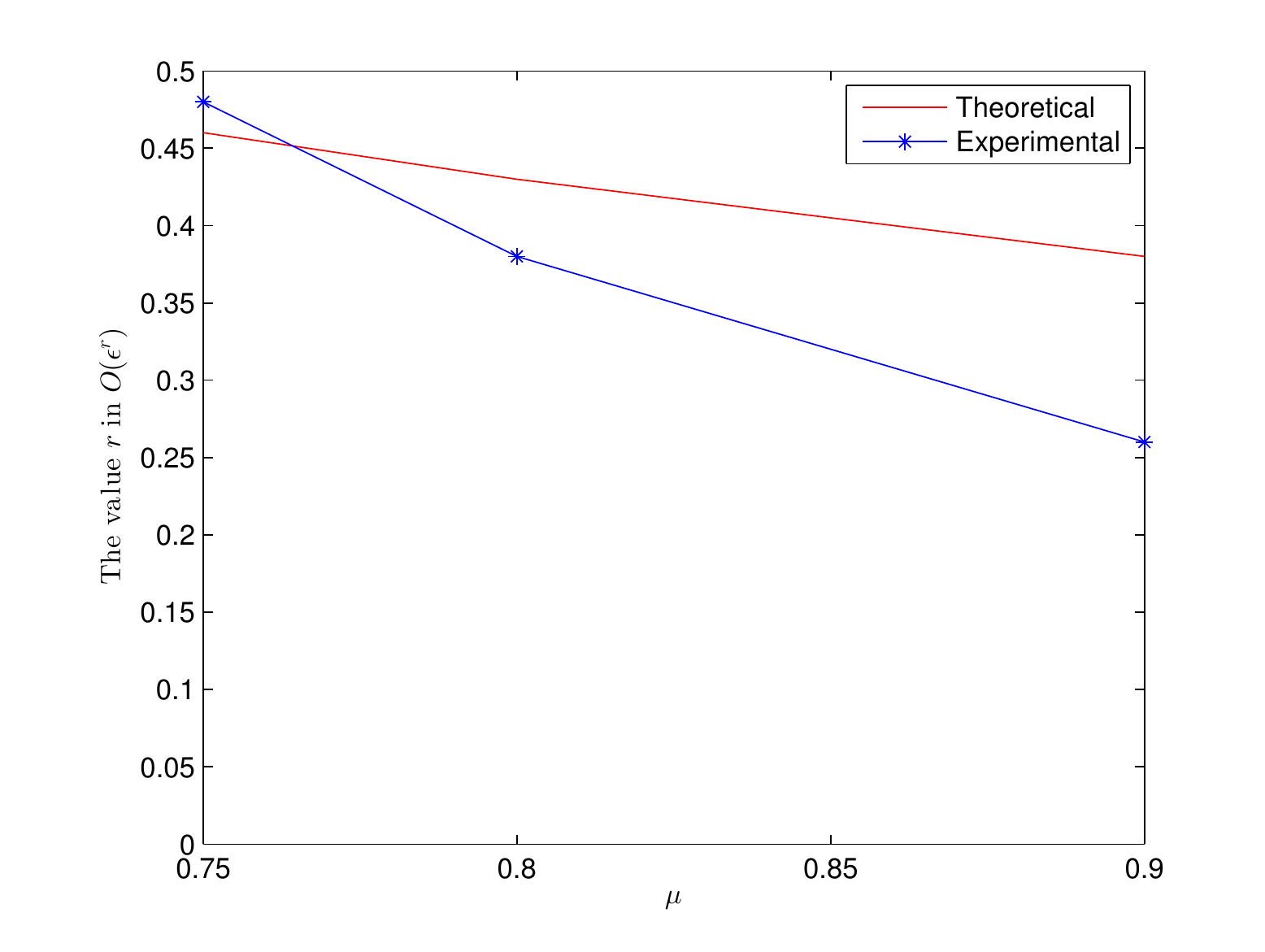}
 \caption{A plot of the $W^{1,2}$-convergence in Case C. The $x$-axis displays the parameter
 $\mu$ in \eqref{fcn:lp}, while the $y$-axis shows the theoretical and experimental convergence rate
 of corresponding function in $L^p$ . The theoretical convergence rate follows from 
 Theorem \ref{thm:rate_lp_robin}, and is of order $O(\eps^{\frac{5}{6}-\frac{1}{p}})$.}\label{fig:compare_Lp} 
\end{figure}

In Figure \ref{fig:compare_Lp}, we see the convergence rate as a function of the parameter $\mu$
in \eqref{fcn:lp}. As expected, the convergence rate becomes worse when $\mu$ increases. The 
experimental rate deteriorates more, however, than the theory suggests. We believe this to be linked to
the challenge of the numerical implementation of such a singular function. Although of practical
importance, we find that dealing with this particular implementation issue is beyond the scope of this article,
and leave it therefore to future research. 

\subsection{Case D: Dirichlet BC with smooth parameters}\label{subsec:caseD}
We will now study the diffuse domain method for a Dirichlet problem. More particularly, we will compare the solutions
$u$ and $u^\eps$ of \eqref{eq:dirichlet_var} and \eqref{eq:dirichlet_var_diffuse},
respectively. The parameters are given as in case A, with the exception
$$  \ \alpha(x_1,x_2) = \frac{1}{\eps^{\sigma}} $$
in order to realize the penalty method.
%
From Theorem \ref{thm:rate_sobolev_dirichlet_p} and Remark \ref{rem:rate_sobolev_dirichlet_p2}, 
the choice of $\beta = \eps^{-1}$ should provide a $L^2$-convergence of order $O(\eps)$, 
whereas the choice of $\beta = \eps^{-3/4}$ should yield a $W^{1,2}$-convergence of order $O(\eps^{3/4})$.

\begin{figure}
\centerline{%
\begin{tabular}{c@{\hspace{1pc}}c}
\includegraphics[scale=.5]{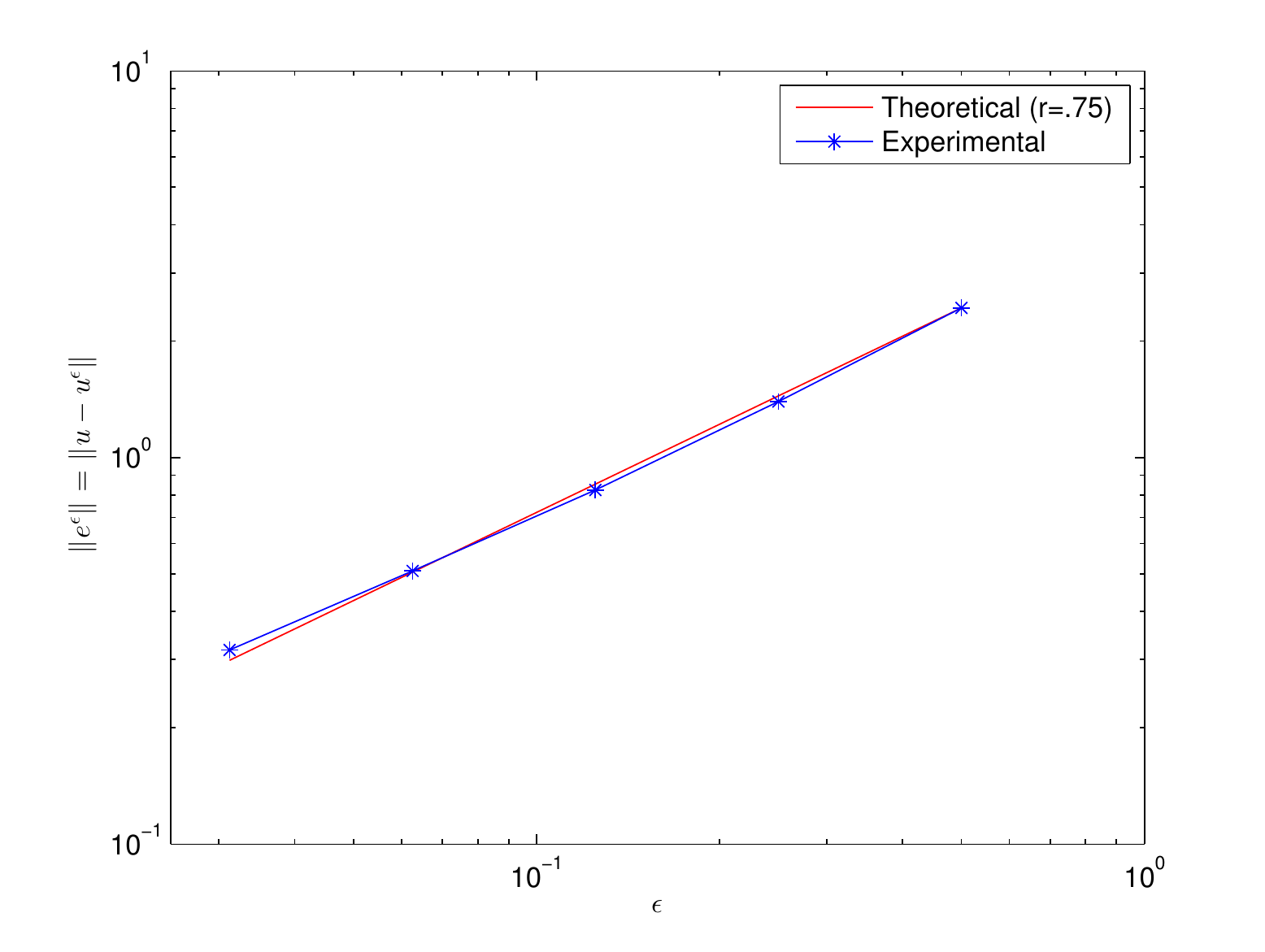} \\
(a)~~ $W^{1,2}$ convergence. $\sigma = 0.75$. \\
\includegraphics[scale=.5]{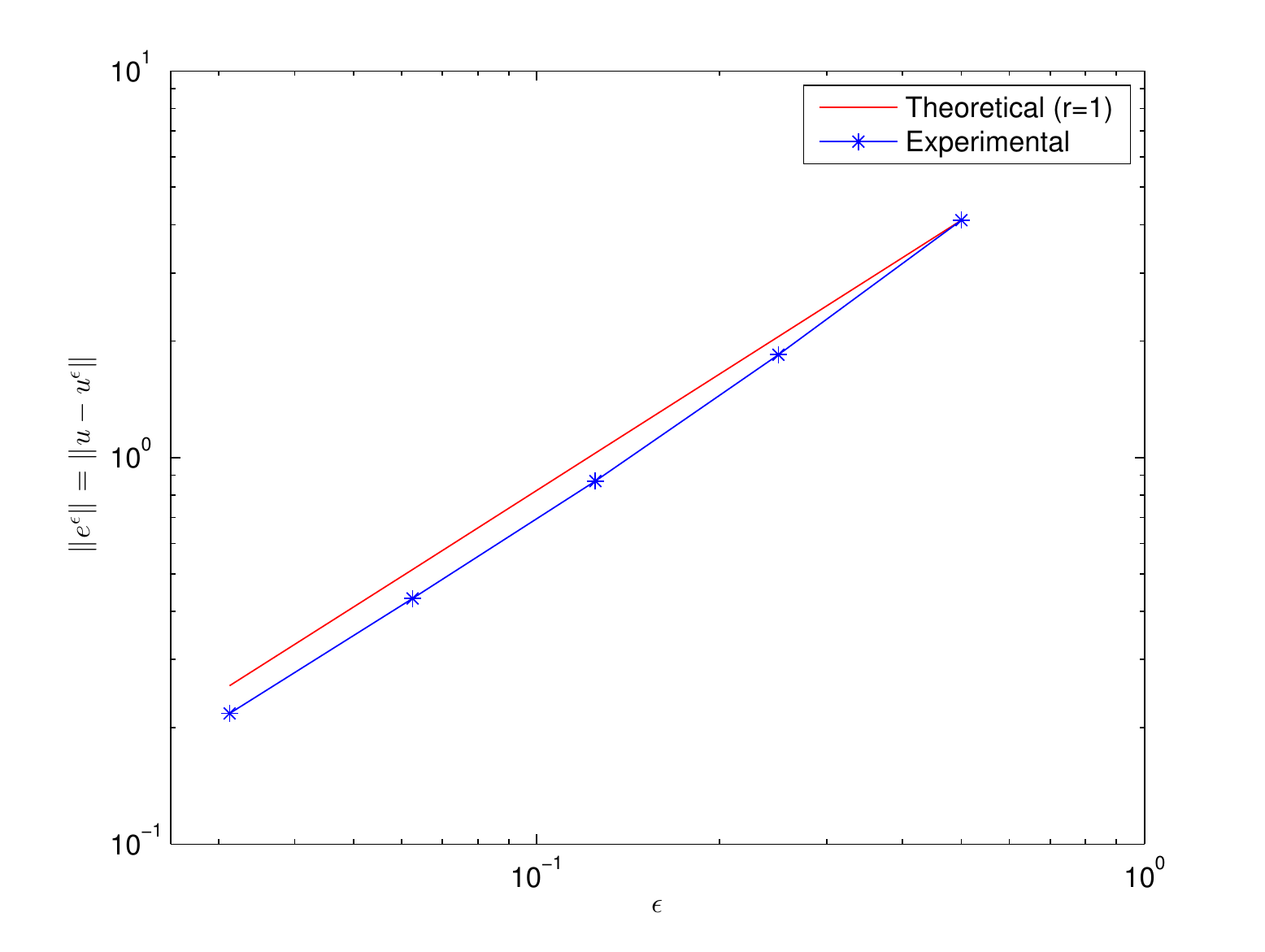} \\
(b)~~ $L^2$ convergence. $\sigma = 1.0$. 
\end{tabular}}
\caption{A log-log plot of the convergence rates in Case D. In each subplot we see the actual 
convergence rate (experimental), compared to the theoretical rate of order $O(\eps^r)$.}\label{fig:logplotD}
\end{figure}

In Table \ref{tab:epsD}(a), we see the convergence rate of the error when $\beta = \eps^{-3/4}$. As expected,
the rate is of order $O(\eps^{3/4})$ in $W^{1,2}$ norm. Furthermore, in Table \ref{tab:epsD}(b), we 
obtain the expected linear convergence in $L^2$ norm. The rates can also be seen in Figure \ref{fig:logplotD}.

In Figure \ref{fig:comp_caseD}, we see a solution of \eqref{eq:dirichlet_var} and \eqref{eq:dirichlet_var_diffuse},
respectively. Although the shape of the solution is visually similar, there is a much larger quantitative 
difference compared to Case A and B.

\begin{table}[ht]
\begin{minipage}[b]{.45\textwidth}
  \centering
    \begin{tabular}{ | l | l | l |}
      \hline
      $\eps$  & $E_{L^2}$ $ (\log_2(\frac{E^k}{E^{k+1}}))$ & $E_{W^{1,2}}$ $ (\log_2(\frac{E^k}{E^{k+1}}))$ \\ \hline \hline
      $2^{-1}$  & 4.434278          & 2.439132       \\ 
      $2^{-2}$  & 2.301163   (0.95) & 1.394967 (0.81) \\ 
      $2^{-3}$  & 1.365375   (0.75) & 0.823765 (0.76) \\ 
      $2^{-4}$  & 0.878726   (0.64) & 0.509194 (0.69) \\ 
      $2^{-5}$  & 0.566966   (0.63) & 0.318188 (0.68) \\ \hline
    \end{tabular}
    (a)~~ $\sigma = 0.75$.
 \end{minipage}\qquad
\begin{minipage}[b]{.45\textwidth}
  \centering
   \begin{tabular}{ | l | l | l |}
      \hline
      $\eps$  & $E_{L^2}$ $ (\log_2(\frac{E^k}{E^{k+1}}))$ & $E_{W^{1,2}}$ $ (\log_2(\frac{E^k}{E^{k+1}}))$ \\ \hline \hline
      $2^{-1}$  & 4.109302          & 2.294747       \\ 
      $2^{-2}$  & 1.844329   (1.16) & 1.188867 (0.95) \\ 
      $2^{-3}$  & 0.868948   (1.09) & 0.616540 (0.95) \\ 
      $2^{-4}$  & 0.432200   (1.01) & 0.331883 (0.89) \\ 
      $2^{-5}$  & 0.217938   (0.99) & 0.186374 (0.83) \\ \hline
    \end{tabular}
    (b)~~ $\sigma = 1.00$.
\end{minipage}
\caption{The error $e^\eps = u - u^\eps$ for different norms.}
\label{tab:epsD}    
\end{table}

\begin{figure}
\centerline{%
\begin{tabular}{c@{\hspace{1pc}}c}
\includegraphics[scale=.16]{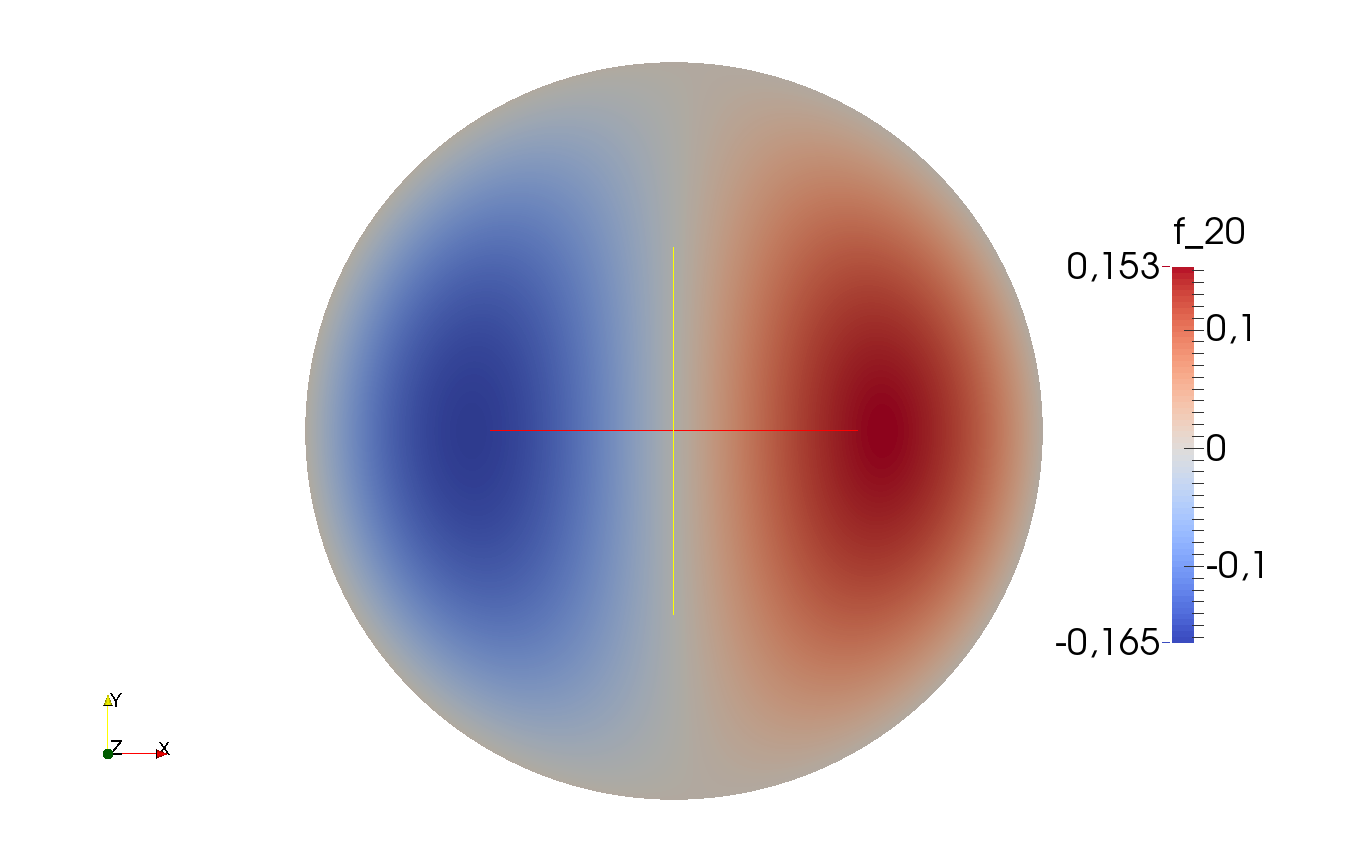} &
\includegraphics[scale=.16]{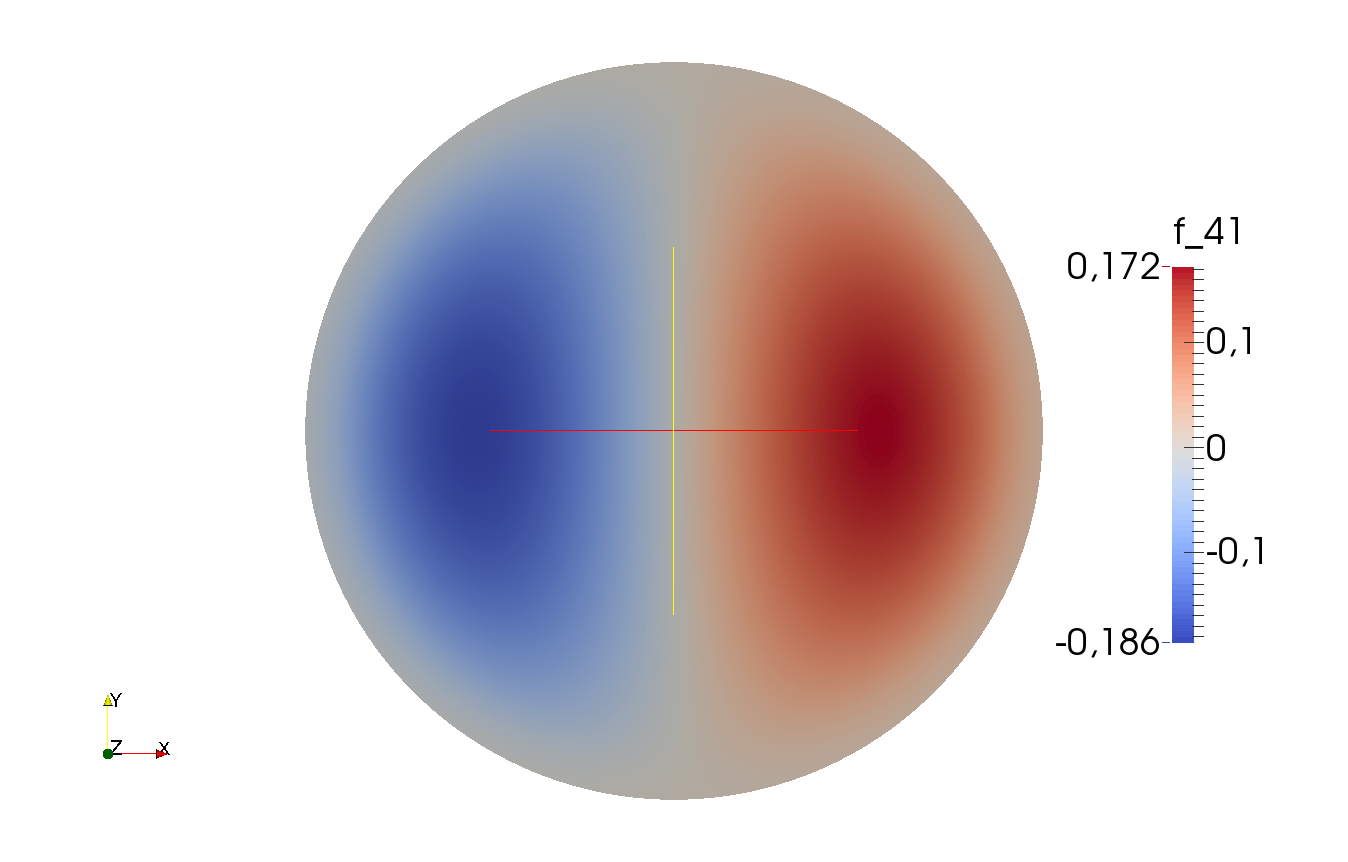} \\
(a)~~ The solution $u$ of \eqref{eq:dirichlet_var}. & (b)~~ The solution $u^\eps$ of
\eqref{eq:dirichlet_var_diffuse}. 
\end{tabular}}
\caption{Comparison of exact solution and diffuse domain solution in Case D. Both solution restricted to $D$.
Here $\eps = 2^{-5}$ and $\sigma = 1.$}\label{fig:comp_caseD}
\end{figure}

\subsection{Case E: Dirichlet BC with smooth parameters}\label{subsec:caseE}
To guarantee a quadratic convergence in $L^2$,  Theorem \ref{thm:rate_sobolev_robin_p}
 requires the domain to be $C^{1,1}$. In this final example, we work with the mesh $D = (0,1) \times (0,1)$,
which is only a Lipschitz domain. All parameters are otherwise identical to in Case A.
We see from Figure \ref{fig:logE} that the convergence rates are unchanged compared to case A despite the lower regularity of the domain. This gives some hope that our results can be extended to general Lipschitz domain or at least piecewise smooth domains, which remains an open question.

\begin{figure}
\centerline{%
\begin{tabular}{c@{\hspace{1pc}}c}
\includegraphics[scale=.35]{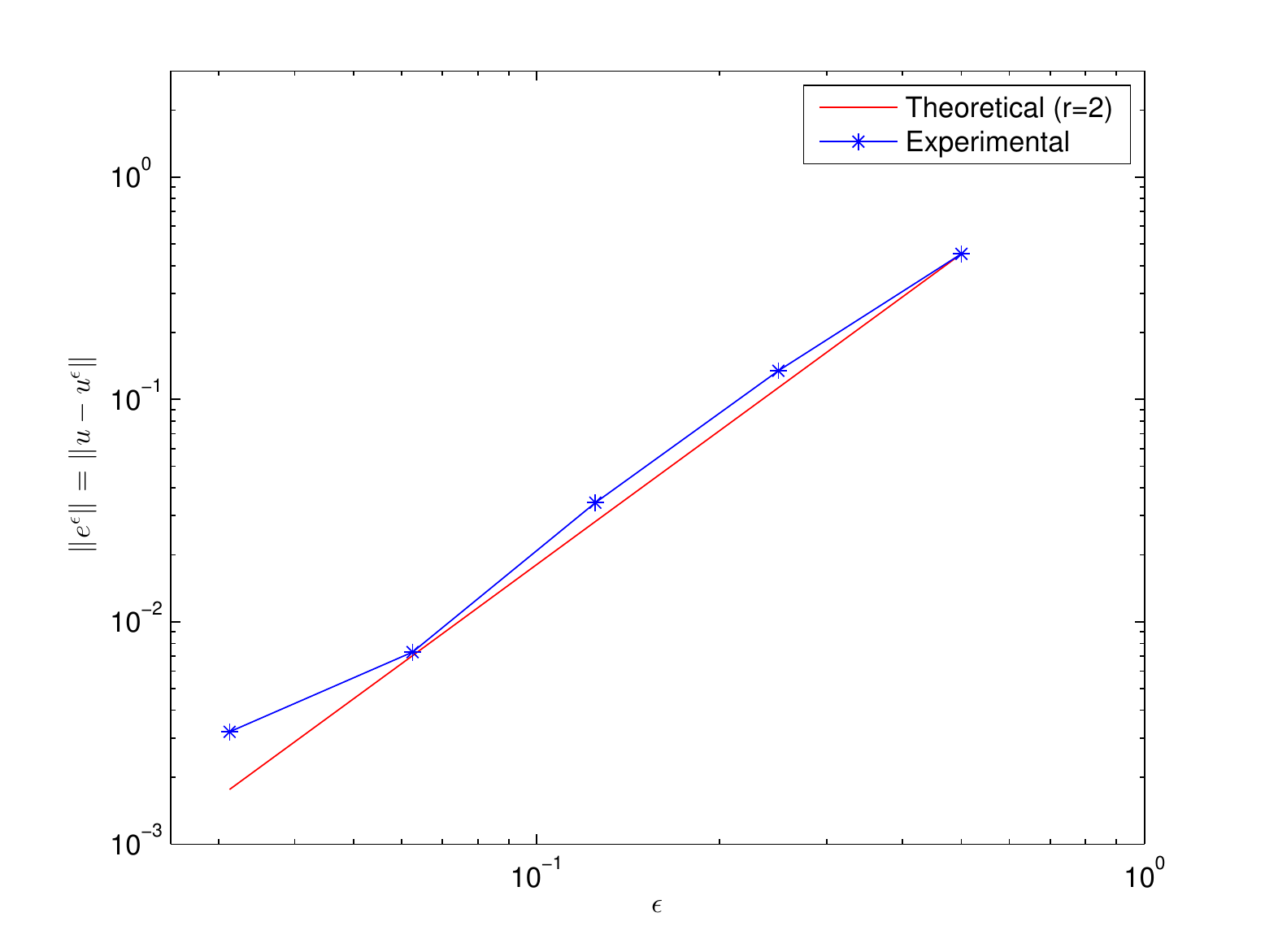} &
\includegraphics[scale=.35]{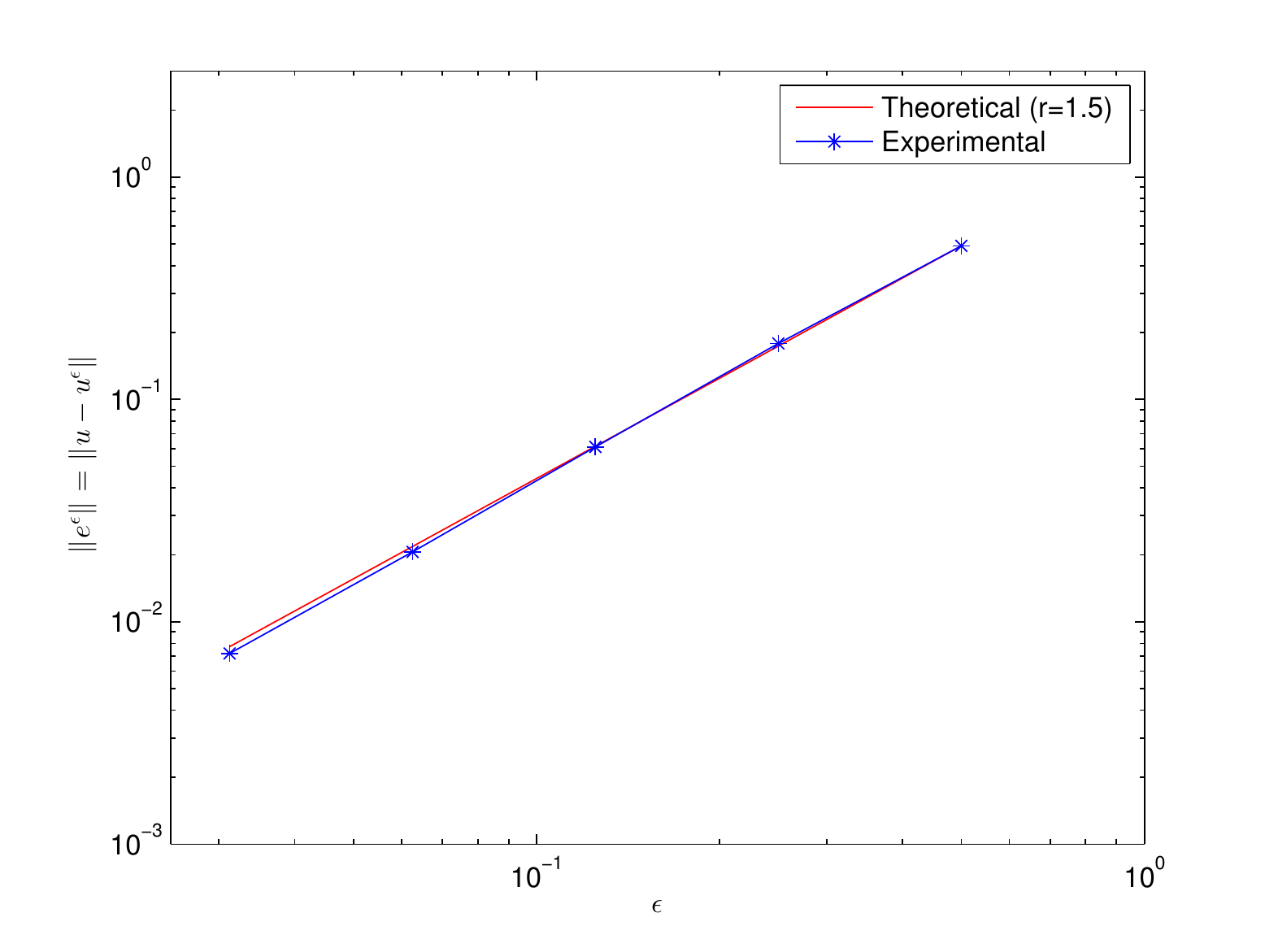} \\
(a)~~ Convergence in $L^2$-norm. & (b)~~ Convergence in $W^{1,2}$-norm. \\
\includegraphics[scale=.35]{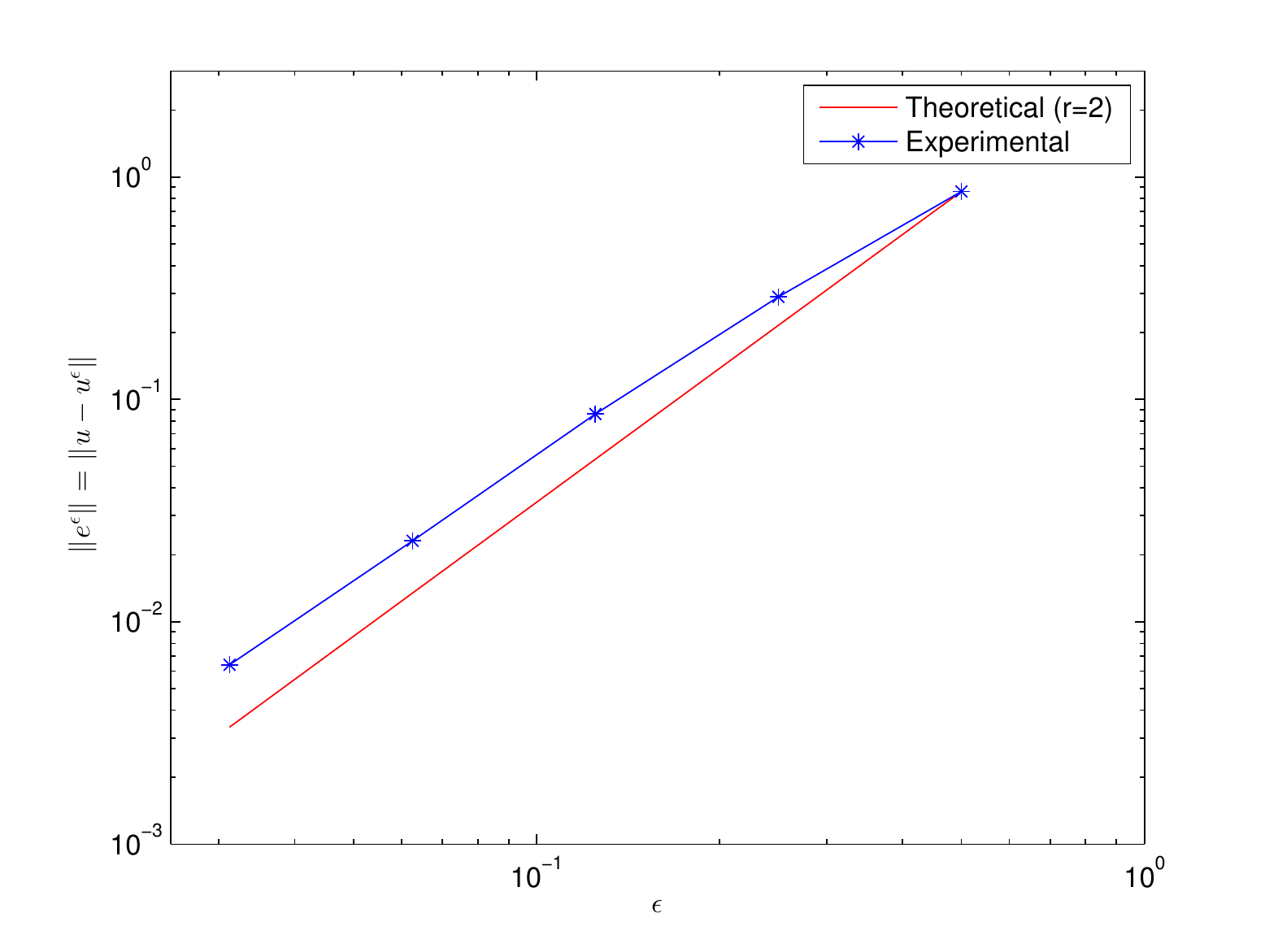} &
\includegraphics[scale=.35]{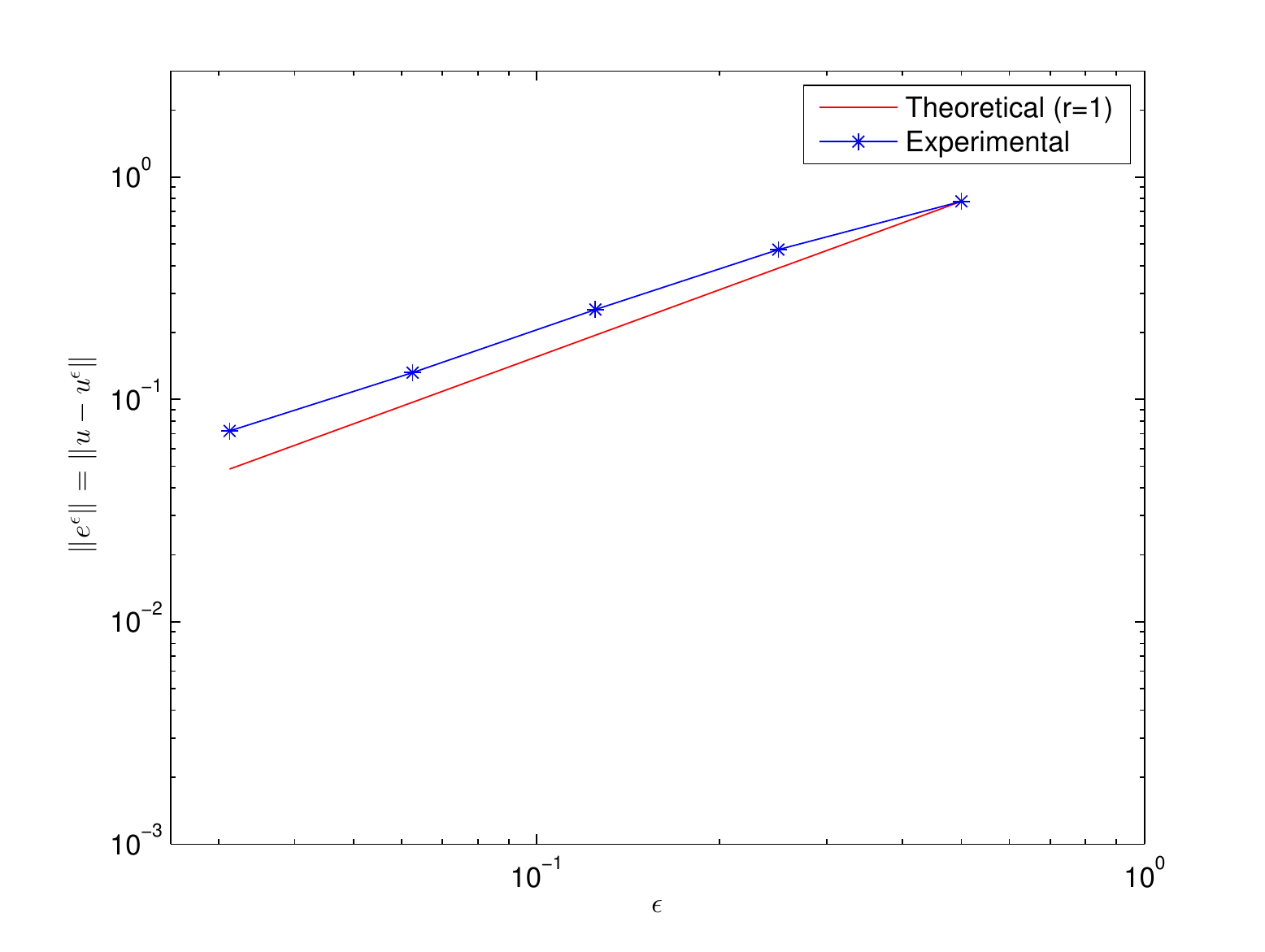} \\
(c)~~ Convergence in $W^{1,1}$-norm. & (d)~~ Convergence in $W^{1,\infty}$-norm. 
\end{tabular}}
\caption{A log-log plot of the convergence rates in Case E. In each subplot we see the actual 
convergence rate (experimental), compared to the theoretical rate of order $O(\eps^r)$.}\label{fig:logE}
\end{figure}


\section{Conclusions}\label{sec:conclusion}
In this work we presented a systematic approach for deriving diffuse domain approaches for second order elliptic problems with usual type of boundary conditions. The advantage of our method is that based on standard variational formulations it readily leads to a relaxed variational formulation, which can be implemented easily, in a straight-forward manner.
We presented a self-contained analysis of the error introduced by the diffuse domain method. Depending on the regularity of the data, we could rigorously prove convergence rates. These rates seem to be sharp as shown by numerical experiments. 
As a by-product of our analysis, we derived trace and embedding theorems as well as Poincar\'e inequalities for weighted Sobolev spaces which are stable with respect to the relaxation parameter $\eps$.
It remains open to fill a gap to transfer our quadratic convergence results to quadratic convergence results in the $L^2(D)$-norm which have been proposed in literature.
Furthermore, a thorough analysis of numerical methods for the diffuse domain method is left for future work.
We are optimistic that time-dependent problems could be treated in a similar manner, further modifications will be needed in the case of evolving surfaces.

\section*{Acknowledgements}

MB and MS acknowledge support by ERC via Grant EU FP 7 - ERC Consolidator Grant 615216 LifeInverse. MB acknowledges support by the German Science Foundation DFG via  EXC 1003 Cells in Motion Cluster of Excellence, M\"unster, Germany.
OLE acknowledges support by DAAD for his one year research stay at WWU M\"unster. 


\end{document}